\documentclass[12pt]{amsart}%

\usepackage{graphicx}

\usepackage{mathtools}
\usepackage{wasysym}
\usepackage[all]{xy}
\usepackage{relsize}

\usepackage{multicol,multirow}
\usepackage{amsmath,amssymb,amsfonts}
\usepackage{mathrsfs}
\usepackage{amsthm}
\usepackage{rotating}
\usepackage{appendix}
\usepackage[numbers]{natbib}
\usepackage{ifpdf}
\usepackage[T1]{fontenc}
\usepackage{newtxtext}
\usepackage{newtxmath}
\usepackage{textcomp}
\usepackage[a4paper,body={16.3cm,22.8cm},centering]{geometry}
\usepackage[colorlinks]{hyperref}
\usepackage{tikz}
\usepackage{orcidlink}

\begin{document}

%\documentclass[12pt]{amsart}

%%%%%%%%%%%%%%%%%%%%%%%%%%%%%%%%%%%%%%%%%%%%%
%% 1. Preamble
%%%%%%%%%%%%%%%%%%%%%%%%%%%%%%%%%%%%%%%%%%%%%

%\usepackage{amssymb, stackrel}
%\usepackage{mathtools}
%\usepackage{txfonts}
%\usepackage{eucal}

%\usepackage{extarrows}

%\usepackage{wasysym}
%\usepackage[all]{xy}
%\usepackage{relsize}

%\usepackage{graphicx}

%\usepackage{float}

%\usepackage{xspace}

%\usepackage[a4paper,body={16.3cm,22.8cm},centering]{geometry}

%\usepackage[colorlinks,final,backref=page,hyperindex,pdftex]{hyperref}
%\usepackage{shuffle}
%\usepackage{pgfplots}

%Kuing
\usetikzlibrary{calc,shapes.geometric}
\tikzset{
baseon/.style={baseline={($(#1)+(0,-0.58ex)$)}},
baseon/.default=current bounding box.center,
every picture/.style=baseon,
lst/.style={},
dst/.style={circle,inner sep=1pt,outer sep=0pt,fill,draw,dst2},
dst2/.style={fill=white},
ddst/.style={diamond,draw,inner sep=1pt},
eest/.style={ellipse,draw,inner sep=1pt,minimum size=2ex},
}

%\usetikzlibrary{external}
%\tikzexternalize
%\tikzset{external/system call={latex \tikzexternalcheckshellescape
% -halt-on-error -interaction=batchmode -jobname "\image" "\texsource";
% dvips -o "\image".eps "\image".dvi}}
%%%%%%%%%%%%%%%%%%%%%%
%% 2. Commands
%%%%%%%%%%%%%%%%%%%%%%

\font \eightrm=cmr8
\font \sevenrm=cmr7
\font \fiverm=cmr5
\newcommand{\nc}{\newcommand}
\nc\smsc{0.8}%small scale
\def\tthreeone{\XX[scale=\smsc]{\xxr{-5}5}}
\def\tthreetwo{\XX[scale=\smsc]{\xxl55}}
\def\tfourone{\XX[scale=\smsc]{\xxr{-4}4\xxr{-7.5}{7.5}}}
\def\tfourtwo{\XX[scale=\smsc]{\xxr{-5}5\xxl{-2}8}}
\def\tfourthree{\XX[scale=\smsc]{\xxl{4}4\xxl{7.5}{7.5}}}
\def\tfourfour{\XX[scale=\smsc]{\xxl{5}5\xxr{2}8}}
\def\tfourfive{\XX[scale=\smsc]{\xxr{-6}6\xxl66}}
\def\oprec{\!\!\joinrel{\ocircle\hskip -12.5pt \prec}\,}
\def\soprec{\,\joinrel{\ocircle\hskip -6.7pt \prec}\,}
\def \endsquare{$\sqcup \!\!\!\! \sqcap$ \\}
\def\diagramme #1{\vskip 4mm \centerline {#1} \vskip 4mm}
\def\att#1{\textcolor{red}{!!!}\textcolor{blue}{#1}}
\nc{\ooverline}[1]{\overline{\overline #1}}
\nc{\uunderline}[1]{\underline{\underline #1}}
\def \restr#1{\mathstrut_{ | \, \scriptstyle{ #1}}}
\def \srestr#1{\mathstrut_{\scriptstyle |}\hbox to
-1.5pt{}\raise-4pt\hbox{$\hskip 1pt\scriptscriptstyle #1$}}
\nc{\mop}[1]{\mathop{\hbox {\rm #1} }\nolimits}
\nc{\gmop}[1]{\mathop{\hbox {\bf #1} }\nolimits}
\def\starz{{\displaystyle\mathop{\star}\limits}_z}
\def\astT{{\displaystyle\mathop{\ast}\limits}_T}
\nc{\smop}[1]{\mathop{\hbox {\sevenrm #1} }\nolimits}
\nc{\ssmop}[1]{\mathop{\hbox {\fiverm #1} }\nolimits}
\nc{\mopl}[1]{\mathop{\hbox {\rm #1} }\limits}
\def\dbar{d\hskip-3pt \raise 4pt\hbox{-}}
\nc{\smopl}[1]{\mathop{\hbox {\sevenrm #1} }\limits}
\nc{\ssmopl}[1]{\mathop{\hbox {\fiverm #1} }\limits}
\renewcommand\baselinestretch{1}

\newcommand{\delete}[1]{}

%\delete{
\nc{\mlabel}[1]{\label{#1}}  % Use this to suppress names
\nc{\mcite}[1]{\cite{#1}}  % Use this to suppress names
\nc{\mref}[1]{\ref{#1}}  % Use this to suppress names
\nc{\mbibitem}[1]{\bibitem{#1}} % Use this to show number name
%}

\delete{
\nc{\mlabel}[1]{\label{#1}  % Use the next two lines to show names
{\hfill \hspace{1cm}{\small\tt{{\ }\hfill(#1)}}}}
\nc{\mcite}[1]{\cite{#1}{\small{\tt{{\ }(#1)}}}}  % Use this lines to show names
\nc{\mref}[1]{\ref{#1}{{\tt{{\ }(#1)}}}}  % Use this lines to show names
\nc{\mbibitem}[1]{\bibitem[\bf #1]{#1}} % Use this to show name
}

%%%%%%%%%%%%%%%%%%%%%%%% Statements
\newtheorem{theorem}{Theorem}[section]

\newtheorem{notation}[theorem]{Notation}%[section]
\newtheorem{corollary}[theorem]{Corollary}%[section]
\newtheorem{proposition}[theorem]{Proposition}%[section]
\newtheorem{lemma}[theorem]{Lemma}%[section]

{\theoremstyle{definition}
\newtheorem{remark}[theorem]{Remark}%[section]
\newtheorem{example}[theorem]{Example}%[section]
\newtheorem{conjecture}[theorem]{Conjecture}%[section]
\numberwithin{equation}{section}

%{\theoremstyle{definition}
\newtheorem{problem}[theorem]{Problem}%[section]
\newtheorem{definition}[theorem]{Definition}%[section]
\newtheorem{definitions}[theorem]{Definitions}%[section]
}

\renewcommand{\labelenumi}{{\rm \alph{enumi}}}
\renewcommand{\theenumi}{\alph{enumi}}

\renewcommand{\labelenumii}{{\rm \roman{enumii}}}
\renewcommand{\theenumii}{\roman{enumii}}
\newcommand\alphlist{a,b,c,d,e,f,g,h,i,j,k,l,m,n,o,p,q,r,s,t,u,v,w,x,y,z}
\newcommand\Alphlist{A,B,C,D,E,F,G,H,I,J,K,L,M,N,O,P,Q,R,S,T,U,V,W,X,Y,Z}
\newcommand\getcmds[3]{\expandafter\newcommand\csname #2#1\endcsname{#3{#1}}}
\makeatletter
%\@for\x:=\alphlist\do{\expandafter\getcmds\expandafter{\x}{cal}{\cal}}      % \cala,\calb,...
\@for\x:=\alphlist\do{\expandafter\getcmds\expandafter{\x}{frak}{\mathfrak}}% \fraka,\frakb,...
\@for\x:=\Alphlist\do{\expandafter\getcmds\expandafter{\x}{frak}{\mathfrak}}% \frakA,\frakB,...
\makeatother

\nc{\bfk}{{\bf k}}
%
%\font\cyr=wncyr10
%\font\cyrs=wncyr5
%\newfont{\scyr}{wncyr10 scaled 550}
%\nc{\sha}{\mbox{\cyr X}}
%\nc{\ssha}{\mbox{\bf \scyr X}}
\nc{\sha}{\shuffle}

\nc{\id}{\mathrm{id}}
\nc{\Id}{\mathrm{Id}}
\nc{\lbar}[1]{\overline{#1}}
\nc{\ot}{\otimes}
\nc{\dep}{\mathrm{dep}}
\nc{\ver}{\mathrm{ver}}

\nc{\tred}[1]{\textcolor{red}{#1}} \nc{\tgreen}[1]{\textcolor{green}{#1}}
\nc{\tblue}[1]{\textcolor{blue}{#1}} \nc{\tpurple}[1]{\textcolor{purple}{#1}}
\nc{\tcyan}[1]{\textcolor{cyan}{#1}} % qian lan se
\nc{\tblk}[1]{\textcolor{black}{#1}}

\nc{\li}[1]{\tpurple{\underline{Li:}#1 }}
\nc{\liadd}[1]{\tpurple{#1}}
\nc{\xing}[1]{\tblue{\underline{Xing:}#1 }}
\nc{\yuan}[1]{\tred{\underline{Yuan:}#1 }}
\nc{\markus}[1]{\tred{\underline{Markus:} #1}}
\nc{\dominique}[1]{\tpurple{\underline{Dominique: }#1 }}
\long\def\ignore#1{}

\newcommand\stset[2]{\tikzset{#1/.style={#2}}}
\newcommand\stadd[2]{\tikzset{#1/.append style={#2}}}
\newcommand\dadd[1]{\stset{dst2}{#1}}
\newcommand\treeo[2][]{\tikz[x=0.7cm,y=0.7cm,line width=0.15ex,
every node/.style={font=\scriptsize,inner sep=1pt,label distance=1pt},#1]{%
\coordinate (o) at (0,0);#2}}%
\newcommand\treeoo[2][]{\treeo[#1]{\path (o) node[dst,name=o]{};#2}}%
\newcommand\cddf[3]{%
\coordinate (#2) at ($(#1)+(#3)$);
\draw[lst] (#1)--(#2);
\node[dst] at (#1) {};}
\newcommand\cdx[4][1]{\cddf{#2}{#3}{#4:#1*0.5cm}}
\newcommand\cdl[2][1]{\cdx[#1]{#2}{#2l}{135}}%left
\newcommand\cdr[2][1]{\cdx[#1]{#2}{#2r}{45}}%right
\newcommand\cdlr[2][1]{%
\foreach \i in {#2} {\cdl[#1]{\i}\cdr[#1]{\i}}}%left right
\newcommand\cdlrx[3][1]{%
\cdx[#1]{#2}{#2l}{180-#3}\cdx[#1]{#2}{#2r}{#3}}%left right with angle
\newcommand\cda[2][1]{\cdx[#1]{#2}{#2a}{90}}%above
\newcommand\cdb[2][1]{\cdx[#1]{#2}{#2b}{-90}}%below
\def\zzz#1`#2...#3`#4...#5`#6@{%
--++(#1)%coordinate(#2)
node[dst,label={#5:$#6$},name=#2]{}
node[midway,auto,#3]{$#4$}
}
\def\ddd#1`#2`#3@{+(#1)node[ddst,name=#2]{$#3$}}
\def\eee#1`#2`#3@{+(#1)node[eest,name=#2]{$#3$}}
\def\xxx#1`#2@{node[midway,auto,inner sep=1pt,#1]{$#2$}}
\def\pp#1`#2`#3@{node[dst,label={#2:$#3$},pos=#1]{}}
\def\oo#1`#2`#3@{\path (o) node[dst,label={#2:$#3$},name=o,#1]{};}
\def\eoo#1`#2@{\node[eest,name=o,#1] at (o) {$#2$};}

%%%%%%%%%%%%this time new method2020-10-01
\newif\ifshowjdq
\showjdqtrue
%\showjdqfalse
\newcommand\setXXclip[3]{%
\def\XXheight{#1}\def\XXdepth{#2}\def\XXwidth{#3}}
\setXXclip{1}{-0.5}{1.1}
\newcommand\scopeclip[1]{\begin{scope}
\clip(-\XXwidth,\XXdepth)rectangle(\XXwidth,\XXheight);#1
\end{scope}}
\newcommand\XX[2][]{%
\tikz[x=0.5cm,y=0.5cm,baseline,inner sep=1.5pt,line width=0.15ex,
every node/.style={font=\scriptsize},#1]{
\scopeclip{\draw (2,2)--(0,0)--(-2,2) (0,-2)--(0,0);
\ifshowjdq\node[dst]at(0,0){};\fi}#2}}
\newcommand\xx[3]{%
\scopeclip{\draw(#1/10,#2/10)--+(#3*45:3);
\ifshowjdq\node[dst]at(#1/10,#2/10){};\fi}}
\newcommand\xxl[2]{\xx{#1}{#2}3}
\newcommand\xxr[2]{\xx{#1}{#2}1}
\newcommand\xxlr[2]{\xxl{#1}{#2}\xxr{#1}{#2}}
\newcommand\xxh[6]{
%\draw(#1/10,#2/10)+(#3*45:#6)arc(#3*45:#4*45:#6);
\draw(#1/10,#2/10)+(0.5*#3*45+0.5*#4*45:#6) node[above] {$#5$};}
\newcommand\xxhu[4][0.15]{\xxh{#2}{#3}13{#4}{#1}}
%for above nodes
\newcommand\nodea[3][]{\node[above=1pt,#1] at (#2) {$#3$};}
\newcommand\nax[2][]{\foreach \i in {#2} {\nodea[#1]{\i,1}{x}}}
\newcommand\naxx[3][]{\foreach \i in {#2} {\nodea[#1]{\i,1}{#3}}}
\newcommand\naxxx[2][]{\foreach \i/\j in {#2} {\nodea[#1]{\i,1}{\j}}}

%%%%%%%%%%%%%%%%%%%%%%%
\makeatletter
\newcommand\simra{\mathrel{\mathpalette\@verra\sim}}
\def\@verra#1#2{\lower.5\p@\vbox{\lineskiplimit\maxdimen \lineskip-.5\p@
\ialign{$\m@th#1\hfil##\hfil$\crcr#2\crcr\rightarrow\crcr}}}
\makeatother

% gaoxing's command
\nc{\dnx}{\Delta_n A} \nc{\dx}{\Delta A} \nc{\dgp}{{\rm deg_{P}}}
\nc{\dgt}{{\rm deg_{T}}} \nc{\dg}{{\rm deg}} \nc{\ida}{ID($A$)} \nc{\tu}{\tilde{u}} \nc{\tv}{\tilde{v}}
\nc{\nr}{\calr_n} \nc{\nz}{\calz_n} \nc{\fun}{\cala_{n,d}}
 \nc{\fbase}{\calb} \nc{\LF}{\mathrm{RF}} \nc{\FFA}{\mathrm{LF}} \nc{\irr}{\mathrm{Irr}}
 \nc{\result}{\bfk\mathrm{Irr}(S_n)}  \nc{\I}{I_{\mathrm{ID},n}^0}
 \nc{\nrs}{\calr_n^\star} \nc{\ii}{\mathrm{I}} \nc{\iii}{\mathrm{II}}
\nc{\intl}{{\rm int}}\nc{\ws}[1]{{#1}}\nc{\deleted}[1]{\delete{#1}}\nc{\plas}{placements\xspace}

%Zhangyuan's Commands
\nc{\bim}[1]{#1}  \nc{\shaop}{\sha_{\Omega}^{+}}  \nc{\shao}{\sha_{\Omega}}
\nc{\bbim}[2]{#1 #2} \nc{\bbbim}[2]{#1,\, #2} \nc{\RBF}{{\rm RBF}}
\nc{\frb}{F_{\RB}} \nc{\shaf}{\ssha_{\tiny{\Omega}}} \nc{\sham}{\diamond_{\tiny{\Omega}}}
\nc{\lf}{\lfloor} \nc{\rf}{\rfloor} \nc{\shan}{\ssha_{\lambda}}
\nc{\rlex}{{\rm {lex}}} \nc{\bb}{\Box} \nc{\ra}{\rightarrow}
%\nc{\e}{{\rm {e}}}
\nc{\DDF}{\mathrm{DD}(X,\,\Omega)}\nc{\DTF}{\mathrm{DT}(X,\,\Omega)} \nc{\DT}{\mathrm{DT}'(\Omega,\,V)}
\nc{\bra}{\mathrm{bra}} \nc{\bre}{\mathrm{bre}}
\nc{\dec}{\mathrm{dec}} \nc{\diamondw}{\diamond_{w}}
\nc{\type}{\mathrm{type}}

%\nc{\sha}{\shuffle}

\nc\caF[1]{\cal{F}_{#1}(X,\,\Omega)}
\nc\calt{\cal{T}(X,\,\Omega)} \nc\caltn{\cal{T}_n(X,\,\Omega)}
\nc\calta{\cal{T}_0(X,\,\Omega)}
\nc\caltb{\cal{T}_1(X,\,\Omega)}
\nc\caltc{\cal{T}_2(X,\,\Omega)}
\nc\caltd{\cal{T}_3(X,\,\Omega)}
\nc\caltm{\cal{T}_m(X,\,\Omega)}
\nc\caltx{\cal{T}(X)}
\nc\calf{\cal{F}(X,\,\Omega)}
\nc\fram{\frak{M}(\Omega,\, X)}
\nc\shaw{\sha^{NC}_w(\Omega,\, X)}
\nc\dw{\diamond_w} \nc\dl{\diamond_\ell}
\nc\shal{\sha^{NC}_\ell(X,\, \Omega)} \nc\shav{\sha^{NC}_w(\Omega,\, V)} \nc\shat{\sha^{NC,1}_w(\Omega,\, T^{+}(V))}
\nc{\cfo}{\cal{F}(X,\,\Omega)}
%\nc{\type}{\mathrm{type}}
%\nc{\cto}{\cal{T}(X,\,\Omega)}
%\nc{\dd}{\mathrm{d}}       \nc{\bb}{\mathrm{b}}
%\nc{\ov}{\overrightarrow}  \nc{\bd}{\mathbin{\bar\diamond}} \nc{\ba}{/\,\backslash}
\nc{\sh}{\rm{Sh}}
\nc{\lar}{\varinjlim}
\nc\XO{(X,\,\Omega)}
\def\cxo#1#2;{\cal{#1}#2\XO}
\nc\lrf[2]{B_{#2}^+(#1)}
\nc{\fd}{\mathrm{\text{typed angularly decorated planar rooted trees}}}
\nc{\rb}{\mathrm{RBFWs}} \nc{\dfw}{\mathrm{DFW{(X)}}} \nc{\tfw}{\mathrm{TFW{(X)}}}
\nc{\tfv}{\mathrm{TFW{(V)}}}

\def\Ve#1,#2,#3;{\vee_{#1,\,(#2,\,#3)}}
\def\bigv#1;#2;#3;{\bigvee\nolimits_{#1}^{#2;\,#3}}
\nc\rjt[2]{\mathrel{\mathop{\longrightarrow}\limits^{#1\hfill}_{\hfill#2}}}
\nc{\pl}{\cal{PLF}}
\nc{\tr}{\cal{RTF}}
\nc{\im}{\mathrm{Im} \, }
\nc{\ff}{\cal{F}_\Omega}
\nc{\tm}{T_\Omega}
\nc{\calp}{\cal{P}}
\makeatletter
\nc\dd{\@ifnextchar'{\ddA}{\ddB}}
\def\ddA'#1;{\rhd'_{#1\,}}
\def\ddB#1;{\rhd_{#1\,}}
\nc{\pbt}{\mathrm{PBT}}
\nc{\ad}{\mathrm{ad}}

%%%%%%%%%%%%%%%%%%%%%%%%
%% 20/9-23
\nc{\hRR}{\hat {\mathbb R}}
\nc{\RR}{{\mathbb R}}
\nc{\NN}{{\mathbb N}}
\nc{\ol}[1]{\overline{#1}}
\nc{\set}{{\rm \bf set}}
\nc{\setx}{{\rm \bf set^{\times}}}
\nc{\setn}{{\rm \bf set_{{\mathbb N}}}}
\nc{\Sb}{{\mathbb S}}
\nc{\cp}{\Vdash}
\nc{\te}{\otimes}
\nc{\sus}{\subseteq}
\nc{\BF}{{\rm BF}\,}
\nc{\Pio}{\Pi^\circ}
\nc{\SM}{{\rm SM}\,}
\nc{\SMat}{{\rm SMat}\,}
\nc{\MMat}{{\rm MMat}\,}
\nc{\ti}{\times}
\nc{\ben}{{\bf 1}}
\nc{\bigsum}{{\mathlarger{\sum}}}
\nc{\nil}{{\mathbf 0}}
\nc{\Hom}{\text{Hom}}
\nc{\cT}{{\mathcal T}}
%%%%%%%%%%%%%%%%%%%%%%%%%
%% Gunnar, added 10.02.24
\newcommand{\btl}{\blacktriangleleft}
\newcommand{\wtl}{\lhd}
\newcommand{\pil}{\rightarrow}
\newcommand{\hN}{\hat{\mathbb N}}
\newcommand{\note}{\noindent {\bf Note.}}
\newcommand{\ppr}{{\prime \prime}}
\newcommand{\barA}{\overline{A}}
\newcommand{\barB}{\overline{B}}
\newcommand{\App}{A^{\prime \prime}}
\newcommand{\cC}{{\mathcal C}}
\newcommand{\cB}{{\mathcal B}}
\newcommand{\PPI}{{PP(I)}}
\newcommand{\DB}{{\Delta_{\cB}}}

\newcommand{\mto}[1]{\stackrel{#1}\longrightarrow}
\newcommand{\ZZ}{\mathbb{Z}}
\newcommand{\QQ}{\mathbb{Q}}
\newcommand{\kk}{{\Bbbk}}
\newcommand{\bH}{{\mathbf H}}

%%%%%%%%
%% Added 240705

\definecolor{tol1}{HTML}{332288}
\definecolor{tol2}{HTML}{AA4499}
\definecolor{tol3}{HTML}{DDCC77}
\definecolor{tol4}{HTML}{44AA99}
\definecolor{tol5}{HTML}{88CCEE}
\definecolor{tol6}{HTML}{CC6677}
\definecolor{tol7}{HTML}{117733}
\definecolor{tol8}{HTML}{882255}

%%%%%%%%%%%%%%%%%%%%%%%
%% Added May-July 2024

\newcommand{\Pre}{{\mathrm {Pre}}}
\renewcommand{\Top}{{\mathrm {Top}}}
\newcommand{\cU}{{\mathcal U}}
\renewcommand{\cT}{{\mathcal T}}
\newcommand{\cS}{{\mathcal S}}
\newcommand{\low}{{\mathrm {low}}}
\newcommand{\bfx}{{\mathbf x}}
\newcommand{\bfu}{{\mathbf u}}

%\end{document}

%%%%%%%%%%%%%
\newcommand{\rleftarrows}{\mathrel{\raise.75ex\hbox{\oalign{%
  $\hfil\scriptstyle\relbar$\cr
  \vrule width0pt height.5ex$\scriptstyle\smash\leftarrow$\cr}}}}
\newcommand{\rightlarrows}{\mathrel{\raise.75ex\hbox{\oalign{%
  $\scriptstyle\rightarrow$\hfil\cr
  $\scriptstyle\vrule width0pt height.5ex\relbar$\cr}}}}
\newcommand{\Rrelbar}{\mathrel{\raise.75ex\hbox{\oalign{%
  $\scriptstyle\relbar$\cr
  \vrule width0pt height.5ex$\scriptstyle\relbar$}}}}
\newcommand{\longrightleftarrows}{\rleftarrows\joinrel\Rrelbar\joinrel\rightlarrows}
%%%%%%%%%%%%%
\newcommand{\bihom}[2]{\, \overset{#1}{\underset{#2}{\longrightleftarrows}}\, }
\newcommand{\pow}{{\mathcal P}}
\newcommand{\ppow}{{\mathcal {PP}}}
\newcommand{\egp}{{\Pi}}
\newcommand{\pre}{{\mathrm {pre}}}
\renewcommand{\top}{{\mathrm {top}}}
\newcommand{\preo}{{\mathrm {preo}}}
\newcommand{\topl}{{\mathrm {topl}}}
\newcommand{\cpre}{\mathrm{cpre}}
\newcommand{\ctop}{\mathrm{ctop}}

\newcommand{\con}{{\mathrm {con}}}
\newcommand{\op}{{\mathrm {op}}}
\newcommand{\ua}[1]{\uparrow \!\! #1}
\newcommand{\da}[1]{\downarrow \! #1}

\newcommand{\uai}[2]{\uparrow_{\! #1} \! #2}
\newcommand{\dai}[2]{\downarrow_{\! #1} \! #2}

\newcommand{\Pcon}{\Pre}
\newcommand{\alin}{\mathrm{Alin}}
\newcommand{\alinb}{\mathrm{AlinBu}}

%%%%%%%%%%%%%%%%%
%% For chapter 11,12
\newcommand{\vect}{\text{\bf vec}}
\newcommand{\minPre}{\text{minPre}}
\newcommand{\EGP}{{\rm{EGP}}}
\newcommand{\PRE}{{\rm{PRE}}}
\newcommand{\MOD}{{\rm{MOD}}}
\newcommand{\oU}{\overline{U}}

%%%%%%%%%%%%%%%%%%%%
%%
\newcommand{\fone}{{\mathbf 1}}
\newcommand{\one}{1}

%%%%%%%%%%%%%%%%%%%%%%%%%%%%%%%%%%%%%%%%%%%%%%
%% 3. Figures
%%%%%%%%%%%%%%%%%%%%%%%%%%%%%%%%%%%%%%%%%%%%%%

%%%%%%%%%%%%%%%%%%%%%%%%%%%%%%%%%%%%%%%%%%%%%%%
%% 4. Beginning document
%%%%%%%%%%%%%%%%%%%%%%%%%%%%%%%%%%%%%%%%%%%%%

\makeatother
%\begin{document}

%\title[Submodular functions, permutahedra, preorders, and cointeractions]
%{Submodular functions,
%  generalized permutahedra, conforming preorders, and cointeracting
%  bialgebras}
%\thispagestyle{empty}

%%%
%\author{Mohamed Ayadi}
%\address{Laboratoire de Math\'ematiques Blaise Pascal,
%CNRS--Universit\'e Clermont-Auvergne,
%3 place Vasar\'ely, CS 60026,
%F63178 Aubi\`ere, France}
%\email{mohamedayadi763763@gmail.com}
%%%

\author{Gunnar Fl\o ystad \, \orcidlink{0000-0001-9796-7530} }
\address{Matematisk institutt,
Universitetet i Bergen, Realfagbygget,
All\'egaten 41
Bergen, Norway}
\email{Gunnar.Floystad@uib.no}
%\orcidlink{0000-0001-9796-7530}
%
\author{Dominique Manchon \, \orcidlink{0000-0001-6865-7162}}
\address{Laboratoire de Math\'ematiques Blaise Pascal,
CNRS--Universit\'e Clermont-Auvergne,
3 place Vasar\'ely, CS 60026,
F63178 Aubi\`ere, France}
\email{Dominique.Manchon@uca.fr}
%\orcidlink{0000-0001-6865-7162}

\title[Submodular functions, permutahedra, preorders, and cointeractions]
{Submodular functions,
  generalized permutahedra, conforming preorders, and cointeracting
  bialgebras}

\tikzset{stdNode/.style={rounded corners, draw, align=right},
greenRed/.style={stdNode, top color=green, bottom color=red},
blueRed/.style={stdNode, top color=blue, bottom color=red}
		}
	%%%%%
	%%%%%
	
%\maketitle

\begin{abstract}Submodular functions $z$ defined on the power set of a finite set
are in bijection with generalized permutahedra $\egp(z)$.
To any such $z$ we define a class of preorders,
{\it conforming} preorders. We show the faces of $\egp(z)$
and the conforming preorders are
in bijection.
We investigate in detail this interplay between submodular functions
and generalized permutahedra on one side, and conforming preorders on
the other side, with many examples.
In particular, the face poset structure of
$\egp(z)$ correspond to two order relations
$\lhd$ and $\btl$ on preorders,
and we investigate their properties.

Ardila and Aguiar \cite{AA2017}
introduced a Hopf monoid of submodular functions/generalized permutahedra.
We show there is a bimonoid of modular functions cointeracting in a
non-standard way. By recent
theory of L.Foissy \cite{Fo2022}, on double bialgebras we get a canonical
polynomial associated to any submodular function.% $\bihom{2}{4}$ and $\pow$
\end{abstract}
         
\keywords{Finite topological spaces, Preorders, Braid fan, Conforming
          preorders, Submodular functions,
          Generalized permutahedra, Bimonoids, Bialgebras, Double bialgebras, Ehrhart polynomials }
        \subjclass[2020]{16T30, 52B05, 16T30, 06A11}
       % \keywords[MSC Codes]{[Primary]{16T30, 52B05};
       %   [Secondary]{06A11, 16T05}}       
\maketitle

\setcounter{tocdepth}{1}

\tableofcontents
	
	%%%%

%\input{preamble}
%\end{document}

%%%%%%%%%%%%%%%%%%%%%%%%%%%%%%%%%%%%%%%%%
%% 5. Body of article
%%%%%%%%%%%%%%%%%%%%%%%%%%%%%%%%%%%%%%%%%

\vspace*{14pt}
\section{Introduction}

Generalized permutahedra is a central class of polyhedra with rich
properties and surprisingly many connections to various
areas of mathematics. They are equivalent to submodular functions, an
important class in optimization and economics. 

A. Postnikov's article \cite{Pos2009} with simple and systematic
presentation, example classes, and enumerative focus,
has had great impact. But there is a long history before that.
Edmonds \cite{Edm70} introduces polymatroids which is a wide class
of generalized permutahedra, generalizing matroids and matroid polytopes.
S. Fujishige's book \cite{Fuj2005}, whose first edition appeared in 1990,
develops in Section 3 all
main results on the interplay of submodular functions and generalized
permutahedra.  Its aim is the application in optimization.
In recent years F. Ardila and M. Aguiar \cite{AA2017} endow generalized
permutahedra or equivalently submodular functions with the structure of a
Hopf algebra and develop cancellation free formulas for the antipode
of classes of Hopf algebras. The braid fan of permutahedra
has been studied in \cite{PRW2009} and \cite{MW2009}.
Its maximal cones are in fact the Weyl
chambers associated to the general linear group. 

\medskip
This article has four central actors:
\begin{itemize}
\item Finite preorders/topologies,
\item Submodular functions,
\item (Extended) generalized permutahedra,
\item Hopf- and bialgebras.
\end{itemize}
We undertake a systematic study of their interplay.
We recall and discuss central insights from the literature and develop
further the understanding of these. In particular we focus on the connection
between finite preorders/topologies on one side, and submodular functions and
extended generalized permutahedra on the other side.
We use this to construct new bialgebras and
comodule bialgebras of submodular and modular functions. By recent
theory of L.Foissy \cite{Fo2022}, to such a cointeraction
of bialgebras is associated a unique morphism
to the polynomial double bialgebra ${\mathbb Q}[x]$.
So to each submodular function we get an associated polynomial, which we describe together with its enumerative properties.

\medskip
By the Alexandrov correspondence, finite preorders and finite topologies
are equivalent objects. (The review article \cite{FlSur}
shows their ubiquity.)
The interplay between generalized permutahedra and preorders
has been developed in various sources.
In \cite{Fuj2005} this interplay is given in terms
of distributive lattices rather than preorders. The braid fan is
studied in \cite{PRW2009} and
\cite{MW2009}. The normal fan of a generalized permutahedron
is a coarsening of this fan. {\it Braid cones} of the braid fan in
${\mathbb R}^I$ corresponds precisely to preorders on the set $I$.

Our starting point here is submodular functions. Let $I$ be a finite set and
$\pow(I)$ the power set of all subsets of $I$. Further denote $\hRR =
\RR \cup \{\infty \}$. A function
\[ z : \pow(I) \pil \hRR, \quad z(\emptyset) = 0, \, z(I) < \infty \]
is {\it submodular} if
\begin{equation} \label{eq:intro-submod}
  z(A \cup B) + z(A \cap B) \leq z(A) + z(B), \quad \text{ for all }
  A,B \sus I.
  \end{equation}
To a submodular function we algebraically associate a class of
{\it conforming} preorders/topologies.
%In \cite{MW2009} submodular functions
%relate to semi-graphoids on $I$. These are an abstraction of
%Conditional Independence Statements for random variables. And the
%associated conforming preorders may be seen as ways of encoding such
%independence properties of the submodular function $z$.
The {\it generalized permutahedron} $\egp(z)$ associated to $z$ is defined by the
inequalities in $\RR^I$:
\[ \sum_{a \in A}x_a \leq z(A) \quad \text{ for each } A \sus I, \]
{together with the equality $\sum_{a\in I} x_a=z(I)$}.
We give a natural Galois correspondence, Section \ref{sec:bij}, between
families of subsets of $I$ and faces of $\egp(z)$ which induces
a bijection
\[ \text{ finite topologies conforming to } z \, \, 
  \overset{1-1}{\leftrightarrow} \, \, \text{ faces of } \egp(z). \]
This is a full generalization to any extended generalized permutahedron
of the notion of ${\mathcal B}$-forests for nestohedra.

The face poset structure of $\egp(z)$ can then be understood in
terms of an order relation on conforming preorders.
This order relation is well-known, but presented in various distinct
ways in the literature. Its earliest instance is likely in \cite{Sta1986},
concerning the face structure of the order polytope.
In \cite[Section 3]{PRW2009} it occurs as {\it contraction} of preorders, and
in \cite{FFM2017} as {\it admissible} topologies. In Section \ref{sec:poset} we
systematically study this order relation by a natural Galois correspondence
and unify and develop the various descriptions in the literature. In
fact there are two equivalent order relations on preorders, which
we denote by $\lhd$ and $\btl$, where $\lhd$ is found in \cite{Sta1986,
FFM2017}
and $\btl$ in \cite{PRW2009}.

To a submodular function $z$ and conforming preorder $P$, we associate
in Sections \ref{sec:bij} and \ref{sec:mod} two new submodular functions:
\begin{itemize}
\item $z_P$: This is the submodular function associated to the
  face $F$ of $\egp(z)$ corresponding to $P$. 
\item $z^P$: The submodular function of the cone of
  $\egp(z)$
  at the face $F$ (the cone is of full dimension equal to that of $\egp(z)$),
  with the affine space of $F$ as the lineality space of the cone.
\end{itemize}
A submodular function is {\it modular} if the inequalities
\eqref{eq:intro-submod} are all equalities.
By \cite{AA2017} the submodular functions have a Hopf algebra structure.
By using $z_P$ and $z^P$
we construct a bialgebra of modular functions, and a comodule structure
on the submodular Hopf algebra over the modular bialgebra, Section
\ref{sec:bim}. We also give a morphism of comodules from the submodular
functions to the modular functions. Using the theory of L. Foissy
\cite{Fo2022} we then associate to any submodular function
a distinguished polynomial in
${\mathbb Q}[x]$. It naturally drops out that this polynomial
counts the number of lattice points in
the {\it interior} of the maximal cones of the normal fan of
the extended generalized permutahedron $\egp(z)$.

\begin{example}
  For the $(n-1)$-dimensional permutahedron whose vertices correspond
  to permutations of $n$ elements, this polynomial is
  \[ n! \cdot \binom{k}{n} = k(k-1)(k-2) \cdots (k-n+1). \]
  For the submodular function associated to a total
  order on $n$ elements, this polynomial is $\binom{k}{n}$.
 \end{example}
 When $z$ is a {\it finite} submodular function,
 this polynomial occurs in \cite[Section 17]{AA2017} as the {\it basic invariant
   polynomial}. For matroids this becomes the Billera-Jia-Reiner polynomial.

 \medskip
 
 There are {\it three milestone sections} of this article:
 
 \begin{itemize}
 \item Section \ref{sec:poset} on the subdivision and contraction
   relations for preorders,
 \item Section \ref{sec:bij} on the bijection of faces of
   the extended generalized permutahedron $\egp(z)$ and conforming
   preorders for $z$, with examples in Section \ref{sec:exa},
 \item Section \ref{sec:bim} on the bialgebras of submodular and
   modular functions and their cointeractions.
 \end{itemize}

 \noindent{\bf Matroids:} For those especially interested in this, matroids
 are considered in Subsections \ref{subsec:pre-mat}, \ref{subsec:exa-mat},
and \ref{subsec:bimon-mat}.

 Otherwise the organization and focus is as follows:

\medskip
 {\noindent \bf Sumodular functions:} Sections \ref{sec:pre}, \ref{sec:conform},
 \ref{sec:mod}, \ref{sec:ext}.
 Section \ref{sec:pre}
 introduces submodular functions. 
 Section \ref{sec:conform}
 defines compatible and conforming preorders of a submodular function.
 Section \ref{sec:mod} introduces modular functions $z$ and describe their
 generalized permutahedra $\egp(z)$.
 Section \ref{sec:ext} shows a surprising result on uniqueness of
 extensions of conforming preorders of a submodular function.

 \medskip
 {\noindent \bf Preorders:} Sections \ref{sec:pre}, \ref{sec:braid},
 \ref{sec:poset}, \ref{sec:conform}, \ref{sec:ext}.
 Section \ref{sec:pre} defines preorders and finite topologies. It
 gives a first look at how submodular functions naturally induce preorders.
 Section \ref{sec:braid} defines the braid fan and shows how braid
 cones correspond to preorders.
 Section \ref{sec:poset} introduces the order relations $\lhd$ and $\btl$
 on preorders and develops their properties.
 Section \ref{sec:conform}
 defines compatible and conforming preorders of submodular functions.

 \medskip
 {\noindent \bf Generalized permutahedra:} Sections \ref{sec:egp},
 \ref{sec:egpclass}, \ref{sec:bij}, \ref{sec:exa}.
 Section \ref{sec:egp} defines extended generalized permutahedra and
 give basic results on dimensions and decompositions.
 Section \ref{sec:egpclass} gives the example classes of Minkowski sums
 and nestohedra. Section \ref{sec:bij} gives the fundamental bijection between
 faces of extended generalized permutahedra and conforming preorders.
 Section \ref{sec:exa} informs on the conforming preorders for example
 classes of generalized permutahedra.

 \medskip
 {\noindent \bf Hopf- and bialgebras:} Sections \ref{sec:bim}, \ref{sec:pol}.
 Section \ref{sec:bim} defines the bialgebras of modular functions and
 the comodule bialgebra of submodular functions. We show how
 these cointeract. It also defines a map
 from submodular functions to sums of modular functions.
 Section \ref{sec:pol} recalls from \cite{Fo2022}
 the definitions of double bialgebras
 $(B,\cdot, \Delta, \delta)$, and the unique morphism $B \pil
 {\mathbb Q}[x]$ associated to it. This gives a distinguished polynomial
 associated to each submodular function, which we describe together with its
 enumerative properties.

 \medskip
 {\noindent \it Acknokwledgement}: The authors thank the Trond Mohn Foundation
   funded project {\it Pure Mathematics in Norway} for travel support for
   research cooperation.

%%%%%%%%%%%%%%%%%%%%%%%%%%%%%%%%%%%%%%%%%%%%%%%
%% Preorders, topologies, submodular functions

\section{Preorders, topologies and submodular functions}
\label{sec:pre}

We introduce two of the main characters of this article:
(finite) preorders/topologies, and submodular functions.
Finite preorders and finite topologies are equivalent objects by
the Alexandroff correspondence (see \cite{FlSur} for this and a historical
discussion).
Submodular functions occur in
combinatorics, optimization, and economics, where many phenomena
have descriptions by them. Submodular
functions which are allowed to take the value $\infty$
give especially an interplay with preorders which we take a first
look at in this section.
(And is developed much further in Section \ref{sec:conform}.)
The recent book \cite{May2016} shows finite topologies as a rich subject.

\subsection{Preorders and topologies}
A preorder $P = (X, \leq)$ on a set $X$ is a relation which is:
\begin{itemize}
\item[] {\it Reflexive:} $x \leq x$ for $x \in X$,
\item[] {\it Transitive:} $x \leq y$ and $y \leq z$ implies $x \leq z$.
\end{itemize}
The preorder is a {\it partial order} if it is also:
\begin{itemize}
\item[] {\it Anti-symmetric:} $x \leq y$ and $y \leq x$ implies $x = y$.
\end{itemize}

Given a preorder $P = (X, \leq)$, we have an {\it opposite preorder}
$P^{\op} = (X,\leq^\op)$ given by $x \leq^\op y$ if $y \leq x$. 
A subset $D \sus X$ is a {\it down-set} for $P$ if whenever $y \in D$ and
$x \leq y$, then $x \in D$. If $x \in X$ we write $\da{x}$
for the down-set $\{ y \in X \, | \, y \leq x \}$.
We have corresponding notions of {\it up-set} and $\ua{x}$.
We write $\leq_P$ for $\leq$ when we need to be explicit about
which preorder we refer to.
The set of preorders on $X$ is denoted $\Pre(X)$ and is itself partially
ordered by $P \preceq Q$ if $x \leq_P y$ implies $x \leq_Q y$ for
every $x,y \in X$.

Given $P = (X, \leq)$, two $x,y \in X$ are {\it comparable} for $P$
if $x \leq y$ or $y \leq x$. If only one of them holds, say $x \leq y$,
then $y$ is {\it strictly} bigger than $x$.
A preorder is {\it total} if any two elements are comparable. It is
{\it linear} if for any two elements one of them
is strictly bigger than the other.

\medskip

Recall that a topology on the set $X$ is a family $\cT$ of subsets of
$X$ such that:

\begin{itemize}
\item $\emptyset$ and $X$ are both in $\cT$,
\item If $\cU \sus \cT$, then $\cup_{A \in \cU} A$ is an element of $\cT$,
\item If $A,B \in \cT$ then $A \cap B$ is in $\cT$.
\end{itemize}
The elements of $\cT$ are the {\it open} subsets of $X$.
Denote the set of topologies on $X$ by $\Top(X)$. This set is
partially ordered by inclusion. It has a minimal element,
the {\it coarse} topology with only $\emptyset$ and $X$ as open sets.

As is well-known, and explained below,
for a {\it finite} set $X$ there are inverse bijections of partially
ordered sets,
the Alexandroff correspondence:
\[ \Pre(X) \bihom{\topl}{\preo} \Top(X)^{\op}. \]
For a preorder $P$, the set of all down-sets of $P$ form a topology
$\topl(P)$. Conversely, given a topology $\cT$,
define the preorder $\preo(\cT)$ by  $x \leq y$
if for every open subset $U$, if $y \in U$, then $x \in U$.

We shall interchangeably use the terminologies of preorder/down-set and
topology/open set on
$X$. In some situations the preorder notion is more natural, and
in other situations the topology notion is more natural.

\medskip

\begin{definition} \label{def:pre-bc}
A preorder $P$ gives rise to two equivalence relations.
First let $x \sim^b_P y$ if $x \leq_P y$ and $y \leq_Px$.
Its equivalence
classes are the {\it bubbles} of $P$. We denote the set of
bubbles of $P$ by $\mathbf b(P)$.

Secondly let $x \sim^c_P y$ if $x$ and $y$ are in the same connected
component of $P$. So there is a sequence
\[ x = x_0 \leq_P x_1 \geq_P x_2 \leq_P \cdots \geq_P x_{2k} = y.\]
We denote by $\mathbf c(P)$ the set of connected components of $P$.

If ${\mathbf c}(P) = {\mathbf b}(P)$,
this is a {\it totally disconnected} preorder. This is the same
as an equivalence relation. 
The corresponding topology is called a totally disconnected topology.
It is a disconnected union of coarse topologies.
\end{definition}

\begin{definition} \label{def:pre-lin}
  Let $P$ be a preorder and $L$ a total preorder (i.e. every pair of elements
  is comparable). Then $L$ is
  a {\it linear extension} of $P$ if:
  \begin{itemize}
  \item $P \preceq L$,
  \item The sets of bubbles $\mathbf b(P)$ and $\mathbf b(L)$ coincide.
  \end{itemize}
\end{definition}

%\end{document}

\subsection{Submodular functions}
\label{subsec:pre-bf}

%\subsection{Definitions and first properties}\label{smf}
%%%
Denote $\hRR = \RR \cup \{\infty \}$. For $I$ a finite set, let
$\pow(I)$ be the power set, i.e. the set of all subsets of $I$. An
{\it extended Boolean function}
on $I$ is an arbitrary function $z : \pow(I) \to \hRR$ such
that $z(\emptyset) = 0$ and $z(I)< \infty$. Write $\BF[I]$ for the set of extended Boolean functions
on $I$. For $S \sus I$ with $z(S) < \infty$,
the {\it restriction} of $z$ to $S$ is
\[ z\restr S : \pow(S) \to \hRR, \, \text{  given by } z\restr S(U) = z(U)
, \quad U \sus S.\]
The {\it corestriction} of $z$ is
\[z_{/S} : \pow(I \backslash S) \to \hRR, \, \text{ given by }
  z_{/S}(U) = z(S \cup U) - z(S), \quad U \sus I \backslash S. \]
If $S \sus T \sus I$ let $z_{T/S}$ be $(z\restr T)_{/S}$.

\medskip
Let $u \in \BF[S]$ and $v \in \BF[T]$ and $I = S \sqcup T$ the disjoint union.
Their product $u \cdot v \in \BF[I]$ is the function given by
\[ u \cdot v(E) = u(E \cap S) + v(E \cap T), \quad E \sus I. \]
An {extended Boolean function} $z : \pow(I) \pil \hRR$
is {\it decomposable} if it can be written as the product
of two Boolean functions on non-empty subsets.

\begin{lemma} \label{lem:pre-zprod}
  Any extended Boolean function $z \in \BF[I]$ is uniquely a product of indecomposable extended
  Boolean functions.
\end{lemma}

\begin{proof}
  Let $\{I_a\}_{a \in A}$ and $\{I_b \}_{b \in B}$ be two partitions of
  $I$ into non-empty sets.
  Let $z_a = z\restr{I_a}$ and similarly define $z_b$. Suppose
$z$ can be written as products  \[ z = \prod_{a \in A} z_a = \prod_{b \in B} z_b\]
  of indecomposables. We claim that each $I_a \sus I_b$
  for some $b$. So fix $a$. Let $z_{ab} = z\restr{ I_a \cap I_b}$.
  For $S \sus I$ let $S_b = S \cap I_b$. Now:
  \[ z(S) = \sum_{b \in B} z(S_b) = \sum_{b \in B} z_b(S_b). \]
  Then for any $S \sus I_a$, we have $S_b \sus I_b \cap I_a$, and we get:
  \[ z_a(S) = \sum_{b \in B} z_{ab}(S_b). \]
  But since $z_a$ is indecomposable, the sum can involve at most
  one non-trivial $z_{ab}$.
  and so at most one $I_a \cap I_b$ is non-empty.
  Thus $I_a \sus I_b$ for some $b$.
  Conversely, each $I_b \sus I_a$ for some $a$, and so the partitions are equal.
\end{proof}

\noindent The function $z \in \BF[I]$ is {\it submodular} if for $S,T \sus I$:
\begin{equation} \label{eq:notions-submod}
  z (S) + z(T) \geq z(S \cup T) + z(S \cap T).
  \end{equation}
  Write $\SM[I]$ for the set of submodular functions on $I$.
  An equivalent description is that a Boolean function is submodular if and
  only if 
  given $I \supseteq B \supseteq A$ and varying
  $S \sus I \backslash B$, then:
  \[ z(B \cup S) - z(A \cup S) \]
  is a weakly decreasing function as $S$ increases for the inclusion relation
  (with the convention that $\infty - \infty$ is any value).
%The following is an interesting and significant observation (although
%we do not use it explicitly below):
The following will be useful.

\begin{lemma}\label{lem:prod-sum}
  Let $z : \pow(I) \pil \hRR$ be a submodular function, and 
$I \supseteq T \supseteq S$.  Let 
 $C = T \backslash S$, and $C = C_1 \sqcup C_2$
  be a disjoint union of sets. Suppose $z(S), z(T) < \infty$. Then
we have a decomposition as a product:
  \begin{eqnarray} \label{eq:prod-zc}
    z^\prime = z_{T/S} = z^\prime\restr{C_1} \cdot z^\prime\restr{C_2}
    \end{eqnarray}
if and only if
\begin{eqnarray} \label{eq:prod-sum}
  z(T) + z(S) 
  = z(S \cup C_1 \cup C_2) + z(S) = z(S \cup C_1) + z(S \cup C_2).
\end{eqnarray}

\end{lemma}

\begin{proof}
If \eqref{eq:prod-zc} holds, then \eqref{eq:prod-sum} holds.
Suppose conversely \eqref{eq:prod-sum} holds. Note then in
particular that all values in that equation are finite.
Let $S_1 \sus C_1$ and $S_2 \sus C_2$. We must show that
\[ z^\prime(S_1 \cup S_2) = z^\prime(S_1) + z^\prime(S_2), \]
equivalently
\[ z(S \cup S_1 \cup S_2) - z(S) = z(S \cup S_1) - z(S) + z(S \cup S_2) - z(S)
  \]
  which is again equivalent to
  \begin{equation} \label{eq:notions-zss}
    z(S \cup S_1 \cup S_2) + z(S) = z(S \cup S_1) + z(S \cup S_2).
  \end{equation}
  By submodularity
  \begin{align*}
    z(S \cup C_1 \cup S_2) + z(S \cup S_1) & \leq z(S \cup S_1 \cup S_2) +
                                             z(S \cup C_1) \\
     z(S \cup S_1 \cup C_2) + z(S \cup S_2) & \leq z(S \cup S_1 \cup S_2) +
                                              z(S \cup C_2)
 \end{align*}
 These imply that if $z(S \cup S_1)$ or $z(S \cup S_2)$ is $\infty$,
 then $z(S \cup S_1 \cup S_2) = \infty$, and so settles \eqref{eq:notions-zss}
 in this case.

 So assume $z(S \cup S_1)$ and $z(S \cup S_2)$ are finite.
 Then \eqref{eq:notions-zss} is equivalent to:
  \begin{eqnarray} \label{eq:prod-zas-one}
    z(S \cup S_1 \cup S_2) - z(S \cup S_1) =  z(S \cup S_2) - z(S).
  \end{eqnarray}
  Fix $S_2$. The above \eqref{eq:prod-zas-one} is an equality when $S_1 =
  \emptyset$. By submodularity
  \[ z(S \cup S_1 \cup S_2) - z(S \cup S_1) \]
  is weakly decreasing with increasing $S_1$.
  % (with the convention that $\infty - \infty$ is all values).
  If therefore
  \begin{eqnarray} \label{eq:prod-zcs}
    z(S \cup C_1 \cup S_2) - z(S \cup C_1) = z(S \cup S_2) - z(S),
    \end{eqnarray}
    we have equality in \eqref{eq:prod-zas-one} for every $S_1 \sus C_1$.
%    with $z(S \cup S_1)$ finite.
    Now \eqref{eq:prod-zcs} can be written
\begin{eqnarray} \label{eq:prod-zas}
      z(S \cup C_1 \cup S_2) - z(S \cup S_2) = z(S \cup C_1) - z(S).
\end{eqnarray}
This equality holds if $S_2 = \emptyset$. By assumption it also holds
if $S_2 = C_2$. Hence by submodularity it holds for every $S_2 \sus C_2$
with $z(S \cup S_2)$ finite. This implies that \eqref{eq:prod-zcs}
and so \eqref{eq:prod-zas-one} holds.
\end{proof}
%%%

\subsection{From submodular functions to
  topologies and preorders}  \label{subsec:prod-subpretop}
The following is a simple and significant observation.
Let $I$ be
finite and given a submodular function $z$. Let
\[\top(z) = \{ S \sus I \, |\, z(S) < \infty \}. \]
This is a topology on $I$. This is because it is closed under unions
and intersection due to the submodular inequality \eqref{eq:notions-submod}.
The topology $\top(z)$ corresponds to a preorder $\pre(z)$ on $I$, 
and $z(S)$ is finite if and only if $S$ is a down-set for this preorder.
So we have maps:
\[ \SM[I]\, \,  \substack{ \top \\ \pil \\ \pil \\ \pre} \, \, 
  \begin{cases} \Top(I) \\ \Pre(I) \end{cases} \]
which via the identification of $\Top(I)^{\op}$ and $\Pre(I)$ is
the same map.  This map has a left adjoint
sending a preorder $P$ to the submodular function
\[ \low_P : \pow(I) \pil \hat{\RR}, \quad \text{where } \low_P(S) =
  \begin{cases} 0, & S \text{ a down-set of } P \\
    \infty, & S \text{ not down-set of } P
  \end{cases}. \]
We see that submodular functions $z : \pow(I) \pil \{ 0, \infty\}$ are
in bijection
with preorders on $I$, \cite[Section 13]{AA2017}.

\medskip
We can decompose $z$ into indecomposable submodular functions
\begin{equation} \label{eq:prod-dec} z = z_1 \cdot z_2 \cdots z_r
  \end{equation}
where $z_i$ is defined on subsets $I_i \sus I$.
This gives a partition $I = I_1 \sqcup I_2 \sqcup \cdots \sqcup I_r$.
Taking the disconnected union of the coarse topologies on each $I_i$,
we get a totally disconnected topology $\ctop(z)$.
Denote the corresponding preorder as $\cpre(z)$.
It is the equivalence relation with the $I_i$ as equivalence classes
(or bubbles).
Note that
\[ \pre(z) \preceq \cpre(z), \quad \ctop(z) \sus \top(z). \]

\subsection{Matroids and polymatroids} \label{subsec:pre-mat}
Let $\NN_0 = \{0, 1, 2, \ldots \}$.
A submodular function $z : \pow(I) \pil \NN_0$ is a {\it polymatroid}
if it is weakly increasing, i.e. $S \sus T$ implies $z(S) \leq z(T)$.
It is a {\it matroid} if also $z(S) < |S|$, where the latter denotes
the cardinality of $S$. 
This is one of several equivalent ways of defining a matroid,
\cite[Thm.1.3.2]{Ox2006}.
We thus see that submodular functions give a common framework for
two central objects in combinatorics: matroids and partial orders.

To preorders there is associated two cointeracting bialgebras,
\cite{FFM2017}. For matroids there is also associated a Hopf algebra, by
restriction and contraction, \cite{DFM2018}.
But there has not yet appeared a corresponding
cointeracting bialgebra. In Section
\ref{sec:bim}, we give in full generality
two cointeracting bialgebras,
one of submodular functions and the other of modular functions,
and this specializes both to the cointeraction
for preorders and to a cointeraction including matroids.

%%%%%%%%%%%%%%%%%%%%%%%%%%%%%%%%%%%%%
%% Extended generalized permutahedra

\section{Extended generalized permutahedra}
\label{sec:egp}

To submodular functions is associated beautiful geometry.
By linear inequalities they define convex polyhedra, and if
$z$ only takes finite values, these are polytopes: the generalized
permutahedra. We recall their definition and basic properties, in particular
with respect to dimension, and we describe how the inequalities cut out faces.
%The material
%here is known, and can be found in \cite{Fuj2005} and also in
%\cite{Pos2009} (but there
%submodular functions are not introduced).
Our take on this is related to preorders
and topologies, and we therefore develop arguments and results
from this perspective.

\subsection{Definition and examples}
Again $I$ is a finite set. Let $\RR I$ be the vector space with
basis $I$. We shall normally write $e_i$ for the basis element $i$.
The elements of $\RR I$ are thus sums $\sum a_i e_i$ where
$a_i \in \RR$. We may also write $(a_i)_{i \in I}$.  
%We recall definitions from \cite{AA2017}:
The \textsl{extended generalized permutahedron (EGP)} associated to a submodular
function $z:\pow(I)\to \hRR$ 
is given by:
\begin{equation} \label{eq:egp-def}
  \egp(z):=\left\{x=(x_i)_{i\in I}\in\RR I \, | \,
    \sum_{i\in I}x_i=z(I) \hbox{ and } \sum_{i\in A}x_i\le z(A) \hbox{ for any }
    A\sus I \hbox{ such that } z(A)< \infty\right\}.
  \end{equation}

This is a convex polyhedron. It is a polytope, a
{\it generalized permutahedron (GP)},
if and only if $z:\pow(I)\to\mathbb R$, i.e. $z$ only takes finite values.
When $z$ is submodular each of the inequalities in \eqref{eq:egp-def}
is optimal: making any of them an equality cuts out a non-empty face
of the EGP, by Lemma \ref{lem:egp-lin} below.
When $x \in \RR I$ and $A \sus I$,
we for short write $x_A = \sum_{i \in A} x_i e_i$. 

\begin{example}
  Given a strictly decreasing sequence
  $\ell_1 > \ell_2 > \cdots  > \ell_n$, for $S \sus I$ let 
  $z(S) = \sum_{i \leq |S|} \ell_i$. This is a submodular function,
  and the associated GP is the permutahedron, the convex hull of
  the vertices $(\ell_{\pi(1)}, \ldots, \ell_{\pi(n)})$ where $\pi$ runs
  through all permutations of $\{1,2, \ldots, n\}$. 
\end{example}

\begin{figure}
\begin{center}
%Første bilde
\begin{tikzpicture}[inner sep=1pt, scale=0.7]

\coordinate (a) at (0,0) ; % [anchor =  north east] {$(1,2,3)$};
\coordinate (b) at (3, 1.71);  %[anchor = north west] {$(2,1,3)$};
\coordinate (c) at (3, 5.15);  %[anchor = north east] {$(3,1,2)$};
\coordinate (d) at (0, 6.86);  %[anchor = north east]{$(3,2,1)$};
\coordinate (e) at (-3, 5.15);  % [anchor = north west] {$(2,3,1)$};
\coordinate (f) at (-3, 1.71);  % [anchor = south west] {$(1,3,2)$};

    %Draw permutohedron: 
            \draw[black, thick, fill=gray, fill opacity = 0.1] 
            (a) -- (b)  -- (c) -- (d) -- (e) -- (f) -- (a);
            %Draw centerpoint and vertices:
            %\draw[black, fill] (0.5,)    circle (1pt);
            \draw[black, fill] (a)     circle (2pt);
            \draw[black, fill] (b)     circle (2pt);
            \draw[black, fill] (c)     circle (2pt);
            \draw[black, fill] (d)     circle (2pt);
            \draw[black, fill] (e)     circle (2pt);
            \draw[black, fill] (f)     circle (2pt);
            % \node[] at (-0.5, 0.8660)   [anchor = east]       { $f$};

\node[] at (a) [anchor = north east] {$(1,2,3)$};
\node[] at (b) [anchor = north west] {$(2,1,3)$};
\node[] at (c) [anchor = south west] {$(3,1,2)$};
\node[] at (d) [anchor = south west] {$(3,2,1)$};
\node[] at (e) [anchor = south east] {$(2,3,1)$};
\node[] at (f) [anchor = north east] {$(1,3,2)$};
      
          \end{tikzpicture}
          \end{center}
        \caption{Two-dimensional permutahedron}
        \label{fig:hexagon}
      \end{figure}
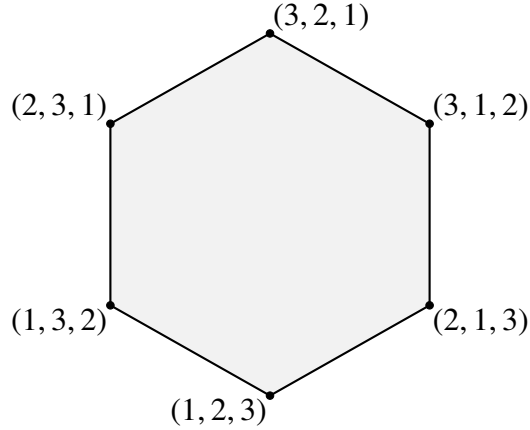

\begin{lemma} \label{lem:egp-prodsum}
  Let $I = S \sqcup T$.
  If $z \in \SM[I]$ decomposes as a product $z = u \cdot v$,
  where $u \in \SM[S]$ and
  $v \in \SM[T]$, then
  $\egp(z)$ on $\RR I$ is the product
  $\egp(u) \times \egp(v) \sus \RR S \times \RR T$.
  \end{lemma}

  \begin{proof}
    If a point is in the product, clearly it is in $\egp(z)$.
    Conversely if $x \in \egp(z)$, write $x = (x^1, x^2)$ where $x^1 \in \RR S$
    and $x^2 \in \RR T$. For $A_1 \sus S$ we then have
    \[ \sum_{i \in A_1} x^1_i = \sum_{i \in A_1} x_i \leq z(A_1) = u(A_1), \]
    and similarly $\sum_{i \in A_2} x^2_i \leq v(A_2)$. 
    Also, letting $A_1 = S$ and $A_2 = T$:
    \[ \sum_{i \in S} x^1_i \leq u(S), \quad \sum_{i \in T} x^2_i \leq v(T). \]
    Summing these we get an equality, so each of them must be an
    equality.
  \end{proof}

  \begin{definition}\label{def:alin}
  Given a preorder $P$ and submodular function $z$,
  the associated affine linear space is:
  \[\alin(P, z) = \{x \, | \, x_A = z(A), \text{ for every }
    \text{ down-set } A \text{ of } P \}. \]
If $z$ is understood we simply write $\alin(P)$. 
\end{definition}

\begin{lemma} \label{lem:egp-lin}
  Let $L$ be a linear extension of $\pre(z)$
  (recall Definition \ref{def:pre-lin}).
  Then $\alin(L) \sus \egp(z)$. In particular:
  \begin{itemize}
   \item[i.] $\alin(L)$ is non-empty and the same holds for $\egp(z)$.
   \item[ii.] For each $S$ with $z(S) < \infty$, equality $x_S = z(S)$
     is attained for some $x \in \egp(z)$.
   \end{itemize}
   \end{lemma}

   \begin{proof}
     Firstly, $\alin(L)$ is defined by $\leq |I|$ equations of triangular form,
     and so it is non-empty. Secondly $L$ may be chosen such that $S$ is
     a down-set of $L$. 

     Let $x \in \alin(L)$, so $x_B = z(B)$ for every down-set
  $B$ of $L$. Let $z(A) < \infty$, equivalently $A$ is a down-set of
  $\pre(z)$.  Let $C$ be the bubble in $A$, maximal
  in $L$. We show by induction on $C$ in $L$ that $x_A \leq z(A)$.
  Let $B = (\dai{L}{}C) \backslash C$. 
  Since $B \cup A = \dai{L}{C}$ and $B$ are both down-sets of $L$,
  we have $x_{ B \cup A} = z(B \cup A)$ and $x_B = z(B)$.
  Since $B$ is a down-set for $L$ it is also a down-set for $\pre(z)$.
 Whence $B \cap A$ is a down-set for $\pre(z)$.
  Then by induction $x_{B \cap A} \leq z(B \cap A)$.
  Thus
  \[ x_B + x_A = x_{ B \cup A} + x_{B \cap A} \leq
    z(B \cup A) + z(B \cap A) \leq z(B) + z(A). \]
  Since $x_B = z(B)$ we get $x_A \leq z(A)$.
\end{proof}

The following proposition and corollary is essentially found
in \cite{Fuj2005}. In particular the argument below is based
on \cite[Theorem 3.36]{Fuj2005}.

\begin{proposition} \label{pro:egp-top}
  Given a submodular function $z$, the set of 
  $A \sus I $ such that $x_A = z(A)$ for every $x \in \egp(z)$,
  forms a topology on $I$. This is the totally disconnected topology $\ctop(z)$.
  In particular, if $z$ is indecomposable each  $x_A = z(A)$ for
  $\emptyset \neq A \neq I$ cuts out a proper face of $\egp(z)$. 
\end{proposition}

\begin{proof} First note the $A = \emptyset$ and $A = I$ fulfill
  the condition.
  Suppose $x_A = z(A)$ and $x_B = z(B)$.
  Then
  \[ x_A + x_B = z(A) + z(B) \geq z(A \cup B) + z(A \cap B) \geq
    x_{A \cup B} + x_{A \cap B} = x_A + x_B. \]
  This forces everything to be equalities and so $x_{A \cup B} = z(A \cup B)$
  and $x_{A \cap B} = z(A \cap B)$. Thus the sets for which we have
  equality form a topology.

  \medskip
  Now assume $z$ is indecomposable.
Suppose there is an $\emptyset \neq A \neq I$ with $x_A = z(A)$ for every
$x \in \egp(z)$. Let $B = I \backslash A$ be the complement.
%  \begin{equation} \label{eq:egp-AB}
%    x_A + x_B  = x_I = z(I) < z(A) + z(B)
%    \end{equation}
%  and so $x_B < z(B)$ for every $x \in \egp(z)$.
Recall that $x_I = z(I)$ for every $x \in \egp(z)$. 
For $b \in B$ let $T_b \sus I$ be a minimal set containing $b$
and with $x_{T_b} = z(T_b)$
  for every $x \in \egp(z)$. 
  We claim $T_b \sus B$. Otherwise let $a \in T_b \cap A$. For each $S$
  containing $b$ and not $a$, those $x$ with
  $x_S = z(S)$ cut out a proper face of $\egp(z)$. Thus there will
  be an $x \in \egp(z)$ with $x_S < z(S)$ for every such $S$.
  Set  $\mu$ to be the minimal of such $z(S) - x_S$,
  and $x^\prime = x + \mu e_b - \mu e_a$.
  Then $x^\prime_J \leq z(J)$ for every $J$ with $z(J) < \infty$,
  so $x^\prime \in \egp(z)$. But this contradicts $x^\prime_A = z(A)$.

  Thus $T_b \sus B$. Since such sets form a topology, also for 
  $\cup_{b \in B} T_b = B$ have $x_B = z(B)$ for every $x \in \egp(z)$.
  But:
  \begin{equation} \label{eq:egp-AB}
    x_A + x_B  = x_I = z(I) \leq z(A) + z(B) = x_A + x_B,
  \end{equation}
  This forces equality above and $I$ would be decomposable by Lemma
  \ref{lem:prod-sum}.

  \medskip
  Lastly, if $z$ is any submodular function with a decomposition
  as in Equation \ref{eq:prod-dec}, then if $x_A = z(A)$, then
  for each $A_i = A \cap I_i$ we would have $x_{A_i} = z(A_i)$, and
  so each $A_i$ is either $\emptyset$ or $I_i$. 
\end{proof}

\begin{corollary} \label{cor:egp-dim}
 If $z$ is indecomposable, there is $x \in \egp(z)$ such
  that $x_A < z(A)$ for every $\emptyset \neq A \neq I$ with
  $z(A) < \infty$. So the dimension of $\egp(z)$ is $|I|-1$.
  More generally, for any submodular $z$ the dimension of $\egp$
  is $|I| - |\ctop(z)|$, where $|\ctop(z)|$ is the number of components of the
  topology $\ctop(z)$.
\end{corollary}

\begin{proof} When $z$ is indecomposable, 
  for each such $A$ the equation $x_A = z(A)$ cuts out a proper face
  of $\egp(z)$. Hence there will be some $x \in \egp(z)$ such that
  $x_A < z(A)$ for every $A \neq \emptyset, I$.
  The only restriction on perturbing $x$ is the equation
  $\sum_{i \in I} x_i = z(I)$. 
  Thus the dimension of $\egp(z)$ is
  $|I|-1$.
  For general $z$ with the decomposition of Equation \ref{eq:prod-dec},
  by Lemma \ref{lem:egp-prodsum} the dimension of $\egp$
  is $\sum_{i = 1}^r (|I_i|-1)$.
  \end{proof}

  %%%%%%%%%%%%%%%%%%%%%%%%%%%%%%%%%%%%%%%%%%%%
  %% Classes of egp

\section{Classes of extended generalized permutahedra}
\label{sec:egpclass}

We recall two constructions of classes of generalized permutahedra.
These are by using Minkowski sums and the nestohedra from building sets.
Building sets and nestohedra are discussed comprehensively in
\cite{AA2017, PRW2009, Pos2009}. They are also discussed in the more
general context of geometric lattices in \cite{FeSt2004}, while
their origin stem from \cite{DC-Proc}.

\subsection{Construction by Minkowski sums}
A special class of EGP's arise as follows.
For $J \sus I$ let the simplex $\Delta_J \sus \RR J \sus \RR I$ be given
by:
\[\Delta_J =\left \{ \sum_{j \in J} \alpha_j e_j \, | \, \alpha_j \geq 0,
  \sum_{j \in J} \alpha_j = 1 \right\}. \]
Given a Boolean
function to the non-negative real numbers $y : \pow(I) \pil \RR_{\geq 0}$,
the Minkowski sum
\begin{equation} \label{eq:EGP-Mink} \sum_{J \sus I} y_j \Delta_J
  \end{equation}
  is a generalized permutahedron.
As such it corresponds to a submodular function as we now
explain. For $i \in I$ let $\one_i$ be the dual basis element of $e_i$.
For $A \sus I$ consider the linear function $\fone_A = \sum_{a \in A}
\one_a$ on
$\RR I$, so $\fone_A(\sum_{i\in I} \alpha_i e_i) = \sum_{i \in A}\alpha_i$. 
The value of this on
a point $p \in y_J \Delta_J$ is $0$ if $A \sus J^c$, it is $y_J$ if
$J \sus A$ and otherwise
$\fone_A(p)$ may take any value in $[0, y_J]$ as $p$ varies.
In particular if $A \cap J \neq \emptyset$, then $\fone_A$
will take the maximum value $y_J$ on some point in $\Delta_J$.
On the Minkowski sum \eqref{eq:EGP-Mink}, the maximum value of $\fone_A$ is
$\sum_{J \cap A \neq \emptyset} y_J$. This is then the value
$z(A)$ of the associated submodular function \cite[Proposition 6.3]{Pos2009} (see also \cite[Proposition 19.1]{AA2017}).

\subsection{Building sets and nestohedra}
\label{subsec:egp-building}
A family of subsets $\cB$ of $I$ is
a {\it building set} on $I$ if it satisfies the following conditions:
\begin{itemize}
\item If $J,K \in \cB$ with $J \cap K \neq \emptyset$, then $J \cup K
  \in \cB$,
\item For every $i \in  I$, we have $\{i \} \in \cB$.
\end{itemize}
The maximal elements of $\cB$ form a partition of $I$, and they
are the {\it connected components} of $\cB$.

\begin{example}
A typical example of a building set arises from a simple graph $G = (I,E)$,
with vertices $I$. A set $J \sus I$ is in the building set if the
induced graph on $J$ is connected.
\end{example}

%For $J \sus I$, let $\Delta_J$ be the simplex in $\RR^J \sus \RR^I$, consisting
%of all points $\sum_{j \in J} \alpha_j e_j$, where $e_j$ is the coordinate vector
%corresponding to $j$, and $\alpha_j \geq 0$ and $\sum_{j \in J} \alpha_j = 1$.

The generalized permutahedron $\Delta_{\cB}$ associated to
the building set $\cB$ is the Minkowski sum
\[ \DB = \sum_{J \in \cB} \Delta_J\]
and is called a {\it nestohedron}.
%By \cite{AA2017}, it  corresponds to the submodular function
The faces of a nestohedron $\DB$ are transparently described via
the notion of 
{\it nested sets} or equivalently
$\cB$-{\it forests}. We recall this notion, but first define
the following.

\begin{definition}
  A preorder $P$ is a {\it forest}, if for each pair $p,q$ of
  incomparable elements in $P$,
  there is no element $r$ in $P$ with $r \geq p$ and
  $r \geq q$.
  A minimal element of $P$ is a {\it root element}. The bubble of such
  a $p$ is a {\it root bubble}.
  %(Note that such a $p$ may occur
  %in a bubble with other elements.) 
  %An element $p$ for which there is no $r$ with $r < p$, is a
  %{\it root element}.
  %The bubbles of the root elements are the {\it roots} of the forest.
\end{definition}

Given a building set $\cB$ on $I$, a $\cB$-{\it forest}  $N$
is a forest on $I$ such that:

\begin{itemize}
\item For any element $p \in I$, the upset $\uai N p$ is in $\cB$.
\item If $p_1, \ldots, p_k$ are pairwise incomparable elements and
  $k \geq 2$,
  the union $\cup_{i = 1}^k \uai N {p_i}$ is {\it not} in $\cB$,
\item The {connected} components of $\cB$ are in bijection with the
  root bubbles of $N$. Any such component is 
  $\uai N r$ for a root element $r$.
 
\end{itemize}

\noindent {\bf Note.}
The family of upsets of a $\cB$-forest form what is called a {\it nested set}
for the building set.

\medskip
%\subsection{The faces of the building permuathedron}
By \cite{Pos2009}, \cite{FeSt2004}, the faces of the given nestohedron
$\DB$ are in bijection with the $\cB$-forests. The vertices
of $\DB$ correspond to the $\cB$-forests which are posets (i.e. every
bubble has cardinality one). This is what is usually called  a rooted
forest, and each component is then a rooted tree.
Section \ref{sec:bij} extends fully this description
of faces:
For {\it any} extended generalized permutahedron in $\RR I$ we give a natural
bijection between its faces and classes of preorders in $\Pre(I)$.

%%%%%%%%%%%%%%%%%%%%%%%%%%%%%%%%%%%%%%%%%%%%%
%% Braid fan

\newcommand{\bfan}{{\cB}}

\section{The braid fan}
\label{sec:braid}

The braid fan in $\RR^I$ has maximal cones in bijection with total
orders on the set $I$. The somewhat amazing thing is that
cones of the braid fan (defined below) are {\it precisely in bijection}
with preorders on the set $I$. The braid fan is a standard object,
see \cite[Chapter 10]{AM2010} for a detailed consideration.
It relates also to representation theory
of Lie groups and Coxeter groups, where its maximal face cones are the
Weyl chambers for type A representations.
Our interest in the braid fan is because normal fans of generalized
permutahedra are coarsenings of the braid fan.

\medskip
Recall the following.
The vector space $\RR I$ has basis $I$, and we write $e_i$ for $i$.
Elements in $\RR I$ are then $\sum_{i \in I} a_i e_i$ where $a_i \in \RR$.
Denote by $\RR^I = (\RR I)^*$ the dual space of linear maps $\RR I \pil \RR$.
The elements $y$ of $\RR^I$ are called {\it directions} for $\RR I$ and
identify as set maps $y : I \pil \RR$. 
We write $\{ \one_i \, | \, i \in I \}$ for the dual basis of
$\{e_i \, | \, i \in I \}$. Points in $\RR^I$ may be written
$b = \sum_{i \in I} b_i \one_i$ where $b_i \in \RR$.
We write $\fone = \sum_{i \in I} \one_i$.
The coordinate function $x_i$ on $\RR^I$ is given by $x_i(b) = b_i$.
Since the dual space $(\RR^I)^*$ naturally identifies as $\RR I$, the coordinate
$x_i$ naturally identifies as $e_i$.

\medskip
For a good, concise and precise outline of the definitions of
cones, polytopes, fans, normal fans of polytopes and more
(in less than two pages),
we refer to \cite[Section 2]{PRW2009}.
The {\it braid arrangement} on $\RR^I$ is defined by the hyperplanes $x_i
= x_j$, for $i, j$ distinct in $I$. It induces a fan on $\RR^I$, the
{\it braid fan} $\bfan_I$. The maximal face cones of this fan,
called {\it chambers}, correspond to total orders on $I$. If
$i_1 < i_2 < \cdots < i_n$ is such a total order,
the corresponding chamber is the set
of points $y$ in $\RR^I$ such that
\[ y_{i_1} \geq y_{i_2} \geq \cdots \geq y_{i_n}. \]
In general the face cones of the braid fan are given by {\it total preorders}
on $I$. A total preorder $L$ on $I$ gives a face cone $k(L)$ consisting
of points in $\RR^I$ such that $y_i \geq y_j$ if $i \leq j$ for $L$.
In particular, for each bubble $C = \{\ell_1, \ell_2, \cdots, \ell_r\}$
  of $L$ we have
\[y_{\ell_1} = y_{\ell_2} = \cdots = y_{\ell_r}. \]
(In the literature a total preorder on $I$ is often called a
set composition of $I$:
  It is an ordered sequence $(C_1, C_2, \ldots, C_r)$ of disjoint
  non-empty subsets of $I$, whose union is $I$, \cite[Sec.10.1.2]{AM2010}.)

  \medskip
  Each hyperplane $x_i = x_j$ determines two half spaces $x_i \geq x_j$
  and $x_j \geq x_i$. Following \cite[Sec.10.2.4]{AM2010} and
  \cite[Sec.3.4]{PRW2009} a {\it cone} of the braid arrangement is an
  intersection of such half spaces.
  A preorder $P$ on $I$ induces such a cone $k(P)$ in $\RR^I$
  given by $x_i \geq x_j$ if $i \leq j$.
  Proposition 3.5 of \cite{PRW2009} comprehensively describes cones
  in the braid arrangement and their bijection to preorders.
  We state an essential part of it below, and also add an equivalent
  description of braid cones.

  \begin{proposition} \hskip 1mm {} \label{pro:braid-cone}

    \begin{itemize}
\item[a.] The correspondence
    \[ \Pre(I) \longrightarrow \text{cones of braid arrangement in } \RR^I,
        \quad P \mapsto k(P) \]
      is a bijection.
\item[b.] The maximal face cones
  of $k(P)$ are the $k(L)$ as $L$ ranges over the linear extensions
  of $P$.
\item [c.] Every cone in $\RR^I$ which is a union of face cones of the braid
      fan $\bfan_I$, is  $k(P)$ for a preorder $P$.
    \end{itemize}
\end{proposition}

Parts a. and b. are found in \cite[Cor.13.8]{AM2010}.
A bit earlier the proof of part a is found in \cite[Prop.3.5]{PRW2009}.
Part b is also stated
after that proposition but the argument not given quite explicitly.
We show here part c.

\begin{proof}
  \ignore{
\noindent{a.} For a preorder $P$, it is clear that $k(P)$ is a cone of the braid
  arrangement. Conversely, given a cone $K$ of the braid arrangement which
  is the intersection of half planes $x_i \geq x_j$ where
  $(i,j)$ ranges over a subset of $I \times I$,
  let $P$ be the transitive closure of this relation.
  Then $k(P) = K$.

  \medskip
\noindent{b.} Note that if $P \preceq Q$ then $k(Q) \sus k(P)$. The inclusion $k(L) \sus k(P)$ therefore holds for each linear extension
  $L$ of $P$. In fact $k(P)$ is the union of all cones $k(L)$ where $L$ runs over the linear extensions of $P$:
  Indeed, given a point $x \in k(P)$, the total preorder $T_x$ on $I$
 given by $i \leq j$ if $x_i \leq x_j$, fulfills $P \preceq T_x$.
 Each bubble of $T_x$ is then a union of bubbles of ${\mathbf b}(P)$. So
 we get induced a partition of ${\mathbf b}(P)$.
 For a part $B$ of this partition, let $P_B$ be $\cup_{{\mathbf b}(P) \in B}
 {\mathbf b}(P)$, a subset of $P$. Put a linear order on the bubbles of
 $B$ giving a linear extension of $P_B$. Putting these total preorders
 together, we get a total preorder $L$ such that $P \preceq L \preceq T_x$.
 So $x \in k(T_x) \sus k(L) \sus k(P)$. 
  Thus $k(P)$ is the union of $k(L)$ over the linear
  extensions $L$ of $P$.
  }
 \noindent{c.}
  Given a cone $K$ as in c., let us show that $K = k(P)$ for a preorder $P$.
  Let $P$ be the preorder associated to $K$ as in the argument for
  part a, so $K \sus k(P)$.
  Define an equivalence
  relation on $I$ by $i \sim j$ if $y_i = y_j$ for every $y \in K$.
  The equivalence classes are the bubbles of $P$.
  For every pair $C,D$ of distinct bubbles there is some point $y^{CD} \in K$
  with $y^{CD}_{C} \neq y^{CD}_{D}$.
  Let $y = \displaystyle \sum_{C,D \in \mathbf b(P)} \alpha_{CD} y^{CD}$ be a general linear combination,
  where $\alpha^{CD} \in \RR$. Then $y_{C} \neq y_{D}$ for each distinct
  pair of bubbles. By assumption of item c., $y$ is in a cone $k(L) \sus K$.
  Then $\mathbf b(L) =\mathbf  b(P)$, and the dimensions of $k(L)$, $K$, and $P$ all coincide with the cardinal $\vert\mathbf b(P)\vert$. The total preorder $L$ yields a total preorder on the bubbles,
  say
  \[ C_1 > C_2 > \cdots > C_{i-1} > C_i > C_{i+1} > C_{i+2} > \cdots > C_d. \]
  Suppose $C_i$ and $C_{i+1}$ are not comparable in $P$. Switching
  $C_i$ and $C_{i+1}$, there is then
  a linear extension $L^\prime$ of $P$ given by
  \[ C_1 > C_2 > \cdots > C_{i-1} > C_{i+1} > C_{i} > C_{i+2} > \cdots > C_d. \]
  We claim that $k(L^\prime) \sus K$. Since $k(L)$ is a closed cone, there
  is a point $\overline{y} \in k(L)$ with
  \[ \overline{y}_{C_1} >  \cdots > \overline{y}_{C_{i-1}} >
    \overline{y}_{C_i} =   \overline{y}_{C_{i+1}} >  \overline{y}_{C_{i+2}}
    > \cdots >  \overline{y}_{C_d}. \]
  There is also a point $x \in K$ with $x_{C_{i+1}} > x_{C_i}$.
  A point $y^\prime = \epsilon x + \overline{y}$ is then in $K$ and for
  $\epsilon$ sufficiently small
  \[ y^\prime_{C_1} >  \cdots > y^\prime_{C_{i-1}} >
    y^\prime_{C_{i+1}} >   y^\prime_{C_{i}} >  y^\prime_{C_{i+2}}
    > \cdots >  y^\prime_{C_d}. \]
  Then $K$ intersects the interior of $k(L^\prime)$ and so $k(L^\prime)$ is
  contained in $K$. We may in this way continue and switch bubbles, such
  that for every linear extension $\tilde{L}$ of $P$ we have $k(\tilde{L})
  \sus K$. Hence the union of all such cones, over the linear extensions,
  are in $K$ and since this union is $k(P)$, we get $K = k(P)$. 
\end{proof}

All face cones of the braid fan contain the line $\RR \fone$.
It is therefore equivalent and convenient to consider the {\it reduced} braid
fan in the quotient space $\RR^I /\RR \fone$. All face cones of the
reduced braid fan then become {\it pointed}, i.e. their minimal face is the
origin.

\medskip
%A point in $\RR^I$, this identifies as a function on $\RR I$,
%defines a direction in $\RR I$.
Given any polyhedron $\Pi$ in $\RR I$,
the points of $\Pi$ for which a direction in $\RR^I$ reaches maximal value
form a face of $\Pi$, the maximal face for the direction.
We consider two directions
equivalent if they have the same maximal face $F$ of $\Pi$. Such
an equivalence class of directions
form an open cone.
The closure of this is the cone $k(F)$ associated to the face $F$.
These closed cones form a fan, the {\it normal fan} 
of the polyhedron. Note that smaller faces $F$ correspond to larger cones
$k(F)$.

%%%%%%%%%%%%%%%
%% Her er en-dimensjonale permutaeder

\begin{figure}
\begin{center}
%Første bilde
\begin{tikzpicture}[scale=1, vertices/.style={draw, fill=black, circle,
inner sep=1.5pt}]

\draw [help lines, white] (-2,-1) grid (3,2);
 %Make axes:
        \draw[black, ->] (-2,0) -- (3.5,0) node [anchor=west] {$x_a$};
        \draw[black, ->] (0,-2) -- (0,3.5) node [anchor=west] {$x_b$};
      %%\draw[black, thick,->] (0,0,0) -- (0,0,3.5) node [anchor=west] {z};

\coordinate (a) at (2,1);
\coordinate (b) at (1,2);

\node[vertices] at (a) {}; 
\node[vertices] at (b) {};
\node[anchor=north] at (a) {$(2,1)$}; 
\node[anchor=south] at (b) {$(1,2)$};

\draw[black, very thick] (a)--(b);

\end{tikzpicture}
\end{center}

 \caption{One-dimensional permutahedron} %, three faces, $I = \{a,b\}$}
        \label{fig:perm-1}
\end{figure}
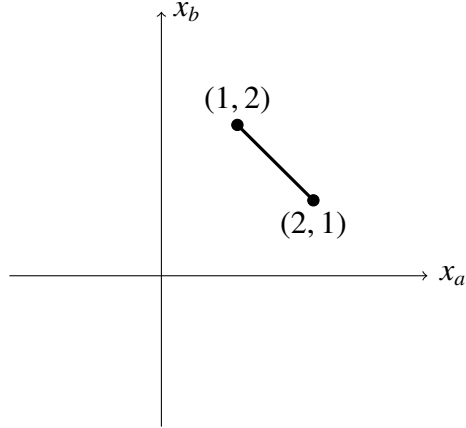

%%%%%%%%%%%%%%%
%% Her er viften til en-dim permutaeder

\begin{figure}
\begin{center}
%Første bilde
\begin{tikzpicture}[scale=0.7, vertices/.style={draw, fill=black, circle,
inner sep=1.5pt}]

\draw [help lines, white] (-2,-1) grid (3,2);
 %Make axes:
        \draw[black, ->] (-3.5,0) -- (3.5,0) node [anchor=west] {$y_a$};
        \draw[black, ->] (0,-3.5) -- (0,3.5) node [anchor=west] {$y_b$};
      %%\draw[black, thick,->] (0,0,0) -- (0,0,3.5) node [anchor=west] {z};

\coordinate (a) at (-3,-3);
\coordinate (b) at (3,3);
\coordinate (c) at (-5,-1);
\coordinate (d) at (1,5);
\coordinate (e) at (-1,-5);
\coordinate (f) at (5,1);

\draw[black, very thick] (a)--(b);

\draw[white, fill=gray, fill opacity = 0.1] (c)--(d)--(f)--(e) ;

%\end{tikzpicture}

\coordinate (p) at (-1.4 + 6.8,0);
\node at (p) {$\rightsquigarrow$}; 
%\begin{tikzpicture}[scale=0.7, vertices/.style={draw, fill=black, circle,
%    inner sep=1.5pt}]

\coordinate (x) at (-1.4 + 8,1.4);
\coordinate (y) at (1.4+8,-1.4);
\coordinate (z) at (0+8,0);

\node[vertices] at (z) {}; 
\draw[black, very thick] (x)--(y);
\end{tikzpicture}

\end{center}
 \caption{Fan of one-dimensional permutahedron $\rightsquigarrow$ reduced fan }
        \label{fig:fanperm-1}
\end{figure}
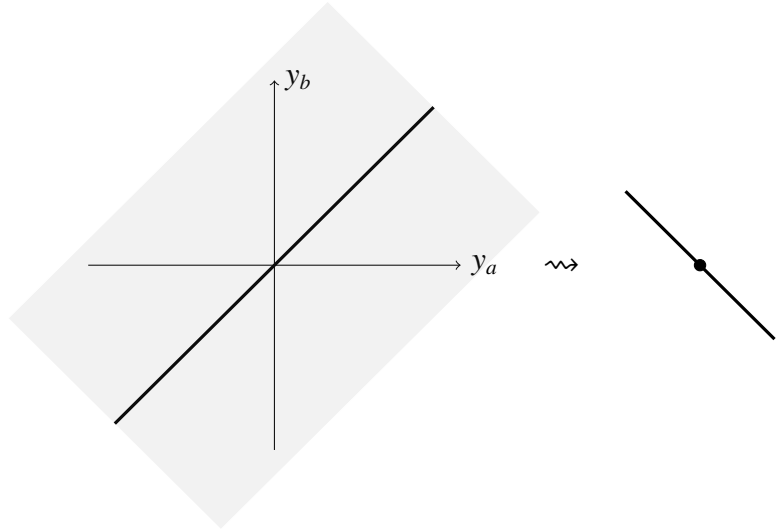

Given fans $U$ and $V$ of $\RR^I$, the fan $U$ is a {\it coarsening} of $V$, if
\begin{itemize}
  \item Each cone of $V$ is contained in a cone of $U$
 \item  Each cone of $U$ is a union of cones of $V$.
 \end{itemize}

 Given an extended generalized permutahedron defined by a submodular
 function $z$, it is in the hyperplane $\sum_{i \in I} x_i = z(I)$.
 Each of its face cones $k(F)$ in its normal fan
 then contains the line $\RR \fone$.
 We may therefore consider its normal fan in $\RR^I/\RR \fone$.

 The generalized permutahedra are precisely the polytopes whose
 normal fan is a coarsening of the braid fan, \cite[Thm.12.5]{AA2017}.
 Polyhedra whose normal fan is a coarsening of a {\it subfan} of the braid
 fan coincide with the class of extended generalized polyhedra,
\cite[Thm.12.7]{AA2017}.
 (But note that a coarsening of the braid fan, may not
 be the normal fan of generalized permutahedron \cite[Section 3]{MW2009}.) 

 \begin{remark} \label{rem:braid-coarse}
   In \cite{MW2009} they define a convex rank test to 
be a coarsening of the braid fan. When the coarsening is associated
to a submodular function, they call this a submodular rank test. Not
all coarsenings of the braid fan are submodular rank tests. 
They show
coarsenings are equivalent to semi-graphoids on $I$, which is an abstraction
of Conditional Independence Statements for random variables. 
In \cite{PRW2009} such a coarsening corresponds to
a {\it complete fan of posets}, since
each braid cone of full dimension $|I|$ corresponds to a poset
(see also Remark \ref{rem:bij-linext}).
%A complete
%fan of posets is defined intrinsically as a collection of posets which
%intersect properly (see Remark \ref{rem:poset-intprop} and whose
%linear extensions disjointly make up all linear orders.
\end{remark}

 \begin{example}
   Figure \ref{fig:perm-1} gives the one-dimensional permutahedron.
   It comes from the submodular function $z : \pow\{a,b\} \pil \RR$
   where $z(\{a\}) = 2 = z(\{b\})$ and $z(\{a,b\}) = 3$.
   Figure \ref{fig:fanperm-1} gives its normal fan, which has two
   maximal face cones.
 \end{example}

 \begin{example}
   Figure \ref{fig:low-1} gives the extended generalized permutahedron
   of $\low_P: \pow\{a,b\} \pil \RR$ where $P$ is the poset given.
   This EGP is defined by the equations $x_a + x_b = 0$ and $x_a \leq 0$.
   Figure \ref{fig:fanlow-1} gives its normal fan, which har two cones,
   with dimensions two and one.
   \end{example}

%%%%%%%%%%%%%%%%%%%%%%%%%%%%%%
%% Her er en-dim til poset 1-2

\begin{figure}
\begin{center}
%Første bilde
\begin{tikzpicture}[scale=1.2, vertices/.style={draw, fill=black, circle,
inner sep=1.5pt}]

\draw [help lines, white] (-2,-1) grid (3,2);
 %Make axes:
        \draw[black, ->] (-2,0) -- (3.5,0) node [anchor=west] {$x_a$};
        \draw[black, ->] (0,-2) -- (0,3.5) node [anchor=west] {$x_b$};
      %%\draw[black, thick,->] (0,0,0) -- (0,0,3.5) node [anchor=west] {z};

\coordinate (a) at (0,0);
\coordinate (b) at (1,2);
\coordinate (i) at (-4,4);

\node[vertices] at (a) {}; 
%\node[vertices] at (b) {};

\draw[black, very thick] (a)--(i);

\coordinate (x) at (-5, -0.5);
\coordinate (y) at (-5, 0.5);
\coordinate (p) at (-5.4, 0);

\node[vertices] at (x) {};
\node[vertices] at (y) {};
%\node[vertices] at (p) {};

\node[anchor = east] at (x) {$a$};
\node[anchor = east] at (y) {$b$};
\node[anchor = east] at (p) {$P$};

\draw[black, thick] (x)--(y);

\end{tikzpicture}
\end{center}
 \caption{The extended generalized permutahedron $\egp(\low_P)$}
        \label{fig:low-1}
\end{figure}
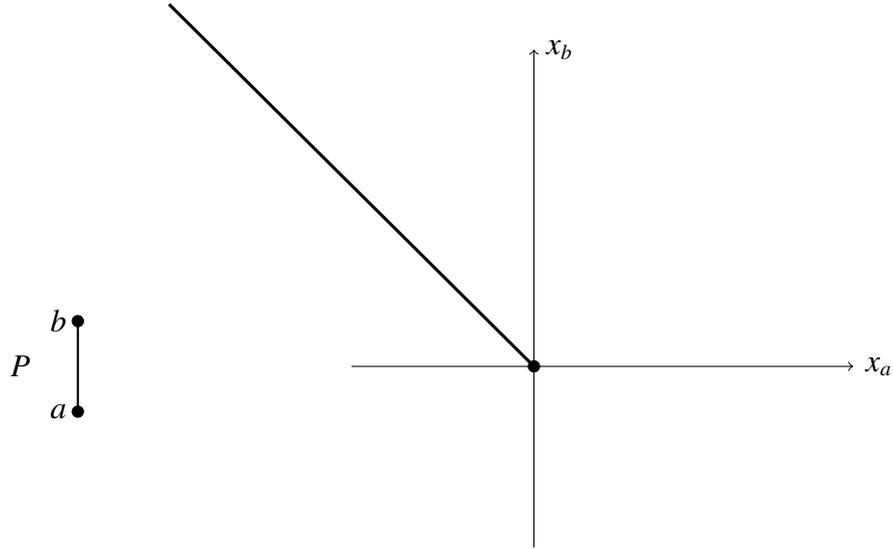

%%%%%%%%%%%%%%%
%% Her er viften til low_P 

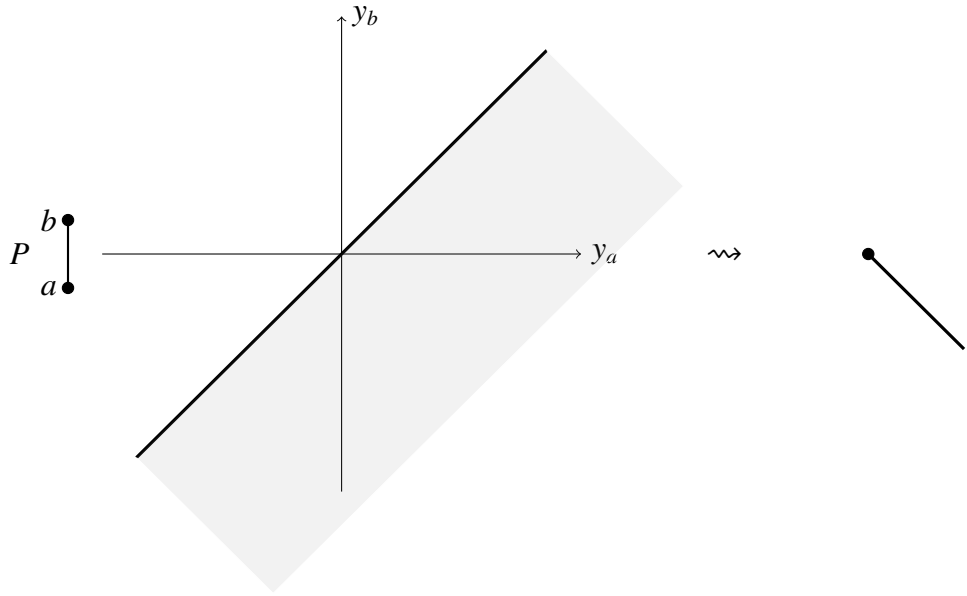
\begin{figure}
\begin{center}
%Første bilde
\begin{tikzpicture}[scale=0.9, vertices/.style={draw, fill=black, circle,
inner sep=1.5pt}]

\draw [help lines, white] (-2,-1) grid (3,2);
 %Make axes:
        \draw[black, ->] (-3.5,0) -- (3.5,0) node [anchor=west] {$y_a$};
        \draw[black, ->] (0,-3.5) -- (0,3.5) node [anchor=west] {$y_b$};
      %%\draw[black, thick,->] (0,0,0) -- (0,0,3.5) node [anchor=west] {z};

\coordinate (a) at (-3,-3);
\coordinate (b) at (3,3);
\coordinate (c) at (-5,-1);
\coordinate (d) at (1,5);
\coordinate (e) at (-1,-5);
\coordinate (f) at (5,1);

\draw[black, very thick] (a)--(b);

\draw[white, fill=gray, fill opacity = 0.1] (a)--(e)--(f)--(b) ;

\coordinate (x) at (-4, -0.5);
\coordinate (y) at (-4, 0.5);
\coordinate (p) at (-4.4, 0);

\node[vertices] at (x) {};
\node[vertices] at (y) {};
%\node[vertices] at (p) {};

\node[anchor = east] at (x) {$a$};
\node[anchor = east] at (y) {$b$};
\node[anchor = east] at (p) {$P$};

\draw[black, thick] (x)--(y);

%%%%%%%%%%%%%%%%%%%%%%%%
\coordinate (p) at (-1.4 + 7.0,0);
\node at (p) {$\rightsquigarrow$}; 
%\begin{tikzpicture}[scale=0.7, vertices/.style={draw, fill=black, circle,
%    inner sep=1.5pt}]

\coordinate (x) at (-1.4 + 7.7,1.4);
\coordinate (y) at (1.4+7.7,-1.4);
\coordinate (z) at (0+7.7,0);

\node[vertices] at (z) {}; 
\draw[black, very thick] (z)--(y);

\end{tikzpicture}
\end{center}
 \caption{Fan and reduced fan of $\egp(\low_P)$}
        \label{fig:fanlow-1}
      \end{figure}

 %%%%%%%%%%%%%%%%%%%%%%%%%%%%%%%%%%%%%%%%%%%%%%%%%%
 %% Sub- and quotient preorders

\section{Subdivisions and contractions of preorders}
\label{sec:poset}

We consider two significant order relations $\lhd$ and $\btl$ between
preorders. They occur a number of places in the literature
and under different terminologies, see Subsection \ref{subsec:poset-prop}.
They are fundamental order
relations for preorders, and we introduce them in a natural way from
a Galois correspondence
on the poset $\Pre(X)$. The order relation $\btl$ is of
significance for us here, as it corresponds to the inclusion of faces of
EGP's, Section \ref{sec:bij}.

Recall that for any set $X$, the set $\Pre(X)$  is partially ordered
by refinement of preorders. In fact, it is a lattice with meet
$\wedge$ and join $\vee$ given by:

\begin{itemize}
\item $x \,(\le_1\wedge\le_2) \,y$ if and only if $x\le_1 y$ and $x\le_2 y$,
\item $x \,(\le_1\vee \le_2) \,y$ if and only if there exist
$x_1,\ldots,x_{2k-1}\in X$ such that 
\[ x=x_0\le_1x_1\le_2x_2\le_1\cdots \le_1 x_{2k-1}\le_2 x_{2k}= y. \]
\end{itemize}
The unique maximal element in $\Pre(X)$
is the coarse preorder for which $x\le y$ for any $x,y\in X$, and the unique minimal element is the trivial (or discrete) preorder for which $x\le y$ implies $x=y$.
%The opposite of a preorder $\le$ is defined by $x\le^{\smop{op}}y$ if and only if $y\le x$.

%From now on, we will denote preorders on $X$ by capital letters $A,B,C,\ldots$, using the infix notation $x\le_A y,x\le_B y,x\le_C y$ for $x,y\in X$ whenever necessary.\\

%%%%%

%For any preorder $P\in \Pre(X)$ and for any $x,y\in X$, we will write $x\sim_P y$ for $x\le_{P\wedge P^{\ssmop{op}}} y$. The \textsl{connected component} of $x\in X$ is the set of all $y\in X$ such that $x\le_{P\vee P^{\ssmop{op}}} y$.

%\section{Two families of preorder sublattices}\label{sect:preorder}
%%%%%

\subsection{A Galois correspondence on $\Pre(X)$}
Fix a preorder $P$ on $X$, and let $\da P$
be the set of all preorders $R$ on $X$ such that $R\preceq P$,
and $\ua P$ the set of preorders $Q$ on $X$ such that $P \preceq Q$.
Let us consider the two order preserving
maps
\begin{equation} \label{eq:pos-FG}
  \da P \,  \bihom{F_P}{G_P} \, \ua P \, ,
\end{equation}
% F_P:\downarrow P\to\uparrow P$ and $ G_P:\overline P\to\underline P$ respectively defined by
given by
\[ F_P(R):=P\vee R^{\smop{op}}, \quad G_P(Q):=P\wedge Q^{\smop{op}}. \]

\begin{lemma}\label{adjunction}
For any three preorders $P,Q,R$ on $X$ such that $R\preceq P\preceq Q$, the following equivalence holds:
\begin{equation}
R\preceq P\wedge Q^{\smop{op}}\Longleftrightarrow P\vee R^{\smop{op}}\preceq Q.
\end{equation}
In other words,
\begin{equation}\label{adj-two}
R\preceq G_P(Q)\Longleftrightarrow F_P(R)\preceq Q.
\end{equation}
So \eqref{eq:pos-FG} is a {\it Galois correspondence} \cite[Paragraph 7.24]{DaPri2002}
between posets.
\end{lemma}

\begin{proof}
From $R\preceq P\wedge Q^{\smop{op}}$ we get $R^{\smop{op}}\preceq P^{\smop{op}}\wedge Q$. Considering the join with $P$ of both sides we get
$$P\vee R^{\smop{op}}\preceq P\vee(P^{\smop{op}}\wedge Q).$$
Since both factors on the right side are $\preceq Q$ we obtain
$$P\vee R^{\smop{op}} \preceq Q.$$
The reverse implication is proven similarly: from $P\vee R^{\smop{op}}\preceq Q$ we get $P^{\smop{op}} \vee R\leq Q^{\smop{op}}$. Considering the meet with $P$ of both sides we get
$$P\wedge(P^{\smop{op}}\vee R)\preceq P\wedge Q^{\smop{op}}.$$
Since both factors to the left are $\succeq R$ we obtain
$$R\preceq P\wedge Q^{\smop{op}}.$$
\end{proof}

%\begin{remark}\label{galois}\rm
%Considering the poset $\big(\mop{Pre}(X),\preceq\big)$ as a category (in which there is at most one morphism from an object to another), Lemma \ref{adjunction} can be reformulated as follows: $(F_P,G_P)$ is a \textsl{Galois correspondence}, i.e. a pair of adjoint functors between the two full subcategories $\underline P$ and $\overline P$, namely
%$$\mop{Hom}\big(R,G_P(Q)\big)=\mop{Hom}\big(F_P(R),Q\big).$$
%\end{remark}

\noindent {\bf Notations:}
For $R$ in $\im G_P \sus \, \da P $
we write $R \lhd P$ and say $R$ is a {\it subdivision} of $P$. By the
properties of a Galois correspondence \cite[7.27]{DaPri2002} we have $R \lhd P$ iff
\[ R=G_PF_P(R)=P\wedge(P^{\smop{op}}\vee R). \]

For $Q$ in $\im F_P \sus \, \ua P$ we write $P \btl Q$ and say
$Q$ is a {\it contraction} of $P$. Again by the
properties of a Galois correspondence we have $P \btl Q$ iff
\[ Q=F_PG_P(Q)=P\vee(P^{\smop{op}}\wedge Q). \]

We denote
$$\underline P= \im  G_P = \{R\in\mop{Pre}(X),\, R\lhd P\},\hskip 12mm
\overline P= \im F_P = \{Q\in\mop{Pre}(X),\, P\blacktriangleleft Q\}.$$

\begin{corollary} \label{cor:pos-bij}
  The maps $F_P$ and $G_P$ restrict to inverse order preserving bijections
\begin{equation} \label{eq:pos-resFG}
  \underline P \,  \bihom{F_P}{G_P} \, \overline P,
\end{equation}
where $R \lhd P$ and $P \btl Q$ correspond if both the following hold:
\[ Q = F_P(R) = P \vee R^\op , \quad R = G_P(Q) = P \wedge Q^{\op}. \]
\end{corollary}

\begin{proof}
  It is a standard fact for Galois correspondences between posets, that
  the maps $F_P$ and $G_P$ give inverse bijections when we restrict
  each of them to the image of the other.
\end{proof}

%This equivalence relation $\sim^b$ from Definition \ref{def:pre-bc}
%giving the bubbles ${\mathbf b}(P)$ is precisely the preorder $P \wedge P^\op$.
The unique minimal element
in $\underline P$ is $P \wedge P^\op$ which is the bubble equivalence
relation $\mathbf b(P)$.
The unique maximal element of  $\overline P$ is $P \vee P^\op$
which is the component equivalence
relation $\mathbf c(P)$.

\begin{remark}
  By Corollary \ref{cor:poset-pc} below it could have seemed natural to
  call $R$ a partition of $P$ when $R \lhd P$. But if $P$ is a set,
  considered as a poset with the discrete preorder, the only $R$  with
  $R \lhd P$ is $R = P$, since each part of $R$ has to be connected.
 We therefore use ``subdivision'' to distinguish it from ``partition''.
 The terminology ``contraction'' for $P \btl Q$ is taken from
 \cite[Section 3]{PRW2009}.
\end{remark}

%%%%%%%%%%%%%%%%%
%% Egen figur

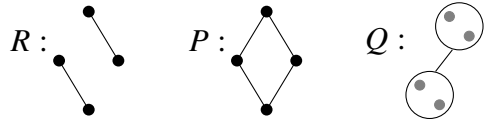
\begin{figure}
\begin{center}
%Første bilde
\begin{tikzpicture}[inner sep=1pt, scale=1.3]
\draw [help lines, white] (-2,-1) grid (3,2);

%% Poset R

%u = 1.8
\node (rr) at (-1.9 , 0.5) {$R:$};
\node (aa2) at (-1.3,0.8) {};
\node (cc2) at (-1.6, 0.3) {};
\node (bb2) at (-1,0.3) {};
\node (dd2) at (-1.3, -0.2) {};

\fill (aa2) circle (0.06);
\fill (cc2) circle (0.06);
\fill (bb2) circle (0.06);
\fill (dd2) circle (0.06);

%\draw (aa2)--(cc2); %--(b)--(d);
\draw (aa2)--(bb2);
%\draw (bb2)--(dd2);
\draw (cc2)--(dd2);

%Poset P

\node (pp) at (-1.9+1.8, 0.5) {$P:$};
\node (aa) at (-1.3+1.8,0.8) {};
\node (cc) at (-1.6+1.8, 0.3) {};
\node (bb) at (-1+1.8,0.3) {};
\node (dd) at (-1.3+1.8, -0.2) {};

\fill (aa) circle (0.06);
\fill (cc) circle (0.06);
\fill (bb) circle (0.06);
\fill (dd) circle (0.06);

\draw (aa)--(cc); %--(b)--(d);
\draw (aa)--(bb);
\draw (bb)--(dd);
\draw (cc)--(dd);

%% Poset Q

%u = 3.6
\node (qq) at (-1.9 +3.6, 0.5) {$Q:$};
\node (aa3) at (-1.25+3.6,0.75) {};
\node (cc3) at (-1.55+3.6, 0.05) {};
\node (bb3) at (-1.05+3.6,0.55) {};
\node (dd3) at (-1.35+3.6, -0.15) {};

\filldraw[gray,thin] (aa3) circle (0.05);
\filldraw[gray,thin] (cc3) circle (0.05);
\filldraw[gray,thin] (bb3) circle (0.05);
\filldraw[gray,thin] (dd3) circle (0.05);

\node (ss3) at (-1.15+3.6,0.65) {};
\draw (ss3) circle (0.24);
\node (tt3) at (-1.45+3.6,-0.05) {};
\draw (tt3) circle (0.24);

%\fill (aa2) circle (0.06);
%\fill (cc2) circle (0.06);
%\fill (bb2) circle (0.06);
%\fill (dd2) circle (0.06);

\draw (-1.22+3.6,0.41)--(-1.39+3.6, 0.19);
%\draw (aa2)--(cc2); %--(b)--(d);
%\draw (aa2)--(bb2);
%\draw (bb2)--(dd2);
%\draw (cc2)--(dd2);

\end{tikzpicture}

%%Slutt på de to bildene på samme nivå
%%%%%%%%%%%%%%%%%%%%%%%%%%%%%%%%%%%%%
%%

\caption{$R \lhd P$
  and corresponding $P \btl Q$}
\label{fig:RPQ}
\end{center}
\end{figure}

%%%%%%%%%%%%%%%%%%%%%%%%%%%%
%%% Figur P ikke btl Q

\begin{figure}
\begin{center}
%Første bilde
\begin{tikzpicture}[inner sep=1pt, scale=1.3]
\draw [help lines, white] (-2,-1) grid (3,2);

%Poset P

\node (pp) at (-1.9, 0.5) {$P:$};
\node (aa) at (-1.3,0.8) {};
\node (cc) at (-1.6, 0.3) {};
\node (bb) at (-1,0.3) {};
\node (dd) at (-1.3, -0.2) {};

\fill (aa) circle (0.06);
\fill (cc) circle (0.06);
\fill (bb) circle (0.06);
\fill (dd) circle (0.06);

\draw (aa)--(cc); %--(b)--(d);
\draw (aa)--(bb);
\draw (bb)--(dd);
\draw (cc)--(dd);

%% Poset Q

%u = 1.6
\node (qq) at (-1.9 +1.8, 0.5) {$Q:$};
\node (aa3) at (-1.3+1.8,0.75) {};
\node (cc3) at (-1.4+1.8, 0.3) {};
\node (bb3) at (-1.2+1.8,0.3) {};
\node (dd3) at (-1.3+1.8, -0.15) {};

\fill (aa3) circle (0.05);
\filldraw[gray,thin] (cc3) circle (0.05);
\filldraw[gray,thin] (bb3) circle (0.05);
\fill (dd3) circle (0.05);

\node (ss3) at (-1.3+1.8,0.3) {};
\draw (ss3) circle (0.18);
%\node (tt3) at (-1.45+3.6,-0.05) {};
%\draw (tt3) circle (0.24);

%\fill (aa2) circle (0.06);
%\fill (cc2) circle (0.06);
%\fill (bb2) circle (0.06);
%\fill (dd2) circle (0.06);

\draw (aa3)--(-1.3+1.8, 0.5);
\draw (-1.3+1.8,0.1)--(dd3);
%\draw (aa2)--(cc2); %--(b)--(d);
%\draw (aa2)--(bb2);
%\draw (bb2)--(dd2);
%\draw (cc2)--(dd2);

\end{tikzpicture}

%%Slutt på de to bildene på samme nivå
%%%%%%%%%%%%%%%%%%%%%%%%%%%%%%%%%%%%%
%%

\caption{$P \preceq Q$ but not 
  $P \btl Q$}
\label{Fig-TT}
\end{center}
\end{figure}

\subsection{Properties of the relations $\lhd$ and $\btl$}
\label{subsec:poset-prop}

The relations $\lhd$ and $\btl$ occur a number of places in the
  literature.
R. Stanley in \cite{Sta1986} introduces the order polytope of a
  poset $P$. The faces of this polytope are in bijection
  with posets $R \lhd P$. Such $R$ are
  called {\it partitions} of $P$ which are {\it connected and compatible}.
  Stanley
  attributes the description of the face structure to Geissinger
  \cite{Geiss1981}.

  Postnikov, Reiner and Williams \cite[Sec.3]{PRW2009} when $P \btl Q$ call
  $Q$ a {\it contraction} of $P$. They define $Q$ the same way as we do, those
  that can be written as $P \wedge R^\op$ for a sub-preorder
  of $P$.

  Ardila and Aguiar in \cite[Sec.15]{AA2017} call $R \lhd P$ a
  {\it positive subposet}
  (they only consider posets) a terminology borrowed from oriented matroids,
  see Definition \ref{def:pos-pos} below.
%Their definition
%is part a in Proposition \ref{pro:poset-OTL}.

Fauvet, Foissy, and Manchon \cite{FFM2017} when $R \lhd P$ call $R$
{\it admissible} (they work with topologies rather than preorders) and
write $R \oprec P$, see Definition \ref{def:pos-adm}. The
corresponding $Q$ with $P \btl Q$ is called a {\it quotient} of $P$
and written $P/R$.

\begin{definition} \label{def:pos-adm}
By \cite[Definition 2.2]{FFM2017} for $R \preceq P$ the relation
$R\oprec P$ is defined by:
\begin{itemize}
\item The set of bubbles of $P \vee R^\op$ and $R \vee R^\op$ are the
  same. (Note that $R \vee R^\op$ is a totally disconnected preorder whose
  bubbles or equivalence classes are the connected components of $R$.)
%For any $x,y\in X$, $x\sim^b_{P\vee R^{\ssmop{op}}} y\Leftrightarrow x\le_{R\vee R^{\ssmop{op}}} y$,
\item The restrictions of $R$ and $P$ to the connected components of $R$ coincide.
\end{itemize}
\end{definition}

\begin{definition} \label{def:pos-pos}
  From \cite[Sec.15.1]{AA2017} {and \cite[Prop. 9.1.2]{BLSWZ1993}},
an $R \preceq P$
  (there done only for posets) is {\it a positive subpreorder} if
 for any loop
\begin{equation}\label{eq:poset-loop}
x=x_0\le_R x_1\ge_P x_2\le_R\cdots\le_R x_{2i-1}\ge_P x_{2i}\le_R\cdots \le_R x_{2k-1}\ge_P x_{2k}=x,
\end{equation}
the descending steps $\ge_P$ are actually $\ge_R$.
\end{definition}

\begin{proposition}\label{pro:poset-OTL}
  Given preorders $R\preceq P$ on $X$. The following are equivalent:
  \begin{itemize}
  \item[a.]$R$ is a positive sub-preorder of $P$,
\item[b.]  $R\oprec P$,
\item[c.] $R \lhd P$
 \end{itemize}
\end{proposition}

\begin{proof}  Assume a. 
So 
the descending steps $\ge_P$ are actually $\ge_R$ in any
loop \eqref{eq:poset-loop}. We have $x\sim^b_{P\vee R^{\ssmop{op}}} y$ if and only
if $y$ is part of such a loop, which therefore amounts to
$x\sim^b_{R\vee R^{\ssmop{op}}} y$. Now if $x\le_P y$ are in the same $R$-connected
component we have a loop
$$x\le_R x_1\ge_Rx_2\le_R\cdots \le_R y\ge_P x$$
in which the last descending step is actually $\ge_R$, hence $x\le_R y$.
So part a implies part b.
Conversely assume b.
So the bubbles of $P \vee R^\op$ and of $P^\op \vee R$ are the same.
%the preorder $R\vee R^{\smop{op}}$ being an equivalence relation, the first condition above can be rewritten as
%$$x\sim_{P^{\ssmop{op}}\vee R} y\Leftrightarrow x\sim_{R\vee R^{\ssmop{op}}} y$$
%for any $x,y\in X$. This means that
For any loop of the form \eqref{eq:poset-loop} (note that
all order relations are in $P^\op \vee R$)
and for any $\ell=0,\ldots,2k$, there exists another loop
\begin{equation}\label{loop-bis}
x=x_0\le_R x'_1\ge_R x'_2\le_R\cdots\le_R x'_{2j-1}\ge_R x'_{2j}=x_\ell\le_R\cdots \le_R x'_{2r-1}\ge_R x'_{2r}=x,
\end{equation}
in which all descending steps are $\ge_R$. Hence $x_\ell$ belongs to the $R$-connected component of $X$. The loop \eqref{eq:poset-loop} is therefore entirely contained in the $R$-connected component of $x$. The second condition then implies that the downsteps $\ge_P$ in \eqref{eq:poset-loop} are actually $\ge_R$.

\medskip
\noindent Equivalence of a and c:
Suppose $R\lhd P$. Given any loop in the form \eqref{eq:poset-loop} and any $i=1,\ldots, k$ we have both $x_{2i-1}\geq_{P}x_{2i}$ and
(by going around the full loop) $x_{2i}\le_{P^{\ssmop{op}}\vee R} x_{2i-1}$.
Hence $x_{2i-1}\ge_{P\wedge(P^{\ssmop{op}}\vee R)}x_{2i}$. But then since
$R=P\wedge(P^{\smop{op}}\vee R)$ we get $x_{2i-1}\ge_R x_{2i}$.
%Now let $x,y$ in the same connected component of $R$. We must show that
%$x \leq_P y$ implies $x \leq_R y$. But being in the same component we
%have $x \leq_{R^\op \vee R} y$. Since $R^\op \vee R \leq P^\op \vee R$
%we get that $x \leq y$ for the relation $P \wedge (P^\op \vee R) = R$.

\medskip

Conversely, assume $R$ is a positive sub-preorder.
Let $x\le_{P\wedge(P^{\ssmop{op}}\vee R)} y$, so $x\le_P y$ and $x\le_{P^{\ssmop{op}}\vee R} y$. So there exists a loop:
$$x\le_R x_1\ge_P x_2\le_R\cdots \ge_P x_{2k}=y\ge_P x.$$
Since $R$ is positive, each downstep is $\ge_R$.
In particular the last one is, from which we get $x\le_R y$, and
so $R = P \wedge (P^\op \vee R)$. 
\end{proof}

\begin{corollary} \label{cor:poset-pc}
    Let $R \lhd P$ correspond to $P \btl Q$. Then for each bubble
    $C$ of $Q$, the restriction $P\restr{C} = R\restr{C}$, and this
    restriction is connected. Also,  
    $R$ is the disconnected union $\sqcup_{C \in b(Q)} R\restr{C}$.
  \end{corollary}

  Here is another characterization of the relations $\lhd$ and $\btl$,
  which often simplifies arguments.

\begin{proposition} \label{pro:poset-cc}
  Let $R \preceq P$ and $P \preceq Q$.
  \begin{itemize}
  \item[a.] $P \btl Q$ iff for every connected convex subset
    $K \sus Q$, the restriction $P_{|K}$ is connected convex in $P$.
  \item[b.] $R \lhd P$ iff for every connected convex subset $K \sus R$,
    the restriction $P_{|K}$ is connected convex in $P$.
  \end{itemize}
  \end{proposition}

  \begin{proof} Note that in part a, $K$ is automatically convex in $P$,
    and in part b, $K$ is automatically connected in $P$. 

    \noindent    a. Assume $P \btl Q$. Then $Q  = P \vee R^\op$ for some $R$.
    Let $x \leq_Q y$ both be in $K$.
    Then
    \[ x = x_0 \leq_P x_1 \geq_R x_2 \leq_P \cdots \geq_R x_{2k} = y.\]
    Since $K$ is convex in $Q$ all $x_i \in K$. Thus $P_{|K}$ is connected.

    Assume the condition on convex connected subsets.
    Let $R^\op := P^\op \wedge Q$. We have $Q \succeq P \vee
    R^\op$ and will show equality. Let  $x \leq_Q y$ and let
    the convex set $K := [x,y]_Q = \{z \in Q \, | \, x \leq z \leq y \}$.
We show that $x \leq y$ for $\leq$ the preorder $P \vee (P^\op \wedge Q)$.
%   We show by induction on the number of $Q$-bubbles in $K$ that $x \leq y$
%   for $\leq$ the preorder $P \vee (P^\op \wedge Q)$.
    Assume $K$ is a single $Q$-bubble.
    Since $P_{|K}$ is convex connected we have a sequence:
    \[ x = x_0 \leq_P x_1 \geq_P x_2 \leq_P \cdots \geq_P x_{2k} = y.\]
%    We argue by induction on $k$.
% and the number of changes of $Q$-bubbles of the $x_j$ as $j$ increases.
Assume we are all the time in the same $Q$-bubble. Then 
    $x_{2i-1} \leq x_{2i}$ for $\leq$ the preorder $P^\op \wedge Q$.
    This shows that $x \leq y$ for $P \vee (P^\op \wedge Q)$.

    Now assume not all $x_j$ are in the same $Q$-bubble.
    We argue by induction on $k$.
    Let $x_j$ be maximal not in the same bubble as $y$. If $j$ is odd
    i)  $x_j = x_{2i-1} \geq_P x_{2i}$ and ii) $x_{2i-1} \leq_Q y$ since
    $x_{2i-1} \in K$. Since $x_{2i}$ and $y$ are in the same bubble,
    we get $x_j$ in this bubble, a contradiction.
    Thus $j$ is even and $x_{j} = x_{2i} <_Q x_{2i+1}$ is a strict order
    relation. In the following $\leq$ is $P \vee (P^\op \wedge Q)$. 
    Since $x_{2i} \leq_P x_{2i+1}$ and $x_{2i+1} \leq y$ (since they
    are in the same $Q$-bubble, a case considered above),
    we have $x_{2i} \leq y$.
    Also  $x \leq_Q x_{2i}$ and by induction
    $x \leq x_{2i}$ for $\leq$ the preorder
    $P \vee (P^\op \wedge Q)$ and so $x \leq y$ for this preorder.
    This shows $Q = P \vee (P^\op \wedge Q)$.

    \medskip
    \noindent  b. Assume $R \lhd P$. Since $K$ is connected, it is
    contained in a connected component of $C$ of $R$. Let now
    $x \leq_P y$ in $K$. They are both in $C$ and as $R_{|C}= P_{|C}$, we have
    $x \leq_R y$. Suppose now $x \leq_P z \leq_P y$. We must show $z \in K$. 

%    Consider the convexity of $K$.
%    Let $x, y \in K$ and let $x \leq_P z \leq_P y$.
    Consider the cycle
   \[ z \leq_R z \geq_P x \leq_R y \geq_P z.\]
    By Proposition \ref{pro:poset-OTL} $x \leq_R z$ and
    $z \leq_R y$, and so $z \in K$.

Assume the condition on connected convex subsets. Consider a cycle:
    \begin{equation} \label{eq:poset-xxx}
      x = x_0 \leq_R x_1 \geq_P x_2 \leq_R \cdots \geq_P x_{2k} = x.
      \end{equation}
    For each $i$ both $x_{2i}$ and $x_{2i+1}$ are in the same connected component
    $K_i$ of $R$. This is a convex subset of $P$ and let $D_i$ be
    the downset of $P$ generated by $K_i$, and $D_i^\prime = D_i \backslash K_i$,
    also a downset of $P$. Since $x_{2i+1} \geq_P x_{2i+2}$, each
    $K_{i+1} \sus D_i$. If $K_{i+1} \neq K_i$ then $K_{i+1} \sus D_i^\prime$
    and $D_{i+1} \sus D_i^\prime$. Since we start with the connected
    component of $x$ and end with it, the upshot is that all $K_i$ must
    be equal.

    So assume $x,y$ in the same connected component of $R$ and
    $x \leq_P y$. We show $x \leq_R y$.
    This will show each $x_{2i-1} \geq_P x_{2i}$ in \eqref{eq:poset-xxx}
    is also $x_{2i-1} \geq_R x_{2i}$ and so $R \lhd P$. 

% Proposition \ref{pro:poset-OTL}
%    to \eqref{eq:poset-xxx}
%    this gives $R \lhd Q$.
    We will have a sequence:
    \begin{equation} 
      x = x_0 \leq_R x_1 \geq_R x_2 \leq_R \cdots \geq_R x_{2k} = y.
    \end{equation}
    We use induction on $k$. If $k = 0$ it is OK. Suppose $k \geq 1$.
    If $x_{2k-1} \leq_R y$, by induction  $x \leq_R x_{2k-2}$ and so
    $x \leq_R y$. So assume we do not have $x_{2k-1} \leq_R y$. 
    Let $D$ be the downset of $R$ generated by $x_1, x_3, \ldots, x_{2k-1}$.
    It contains $y$.
    Let $\downarrow y$ be the downset of $R$ generated by $y$, and
    $K = D \backslash \downarrow y$, a connected convex subset of $R$.
    If not $x \leq_R y$, then $x \in K$ and $P_{|K}$ is convex and contains
    $x$ and $x_{2k-1}$. Since $x \leq_P y \leq_P x_{2k-1}$ we get $y \in K$,
    a contradiction. Thus we have $x \leq_R y$.
  \end{proof}

The following will be useful later:
\begin{corollary}  \label{cor:pos-PbtrQ}
  $P \btl Q$ if and only if $P \preceq Q$ and: 
  \begin{itemize}
  \item[1.] For every bubble $B$ in $Q$, the restriction $P\restr{B}$
    is a connected preorder.
  \item[2.] If $B_1 <_Q B_2$ is a cover relation of bubbles in $Q$, there exists
    $p_1 \in B_1$ and $p_2 \in B_2$ such that $p_1 <_P p_2$.
  \end{itemize}
\end{corollary}

\begin{proof}
  Clearly $P \btl Q$ implies 1 and 2. Conversely if 1 and 2 hold,
  then $P_{|K}$ will be convex connected whenever $K$ is so in $Q$.
\end{proof}

\begin{corollary} \label{cor:poset-bubPQ}

If $P \btl Q$ and the bubbles 
${\mathbf b}(P) ={\mathbf b}(Q)$, then $P = Q$.
\end{corollary}

\begin{proof}
  If $B_1 <_Q B_2$ is a cover relation, then $B_1$ and $B_2$ are
  also bubbles of $P$ and we have $B_1 <_P B_2$.
  \end{proof}
 
\subsection{The lattices $\underline{P}$ and $\overline{P}$}

\begin{theorem}\label{distrib-lattice} 
  The relations $\lhd$ and $\btl$ on $\Pre(X)$ are partial orders
  (we have usually written a partial order as $\leq$, i.e. $<$ with
  a dash below, but here the dash is omitted).
  \begin{itemize}
  \item[a.] $(\underline P, \lhd)$ and $(\overline P, \btl)$ are
    partially ordered sets induced from $(\Pre(X), \preceq)$.
  \item[b.] The preorders $P$ and ${\mathbf b}(P) = P \wedge P^{\op}$ are
    the unique maximal and minimal elements of $(\underline P, \lhd)$.
    The preorders ${\mathbf c}(P) = P \vee P^{\op}$ and $P$ are the
    unique maximal and minimal elements of  $(\overline P, \btl)$
  \item[c.] When $X$ is finite $(\underline P, \lhd)$ is a lattice
    with meet $\wedge$. The join of $R_1$ and $R_2$ is
    \[ F_P \circ (G_P(R_1) \vee G_P(R_2)) = P \wedge (P^\op \vee R_1 \vee R_2). \]
   \item[d.] When $X$ is finite
    $(\overline P, \btl)$ is a lattice with 
    join $\vee$. The meet of $Q_1$ and $Q_2$ is
    \[ G_P \circ (F_P(Q_1) \wedge F_P(Q_2)) = P \vee (P^\op \wedge Q_1 \wedge Q_2). \]
   \end{itemize}
\end{theorem}
%  \textcolor{blue}{d. should be c. and e. should be d.}
%  \begin{itemize}
%    \item[i.] $P$ is maximal in 
%For any preorder $P$ on a set $X$ 
%\begin{itemize}
%\item the quadruple $(\underline P, \lhd,\wedge,\vee)$ is a sublattice of $(\mop{Pre}(X),\preceq,\wedge,\vee)$ with unique maximum $\mathbf 1_{\underline P}=P$ and with unique minimum $\mathbf 0_{\underline P}=P\wedge P^{\smop{op}}$,
%\item the quadruple $(\overline P, \blacktriangleleft,\wedge,\vee)$ is a sublatice of $(\mop{Pre}(X),\preceq,\wedge,\vee)$ with unique minimum $\mathbf 0_{\overline P}=P$ and with unique maximum $\mathbf 1_{\overline P}=P\vee P^{\smop{op}}$,
%\item $F_P:\underline P\to\overline P$ is a lattice isomorphism with inverse $G_P$.
%\end{itemize}
%\end{theorem}

\begin{proof}
  Transitivity of $\lhd$ and $\btl$ follow by Proposition
  \ref{pro:poset-cc}.

  a. Let  $S \preceq R$ with $S \lhd P$ and $R \lhd P$. If
  $K$ is a convex connected subset of $S$, it is connected in $R$.
  It is convex in $P$ and so also in $R$. Thus $S \lhd R$. 

  Let $ Q \preceq T$ with $P \btl Q$ and $P \btl T$. If $K$ is a convex
  connected subset of $T$, it is also convex in $Q$.
It is connected in $P$ and so also in $Q$. Thus $Q \btl T$.

  \medskip
  
  b. It is clear that $P$ is maximal in $\underline P$, and so by the
  isomorphism ${F_P}$, ${\mathbf c}(P)$ is maximal in $\overline{P}$.
 Similarly $P$ is minimal in $\overline P$ and so ${\mathbf b}(P)$ is minimal
  in $\underline{P}$.
  
  c. Let $R_1 \lhd P$ and $R_2 \lhd P$. Let $x \leq_{R_1 \wedge R_2} y$
  be in the same convex connected subset $K$ of $R_1 \wedge R_2$.
  Then they are in
  the same connected component $K_1$ of $R_1$, respectively $K_2$ of $R_2$.
  Since $R_{i,|K_i} = P_{|K_i}$, if $x \leq_P z \leq_P y$ then for $i = 1,2$
  we have
  $x \leq_{R_i} z \leq_{R_i} y$. Thus $x \leq_{R_1 \wedge R_2} z
  \leq_{R_1 \wedge R_2} y$ and so $z \in K$. Thus
  the image of the convex connected subset $K$ of
  $R_1 \wedge R_2$ is convex connected in $P$. Whence
  $R_1 \wedge R_2 \lhd P$. The join of $R_1$ and $R_2$ is computed
  by transferring the join on $\overline{P}$ to $\underline{P}$. 

%That $\vee$ is not the
%  join in the lattice can be seen from an example where $P$ is the four
%  element diamond poset. 

  {d.} Now let $P \btl Q_1$ and $P \btl Q_2$. Let $x \leq y$ be in a
  convex connected subset  $K$ of $Q_1 \vee Q_2$.
  It is clear that $P_{|K}$ is convex. We have
  \[ x = x_0 \leq_{Q_1} x_1 \leq_{Q_2} x_2 \leq \cdots \leq_{Q_2} x_{2k} = y.\]
  But then each pair $x_r, x_{r+1}$ is connected in $Q_1$ or in $Q_2$, and
  so they are connected in $P$. Thus $x$ and $y$ are connected in $P$.
   Whence $P \btl (Q_1 \vee Q_2)$.
  The meet of $Q_1$ and $Q_2$ is computed
  by transferring the meet on $\underline{P}$ to $\overline{P}$.
  
\end{proof}

%\begin{remark} \label{rem:poset-intprop}
%  For two preorders, if $Q \btl (Q \vee Q^\prime)$ and
%  $Q^\prime \btl (Q \vee Q^\prime)$, then $Q$ and $Q^\prime$ are said to inters%ect
%  {\it properly}.
%  See Remark \ref{rem:bij-int} for how this
%  relates to faces of EGP's.
%\end{remark}

%The criterion given in Proposition \ref{pro:poset-OTL} is the generalization
%  to preorders of the notion of \textsl{positive sub-poset} (\cite[Section 15]{AA2017}, \cite[Prop. 9.1.2]{BLSWZ1993}).
%It is unlikely that the lattices $\underline P$ and $\overline P$ are distributive.

%%%%%%%%%%%%%%%%%%%%%%%%%%%%%%%%%%%%%%%%%%%%
%% Conforming preorders

\section{Conforming preorders of submodular functions}
\label{sec:conform}

This section ``marries'' submodular functions and preorders.
We introduce the important notion of preorders compatible and conforming
to a submodular function $z$. For preorders $P, Q$ conforming
to $z$, we show that if $P \preceq Q$ they are related by the stronger
relation $P \btl Q$.
The geometric aspects  of this will be investigated in the next
Section \ref{sec:bij} where we show that preorders conforming
to $z$ are in bijection with the faces of the EGP $\egp(z)$.

\begin{definition}
  Given a submodular function $z : \pow(I) \to \hRR$, a preorder $P$ on $I$
  is {\sl compatible} with $z$ if:
\begin{itemize}
\item $z(A) < \infty$ for any downset $A$ of $P$,
\item For down-sets $B \supseteq A$ of $P$, let $C$ be the convex
  subset $B \backslash A$. If the induced preorder on $C$ decomposes as
  a disjoint union $C = C_1 \sqcup C_2$, then $z^\prime = z_{B / A}$ decomposes
  as a product $z^\prime = z^\prime\restr{C_1} \cdot z^\prime\restr{C_2}$
\end{itemize}
If the topology $\cT$ corresponds to $P$ we also call $\cT$ compatible with
$P$.
\end{definition}

\begin{remark} Note that if $P \preceq Q$ and $P$ {is compatible with} $z$, then
  $Q$ is compatible with $z$.
\end{remark}

\begin{remark} In \cite{MW2009}, if a semi-graphoid comes from
  a submodular rank test
  {associated with a submodular function $z$} (see Remark \ref{rem:braid-coarse}), then
for $i,j \in I$ and $K \sus I \backslash \{i,j\}$, 
the Conditional Independence Statement $i \sqcup j | K$
corresponds to the following: Let $M$ be {a poset structure on $I$ such that:
\begin{itemize}
\item $x\le_M i\le_M y$ and $x\le_M j\le_M y$ for any $x\in K$ and
  $y\in J := I \backslash
(K \cup \{i,j\})$,
\item $i$ and $j$ are incomparable,
\item The restrictions of $M$ to both $K$ and $J$ are total orders.
\end{itemize}
Then $i \sqcup j | K$ amounts to $M$ being compatible with $z$, and} $z_{/K}(\{i,j\}) = z_{/K}(i) + z_{/K}(j)$.
\end{remark}

\begin{proposition} \label{pro:com-BA}
  Suppose the preorder $P$ is compatible with the submodular function $z$.
  If $B \supseteq A$ are down-sets of $P$ and similarly
  $B^\prime \supseteq A^\prime$, such that
  $B \backslash A = B^\prime \backslash A^\prime=C$, then
  \[ z_{B/ A} = z_{B^\prime / A^\prime}. \]
  So this depends only on $C$ and not on the particular way of writing
  $C$ as a difference of down-sets. We therefore write $z_C$ for
  $z_{B /A}$. 
  \end{proposition}

\begin{proof}
  Let $D$ be the down-set in $P$ generated by $C = B \backslash
  A$. Then $E = D \backslash C$ is also a down-set
  in $P$. It is enough to show that $z_{B / A}
  = z_{D / E}$. Write $B$ as disjoint unions of underlying sets:
  \[ B = D \cup X = E \cup C \cup X. \]
  We claim that $B \backslash E$ is the disjoint union of the two
  preorders $X$ and $C$. To see this, if $x \in X$ and $x \leq c \in C$,
  then $x \in D$, which is impossible in view of $D\cap X=\emptyset$. If $x > c \in C$, then
  $x \not \in C$ and $x \in B$, and so $x \in A$. Since $A$ is a down-set
  we would then have $c \in A$, but $C$ is disjoint from
  $A$. Hence $x$ and $c$ are incomparable.\\
  
  \noindent Since $P$ is compatible with $z$ we get
  \[ z^\prime = z_{B /E} =
    z^\prime_C \cdot z^\prime_X. \]
  For $S \sus C$ we now have:
  \[ z^\prime (S \cup X) = z_{B / E} (S \cup X)
    = z^\prime_C(S) + z^\prime_X(X). \]
  Then
\[ z(E \cup S \cup X) -  z (E)
  = z(E \cup S) - z(E) + z(E \cup X)
  - z(E), \]
which gives:
\[
  z(E \cup S \cup X) - z(E \cup X)
  = z(E \cup S) - z(E).
\]
As $A = E \cup X$ this is
\[ z(A \cup S ) - z(A) = z (E \cup S)
  - z(E), \]
that is:
\[ z_{B / A}(S) = z_{D / E}(S). \]
\end{proof}

\begin{definition} Let $P$ be a preorder compatible with a submodular function
  $z$. Let $\alinb(P)$ be the affine linear space defined by the
  equations $x_C = z_C(C)$ as $C$ ranges over the bubbles of $P$.
\end{definition}

The following is essentially \cite[Lemma 3.33]{Fuj2005}.
\begin{lemma} \label{lem:conf-alinbu}
  Let $P$ be compatible with $z$, and $L$ a linear extension
  of $P$. Then we have
    \[ \alin(P) = \alinb(P) = \alin(L),\]
    where $\alin(P)$ and $\alin(L)$ are given by Definition \ref{def:alin}.
\end{lemma}

\begin{proof}
  We first note that
\[ \alin(P) \sus \alinb(P) \sus \alin(L). \]
Indeed, if $x$ is in $\alin(P)$, then $x_A = z(A)$ for any downset in $P$.
So if $C = B \backslash A$ is a bubble with $A,B$ downsets, then
\[ x_C = x_B - x_A = z(B) - z(A) = z_C(C), \] and $x$ is in $\alinb(P)$.
Let $x$ in $\alinb(P)$. Then given a down-set $B$ of $L$, let $C$
be the maximal bubble of $B$ and $C = B \backslash A$. By induction
on the maximal element, we may assume $x_A = z(A)$. Note that
both $B$ and $A$ are down-sets of $P$.
Then
\[ x_B = x_C + x_A = z_C(C) + z(A) = z(B) \]
and so $x$ is in $\alin(L)$.

Now let us show $\alin(L) \sus \alin(P)$. We show by induction that
the defining linear form $\fone_B - z(B)$ (see Section \ref{sec:egpclass} for
notation $\fone_B$) of $\alin(P)$ is a linear combination
of the defining forms for $\alin(L)$.\\

Let $B$ be a down-set of $P$,
and let $C$ be the bubble of $B$, maximal in $L$.
Write $C = B^\prime \backslash A^\prime$ as a difference between down-sets
of $L$. These are then also down-sets for $P$.
Then
\[ z_C(C) = z(B) - z(A) = z(B^\prime) - z(A^\prime). \]
By induction $\fone_{A} - z(A)$ is a form which is a linear
combination of defining forms for $\alin(L)$.
Now we get
\begin{align*}  \fone_B - z(B) & = (\fone_A - z(A)) + (\fone_C - z_C(C)) \\
& = (\fone_A - z(A)) + (\fone_{B^\prime} - z(B^\prime)) - (\fone_{A^\prime} -
z(A^\prime)).
\end{align*}
Thus $\fone_B - z(B)$ is a linear combination of defining forms for $\alin(L)$
and so $\alin(L) \sus \alin(P)$.
\end{proof}

\begin{definition}
  Let $z : \pow(I) \pil \hRR$ be a submodular function. A preorder
  $P$ {\it conforms} to $z$ if:
  \begin{itemize}
  \item $P$ is compatible with $z$,
  \item For any convex subset $C$ of $P$, if $z^\prime = z_C$
    decomposes as a product
    \[ z^\prime = z^\prime\restr{C_1} \cdot z^\prime\restr{C_2}, \]
    then the preorder on $C$ decomposes as a disconnected union of preorders
    $C = C_1 \sqcup C_2$.
  \end{itemize}
  By definition of compatible, that $P$ conforms to $z$ means:
  \begin{itemize}
    \item $z(A) < \infty$ for any downset $A$ of $P$,
    \item For $C$ convex in  $P$, $z^\prime = z_C$
      decomposes as a product $z^\prime = z^\prime\restr{C_1} \cdot
      z^\prime\restr{C_2}$  if and only if the preorder on $C$ decomposes
      as a disconnected union of preorders $C = C_1 \sqcup C_2$.
  \end{itemize}
  If the topology $\cT$ corresponds to $P$ we also say $\cT$ conforms to $z$.
 Denote by $\Pcon(z)$ the set of preorders conforming to $z$, and by
 $\Top(z)$ the topologies conforming to $z$. 
\end{definition}

\begin{lemma}
  $P$ (resp. $Q$) conforms to the submodular function $u$ (resp. $v$) if
  and only if 
  the disjoint union $P \sqcup Q$ conforms to $u \cdot v$.
\end{lemma}

\begin{proof} Clear.
\end{proof}

\begin{remark}
  If $z$ is indecomposable, the coarse preorder conforms to $z$.
  More generally, by Proposition \ref{pro:egp-top} the totally disconnected
  preorder $\cpre(z)$ conforms to $z$.
\end{remark}

\begin{proposition} \label{pro:conf-PbtrQ}
  Let $P$ conform to $z$ and $P \preceq Q$. Then $Q$ conforms to $z$
  if and only if $P \btl Q$. Equivalently, the up-set 
  $\ua_{\Pre(z)} P$ coincides with $ \overline P$. 
\end{proposition}

\begin{proof} Assume $Q$ conforms to $z$, and let $C$ be convex in $Q$. 
  When $Q_{|C}$ is connected, then $z_C$ is indecomposable. Whence
  $P_{|C}$ is connected also. By Proposition \ref{pro:poset-cc},
  $P \btl Q$. 

  Assume conversely $P \btl Q$. Let $C$ be convex in $Q$
  If $Q_{|C}$ is disconnected, then $P_{|C}$ is disconnected, and $z_C$
  decomposes.
  If $Q_{|C}$ is connected, then $P_{|C}$ is connected by
  Proposition \ref{pro:poset-cc}, and so $z_C$ is indecomposable.
  Thus $Q$ conforms to $z$.
\end{proof}

\begin{corollary} Given a preorder $P$ on $I$.
  Then $\Pre(\low_P) = \overline P$. 
\end{corollary}

\begin{proof} Let $A \sus B$ be down-sets of $P$, and
  $B \backslash A = C_1 \cup C_2$ a disjoint union of sets.
  We show that $P$ conforms to $\low_P$. Clearly $P$ is then the
  unique minimal element in $\Pre(\low_P)$, and so will give the
  statement.
  By Lemma \ref{lem:prod-sum}, that $P$ conforms to $\low_P$ means that 
  $A \cup C_1$ and $A \cup C_2$ are down-sets of $P$ iff
  \[ z(B) + z(A) = z(A \cup C_1) + z(A \cup C_2). \]
  But the left side is $0$. If we have equality, the right sides are
  $0$ and so give two down-sets of $P$. Conversely if the sets on the
  left are down-sets the value on the left is $0$ and we have equality.
%  Clearly $P$ is the unique minimal element in $\Pre(\low_P)$, and
%  so this is $\overline{P}$. 
\end{proof}

%Given a submodular $z$, let
%Let $I_1, I_2, \ldots I_r$ be the unique partition of $I$
%giving the unique decomposition of $z$ into indecomposable functions,
%as given in Lemma \ref{lem:pre-zprod}.
%The partition induces an equivalence relation (which is a preorder on $I$),
%which we denote $c(z)$. 

%%%%%%%%%%%%%%%%%%%%%%%%%%%%%%%%%%%%%%%%%%%%%%%%%%
%% Bijections, faces, preorders, fans

\section{Bijection between faces and conforming preorders} \label{sec:bij}

For a given submodular function $z : \pow(I) \pil \hRR$
we give a Galois correspondence between topologies/preorders on the
set $I$ and faces of the extended generalized permutahedron
$\egp(z)$. This Galois correspondence
gives a bijection between the topologies/preorders conforming to $z$ and
the non-empty faces of $\egp(z)$. Every face of $\egp(z)$ is also
an EGP and we give a simple description of its submodular function.
We do not claim strong originality of the main result in this section,
Theorem \ref{thm:bij-main}, on the bijection of conforming preorders
and faces of $\egp(z)$, see the discussion before the theorem.
But the notion of conforming preorders is new and so is the presentation
and development of the theory in terms of these.

\medskip
For a point $x \in \RR I$, recall the notation $x_A = \sum_{i \in A} x_i$
for a subset $A \sus I$.
By Section \ref{sec:egp} the extended generalized permutahedron of
a submodular function $z : \pow (I) \pil \hRR$ is
\begin{equation*} 
  \egp(z):=\left\{x=(x_i)_{i\in I}\in\RR I \, | \,
    x_I = z(I) \hbox{ and } x_A \leq z(A) \hbox{ for any }
    A\sus I \hbox{ such that } z(A)< \infty.\right\}
  \end{equation*}
Denote by $\Gamma(z)$ the set of faces of $\egp(z)$.

\subsection{A Galois correspondence between conforming
  preorders and faces of $\egp(z)$}
To a subset $\cS$ of the power set of $\pow(I)$, i.e. a family
of subsets of $\pow(I)$, we may associate a face of $\egp(z)$,
an element of $\Gamma(z)$:
\begin{equation} \label{eq:bij-phi}
  \Phi_z(\mathcal S):=\left\{x\in \egp(z) \,| \, x_A=z(A)
    \hbox{ for any } A\in \mathcal S\right\}.
  \end{equation}
The map $\Phi_z$ is surjective on $\Gamma(z)$ since
non-empty faces of $\egp(z)$ are defined by making some of
the defining inequalities into equalities.
Conversely given a face $G$ of $\egp(z)$, we get a subset of $\pow(I)$:
\[\Psi_z(G) = \{ A \, | \, z(A) < \infty, \, x_A = z(A)
\text{ for every point } x \in G \}. \]
It is readily seen that
\begin{equation} \label{eq:bij-inc}
  \Psi_z(G) \supseteq \cS \Leftrightarrow G  \sus \Phi_z(\cS).
  \end{equation}
  Letting $\ppow(I)$ be the power set of $\pow(I)$, i.e. families of subsets
  of $I$, 
this means we have a Galois correspondence
\begin{equation} \label{eq:bij-gal}
  \ppow(I)^{\op} \, \bihom{\Phi_z}{\Psi_z} \, \Gamma(z).
  \end{equation}

\begin{proposition} \label{pro:bij-psitop}
  For any $G$ in $\Gamma(z)$, $\Psi_z(G)$ is a topology.
  If $G$ is empty then $\Psi_z(G) = \top(z)$. If $G$ is non-empty,
  $\Psi_z(G)$ conforms to $z$.
\end{proposition}

\begin{proof}
   % Inclusion $\mathcal S\subseteq \widehat{\mathcal S}$ and Equality  $\Phi_z(\widehat{\mathcal S})=\Phi_z(\mathcal S)$ are immediate.
  If $G$ is empty, then $\Psi_z(G)$ is the topology $\top(z)$ corresponding
  to $\pre(z)$. Assume $G$ is non-empty and let 
 $A,B \in \Psi_z(G)$. Then for any $x \in G$:
\begin{equation*}
  z(A \cup B) + z(A \cap B) \geq x_{A \cup B} + x_{A \cap B} =
  x_A + x_B =   z(A)+z(B).
\end{equation*}
The definition of submodularity implies that the above are equalities, so:
\[z(A\cap B) = x_{A \cap B} \hskip 6mm\hbox{and}\hskip 6mm
  z(A\cup B)= x_{A \cup B}.\]
Therefore, $A\cap B$ and $A\cup B$ belong to $\Psi_z(G)$, and so this
is a topology.

\medskip Assume further that $G$ is non-empty.
We show that $\Psi_z(G)$ conforms to $z$. 
Let $A \sus B$ be in $\Psi_z(G)$
so $x_A = z(A)$ and $x_B = z(B)$ for $x \in G$. 
Let  $C:=B\setminus A=C_1\cup C_2$ be a disjoint union of sets.
By Lemma \ref{lem:prod-sum} to show that $\Psi_z(G)$ conforms to $z$,
we must show that $A \cup C_1$ and $A \cup C_2$
are both in $\Psi_z(G)$ iff
\begin{align} \label{eq:bij-zba}
  z(B) + z(A) = z(A \cup C_1) + z(A \cup C_2).
  \end{align}
But
\[z(A) + z(B) = x_A + x_B =  x_{A \cup C_1} + x_{A \cup C_2}
  \leq z(A \cup C_1) + z(A \cup C_2). \] 
Thus \eqref{eq:bij-zba} holds iff the above inequality is an equality and
this is iff 
$A \cup C_1$ and $A \cup C_2$ are both in $\Psi_z(G)$.
\end{proof}

%Note that
%\[ x_A + x_B = x_{A \cup C_1} + x_{A \cap C_2}. \]

%If $A \cup C_1$ and $A \cup C_2$ are in $\Psi_z(G)$, 

% Let $x \in G$. By Lemma \ref{lem:prod-sum}:
%\begin{equation*}
%  z':=z_{B/A}=z'\restr{C_1}.z'\restr{C_2} \Longleftrightarrow z(B)+z(A)=z(A\cup C_1)+z(A\cup C_2)
%\end{equation*}
%For any $x \in G$ we get
%\begin{eqnarray*}
%&&\sum_{i\in A}x_i+\sum_{i\in B} x_i=z(A\cup C_1)+z(A\cup C_2)\\
%&\Longleftrightarrow &2\sum_{i\in A}x_i+\sum_{i\in C_1}x_i+\sum_{i\in C_2}x_i=z(A\cup C_1)+z(A\cup C_2)\\
%  &\Longleftrightarrow &z(A\cup C_1)=\sum_{i\in A}x_i+\sum_{i\in C_1}x_i \hbox{ and }z(A\cup C_2)=\sum_{i\in A}x_i+\sum_{i\in C_2}x_i
%\end{eqnarray*}
%That this holds for any $x \in G$ is equivalent to 
%\begin{eqnarray*}
%&&A\cup C_1 \hbox{ and }A\cup C_2 \hbox{ are both open in }
%\cT \\
%& \Longleftrightarrow &C=C_1\sqcup C_2\hbox{ is a disconnected union}.
%\end{eqnarray*}
%Thus $\Psi_z(G)$ conforms to $z$.
%\end{proof}
\noindent The Galois correspondence \eqref{eq:bij-gal} then restricts to
a Galois correspondence 
\[ \Top(I)^\op \, \bihom{\Phi_z}{\Psi_z} \, \Gamma(z) \]
where $\Phi_z$ is surjective (by standard properties of Galois correspondences).
Changing  the perspective to preorders, we have a
Galois correspondence
\[ \Pre(I) \, \bihom{\Phi_z}{\Psi_z} \, \Gamma(z).\]
%  \quad  \Pre(z) \, \overset{1-1}{\longleftrightarrow} \, \Gamma^*(z)^{\op}. \]

\subsection{The submodular functions of faces}

\begin{definition}
  Let $P$ be a preorder compatible with $z$.  Denote
  \[ z_P = \prod_{C \in {\mathbf b}(P)} z_C. \]
\end{definition}

\begin{proposition} \label{pro:faces-egp} Let $P$ be a preorder.
  {The face}
$\Phi_z(P)$ is non-empty if and only if
  $P$ is compatible with $z$.
  In this case $\Phi_z(P) = \egp(z_P)$.
  If $P$ conforms to $z$, the dimension of this face is
  $|I| - |{\mathbf b}(P)|$.
\end{proposition}

\begin{proof}
Suppose $\Phi_z(P)$ is non-empty. 
  We have $\Psi_z \circ \Phi_z(P) \preceq P$
  and the left part conforms to $z$ by Proposition \ref{pro:bij-psitop}.
  Then $P$ is compatible with $z$.

  Conversely, assume  $P$ is 
  compatible with $z$. We show $\Phi_z(P)$ is $\egp(z_P)$. As $z_P$
  is submodular, this also implies that $\Phi_z(P)$ is non-empty by
  Lemma \ref{lem:egp-lin}.
    For $S \sus I$ and $x$ a point in $\RR I$,
    recall the notation $x_S = \sum_{i \in S} x_i$. 
    The face $\Phi_z(P)$ is defined by the following equalities and
    inequalities:
    \begin{itemize}
    \item $x_A = z(A)$ for $A$ a down-set in $P$,
    \item $x_S \leq z(S)$ for $S$ a subset of $I$.
      \end{itemize}
      The extended generalized permutahedron $\egp(z_P)$ is defined
      by
    \begin{itemize}
    \item For each bubble $C$ of $P$, $x_C = z_C(C)$,
    \item For $S \sus C$, $x_S \leq z_C(S)$.
    \end{itemize}
   
   If $C$ is a bubble we may write $C = B \backslash A$ for down-sets
   $A \sus B$ of $P$. Suppose first that the point $x$ belongs to $\Phi_z(P)$.
   We have $x_B = z(B)$ and
    $x_A = z(A)$, so:
    \[ x_C = x_B - x_A = z(B) - z(A) = z_C(C). \]
    Furthermore for $S \sus C$
    we have  $x_{A \cup S} \leq z(A \cup S)$ and so:
    \[ x_S = x_{A \cup S} - x_A \leq z(A \cup S) - z(A) = z_C(S),\]
    whence $x \in \egp(z_P)$.

    \medskip
    Suppose conversely $x \in \egp(z_P)$. Given a down-set $B$ of $P$,
    we may find a filtration of down-sets:
    \begin{equation} \label{eq:coint-Bfilt}
      B = B_0 \supset B_1 \supset \cdots \supset B_r = \emptyset
      \end{equation}
    such that each $C_i = B_i \backslash B_{i+1}$ is a bubble.
 %   This also gives all bubbles in $P$. We 
    Then
    \[ z(B) = \sum_{i \geq 0} \big((z(B_i) - z(B_{i+1})\big) = \sum_{i}
      z_{C_i}(C_i) = \sum_i x_{C_i} = x_B. \]

    Now given any subset $S \sus I$. We may take $B = I$ in the filtration
    above, and let $S^i = S \cap B_i$, and
    \[ S_i = S^i \backslash S^{i+1} = S \cap C_i. \]
    By submodularity
    \[ z(S^i) - z(S^{i+1}) = z(S_i \cup S^{i+1}) - z(S^{i+1})
      \geq z(S_i \cup B_{i+1}) - z(B_{i+1}) = z_{C_i}(S_i). \]
    We then get
    \[ z(S) = \sum_{i \geq 0} z(S_i) - z(S_{i+1}) \geq
      \sum_{i \geq 0} z_{C_i}(S^i) \geq \sum_{i \geq 0} x_{S_i} = x_S. \]
    Whence $x \in \Phi_z(P)$.

    \medskip
    If $P$ conforms to $z$, then for each bubble $C$ of $P$, the submodular function $z_C$ is
    indecomposable. Whence the dimension of $\Pi(z_C)$ is $|C|-1$.
    By Lemma \ref{lem:egp-prodsum}, $\Pi(z_P)$ is the product
    $\prod_{C \in \mathbf b(P)}\Pi(z_C)$, and by Corollary 
    \ref{cor:egp-dim} the dimension is $\sum_{C \in \mathbf b(P)} (|C|-1) = |I| - |\mathbf b(P)|$.
      \end{proof}
  
      The formula for dimension of a face is also given in \cite[Thm.3.32]{Fuj2005} but is there formulated in terms of distributive lattices.

\subsection{Bijections of faces, cones, and conforming preorders}
Recall that $\Top(z)$ are the topologies conforming to $z$. 
Denote by $\Top^{cp}(z)$ (resp. $\Pre^{cp}(z)$) the larger class of
topologies (resp. preorders)
compatible with $z$, and
denote by $\Gamma^*(z)$ the set of non-empty faces of $\egp(z)$.
By Proposition \ref{pro:faces-egp} we 
have a Galois correspondence
\[ \Top^{cp}(z)^\op \, \bihom{\Phi_z}{\Psi_z} \, \Gamma^*(z), \, 
  \text{ or equivalently } \,
  \Pre^{cp}(z) \, \bihom{\Phi_z}{\Psi_z} \, \Gamma^*(z).
\]
In particular, by standard properties of Galois correspondences,
we have a bijection between the images
\begin{equation} \label{eq:bij-gal2}
  \im \Psi_z \overset{1-1}{\longleftrightarrow} \im \Phi_z.
\end{equation}

The following theorem is essentially found in \cite{PRW2009} and in
\cite{Fuj2005}.
In \cite[Theorem 3.30]{Fuj2005}, it is formulated in terms of distributive
lattices. One must then restrict the setting to the case where
$\pre(z)$ is a poset, compared to our fully general setting.

That the face poset of a generalized permutahedron
corresponds bijectively to a class of preorders
is also immediate from \cite[Section 3.4]{PRW2009}.
They do not bring in submodular
functions. It is there formulated in terms
of the normal fan: The cones in the normal fan
correspond to preorders, and a face cone is incident to another
face cone iff the corresponding preorders are related by the
contraction $\btl$. 

%However, reference \cite{PRW2009} does not employ submodular functions, and
%the second reference \cite{Fuj2005} restricts itself by using distributive
%lattices instead of preorders.
%The following is essentially a reformulation of \cite[Theorem 3.30]{Fuj2005},
%but our setting is more general.
%There it is formulated in terms of distributive lattices. However
%this forces \cite{Fuj2005} to restrict the setting to the case where
%$\pre(z)$ is a poset. We formulate it in terms of preorders, and
%this allows a full and general setting.

%That the face poset of a generalized permutahedron
%corresponds bijectively to a class of preorders
%is also immediate from \cite{PRW2009}. There it is formulated in terms
%of the normal fan. It is shown the the cones in the normal fan
%corresponds to preorders, and a face is incident to another
%face iff the preorders are related by the coarsening $\btl$. 

\begin{theorem} \label{thm:bij-main} The map $\Phi_z$ is surjective.
  The image $\im \Psi_z$ is precisely the set $\Pre(z)$ of all preorders
  conforming to $z$. Whence we get a bijection
  \[\Pre(z) \, \overset{1-1}{\longleftrightarrow}
    \, \Gamma^*(z).
  \]
 % The dimension of the face $\Phi_z(\cT)$ is $|I| - |{\mathbf b}(P)|$
 % where $P$ is the preorder corresponding to $\cT$.
\end{theorem}

%\textcolor{blue}{Figure 1 and Figure 8 seem to coincide: Couldn't Figure 1 be safely removed?}
\begin{proof} As noted after \eqref{eq:bij-phi}, $\Phi_z$ is
  surjective on $\Gamma^*(z)$, and so for $G \in \Gamma^*(z)$,
  we have $\Phi_z \circ \Psi_z(G) = G$ (by standard properties of
  Galois correspondences), and so the restricted
  $\Phi_z$ is also surjective.
%The restriction of  $\Phi_z$ to $\Pcon(z)$ is surjective.
% comes from Proposition \ref{pro:bij-Fz}.
Let $Q$ be a preorder conforming to $z$.
Let $F = \Phi_z(Q)$ and $P = \Psi_z(F)$. Then every down-set of $Q$
must be a down-set for $P$, and hence $P \preceq Q$.
Since both $P$ and $Q$ conform to $z$, by Proposition \ref{pro:conf-PbtrQ},
$P \btl Q$. But then, if $P \neq Q$, by Corollary \ref{cor:poset-bubPQ} 
$|\mathbf b(P)| > |\mathbf b(Q)|$. Then the dimensions of
$\Phi_z(P)$ and of $\Phi_z(Q)$ are distinct,
contradicting that these faces are both $F$.
Whence $Q = \Psi_z \circ \Phi_z(Q)$.

%\medskip If $G$ is a  non-empty face of $\Gamma(z)$, it is in the image
%of $\Phi_z$ in  \eqref{eq:bij-gal}. But then $\phi_z \circ \psi_z(G) = G$,
%so $G$ is in the image of the correspondence \eqref{eq:bij-gal2}.
\end{proof}

\begin{example}
  Figure \ref{fig:toperm} shows the two-dimensional standard
  permutahedron. Its submodular function
 $z : \pow(\{a,b,c\}) \pil \RR$ has
 \[ z(S) = \begin{cases} 3, & |S| = 1, \\ 5, & |S| = 2 \\
      6, & |S| = 3 \end{cases}. \]
  Figure \ref{fig:toperm-conform} gives the conforming preorders
  corresponding to the faces of the permutahedron.
  Note that the point $(1,3,2)$ corresponds to the ordered set $b < c < a$ since
  its non-empty down-sets are
  $\{ b \}, \{b,c \}$ and $\{a,b,c\}$ with $z$-values
  respectively $3, 5$ and $6$.
\end{example}

%%%%%%%%%%%%%%%%%%%%%%%%%%%%%%%%%%%%%%%%%%%%%%
%% Totale ordener tre elementer

%%123
\newcommand{\cba}{
  \begin{tikzpicture}[scale=0.7, vertices/.style={draw, fill=black, circle, inner sep=1.5pt}]

%\node (31c) at (5.2-1,0) {c.};
\node [vertices] (31) at (-0.5,0) {};
\node [vertices] (32) at (-0.5,0.5) {};
\node [vertices] (33) at (-0.5,1) {};

\node[anchor = east] at (31) {$c$ \hskip 4mm {}};
\node[anchor = east] at (32) {$b$ \hskip 4mm {}};
\node[anchor = east] at (33) {$a$ \hskip 4mm {}};
\draw (31)--(32);
\draw (31)--(33); 

\end{tikzpicture}
}

%%%%%%%%%%%%%%%%%%%%%%%%%%%%%%%%%%%%%%%%%%%%%%%%%%%

%%213 
\newcommand{\cab}{
  \begin{tikzpicture}[scale=0.7, vertices/.style={draw, fill=black, circle, inner sep=1.5pt}]

%\node (31c) at (5.2-1,0) {c.};
\node [vertices] (31) at (-0.5,0) {};
\node [vertices] (32) at (-0.5,0.5) {};
\node [vertices] (33) at (-0.5,1) {};

\node[anchor = east] at (31) {$c$ \hskip 4mm {}};
\node[anchor = east] at (32) {$a$ \hskip 4mm {}};
\node[anchor = east] at (33) {$b$ \hskip 4mm {}};
\draw (31)--(32);
\draw (31)--(33); 

\end{tikzpicture}
}

%% 312

\newcommand{\acb}{
  \begin{tikzpicture}[scale=0.7, vertices/.style={draw, fill=black, circle, inner sep=1.5pt}]

%\node (31c) at (5.2-1,0) {c.};
\node [vertices] (31) at (-0.5,0) {};
\node [vertices] (32) at (-0.5,0.5) {};
\node [vertices] (33) at (-0.5,1) {};

\node[anchor = east] at (31) {$a$ \hskip 4mm {}};
\node[anchor = east] at (32) {$c$ \hskip 4mm {}};
\node[anchor = east] at (33) {$b$ \hskip 4mm {}};
\draw (31)--(32);
\draw (31)--(33); 

\end{tikzpicture}
}

%% 321

\newcommand{\abc}{
  \begin{tikzpicture}[scale=0.7, vertices/.style={draw, fill=black, circle, inner sep=1.5pt}]

%\node (31c) at (5.2-1,0) {c.};
\node [vertices] (31) at (-0.5,0) {};
\node [vertices] (32) at (-0.5,0.5) {};
\node [vertices] (33) at (-0.5,1) {};

\node[anchor = east] at (31) {$a$ \hskip 4mm {}};
\node[anchor = east] at (32) {$b$ \hskip 4mm {}};
\node[anchor = east] at (33) {$c$ \hskip 4mm {}};
\draw (31)--(32);
\draw (31)--(33); 

\end{tikzpicture}
}
%% 231

\newcommand{\bac}{
  \begin{tikzpicture}[scale=0.7, vertices/.style={draw, fill=black, circle, inner sep=1.5pt}]

%\node (31c) at (5.2-1,0) {c.};
\node [vertices] (31) at (-0.5,0) {};
\node [vertices] (32) at (-0.5,0.5) {};
\node [vertices] (33) at (-0.5,1) {};

\node[anchor = east] at (31) {$b$ \hskip 4mm {}};
\node[anchor = east] at (32) {$a$ \hskip 4mm {}};
\node[anchor = east] at (33) {$c$ \hskip 4mm {}};
\draw (31)--(32);
\draw (31)--(33); 

\end{tikzpicture}
}

%% 132

\newcommand{\bca}{
  \begin{tikzpicture}[scale=0.7, vertices/.style={draw, fill=black, circle, inner sep=1.5pt}]

%\node (31c) at (5.2-1,0) {c.};
\node [vertices] (31) at (-0.5,0) {};
\node [vertices] (32) at (-0.5,0.5) {};
\node [vertices] (33) at (-0.5,1) {};

\node[anchor = east] at (31) {$b$ \hskip 4mm {}};
\node[anchor = east] at (32) {$c$ \hskip 4mm {}};
\node[anchor = east] at (33) {$a$ \hskip 4mm {}};
\draw (31)--(32);
\draw (31)--(33); 

\end{tikzpicture}
}

%%%%%%%%%%%%%%%%%%%%%%%%%%%%%%
%% Konforme til kantene, med boblen OPPE
%% Leses nedenfra og opp

%% Kant 12-3
\newcommand{\csab}{
  \begin{tikzpicture}[scale=0.9, vertices/.style={draw, fill=black, circle, inner sep=1.5pt}]

%\node (31c) at (5.2-1,0) {c.};
\node [vertices] (31) at (-0.5,-0.2) {};
\node [vertices, gray] (32) at (-0.33,0.7) {};
\node [vertices, gray] (33) at (-0.67,0.7) {};

\node[anchor = east] at (31) {$c$ \hskip 0.1pt {}};
\node[anchor = west] at (32) {$\, b${}};
\node[anchor = east] at (33) {$a\, $};

\node (ss3) at (-0.5, 0.7) {};
\draw (ss3) ellipse (0.58 and 0.42);

%\fill (dd3) circle (0.05);

\draw (31)--(-0.45, 0.25);

\end{tikzpicture}
}

%% Kant 23-1
\newcommand{\asbc}{
  \begin{tikzpicture}[scale=0.9, vertices/.style={draw, fill=black, circle, inner sep=1.5pt}]

%\node (31c) at (5.2-1,0) {c.};
\node [vertices] (31) at (-0.5,-0.2) {};
\node [vertices, gray] (32) at (-0.33,0.7) {};
\node [vertices, gray] (33) at (-0.67,0.7) {};

\node[anchor = east] at (31) {$a$ \hskip 0.1pt {}};
\node[anchor = west] at (32) {$\, c${}};
\node[anchor = east] at (33) {$b\, $};

\node (ss3) at (-0.5, 0.7) {};
\draw (ss3) ellipse (0.58 and 0.42);

%\fill (dd3) circle (0.05);

\draw (31)--(-0.44, 0.32);

\end{tikzpicture}
}

%% Kant 13-2
\newcommand{\bsac}{
  \begin{tikzpicture}[scale=0.9, vertices/.style={draw, fill=black, circle, inner sep=1.5pt}]

%\node (31c) at (5.2-1,0) {c.};
\node [vertices] (31) at (-0.5,-0.2) {};
\node [vertices, gray] (32) at (-0.33,0.7) {};
\node [vertices, gray] (33) at (-0.67,0.7) {};

\node[anchor = east] at (31) {$b$ \hskip 0.1pt {}};
\node[anchor = west] at (32) {$\, c${}};
\node[anchor = east] at (33) {$a\, $};

\node (ss3) at (-0.5, 0.7) {};
\draw (ss3) ellipse (0.58 and 0.42);

\draw (31)--(-0.57, 0.27);

\end{tikzpicture}
}

%%%%%%%%%%%%%%%%%%%%%%%%%%%%%%
%% Konforme til kantene, med boblen NEDE
%% Leses nedenfra og opp

%% Kant 12-3
\newcommand{\absc}{
  \begin{tikzpicture}[scale=0.9, vertices/.style={draw, fill=black, circle, inner sep=1.5pt}]

%\node (31c) at (5.2-1,0) {c.};
\node [vertices, gray] (31) at (-0.33,-0.2) {};
\node [vertices, gray] (32) at (-0.67,-0.2) {};
\node [vertices] (33) at (-0.5,0.7) {};

\node[anchor = west] at (31) {$\, b$ };
\node[anchor = east] at (32) {$a \,$ };
\node[anchor = east] at (33) {$c$ \hskip 0.1pt {}};

\node (ss3) at (-0.5, -0.2) {};
\draw (ss3) ellipse (0.58 and 0.42);

%\fill (dd3) circle (0.05);

\draw (33)--(-0.57, 0.25);

\end{tikzpicture}
}

%% Kant 12-3
\newcommand{\bcsa}{
  \begin{tikzpicture}[scale=0.9, vertices/.style={draw, fill=black, circle, inner sep=1.5pt}]

%\node (31c) at (5.2-1,0) {c.};
\node [vertices, gray] (31) at (-0.33,-0.2) {};
\node [vertices, gray] (32) at (-0.67,-0.2) {};
\node [vertices] (33) at (-0.5,0.7) {};

\node[anchor = west] at (31) {$\, c$ };
\node[anchor = east] at (32) {$b \,$ };
\node[anchor = east] at (33) {$a$ \hskip 0.1pt {}};

\node (ss3) at (-0.5, -0.2) {};
\draw (ss3) ellipse (0.58 and 0.42);

%\fill (dd3) circle (0.05);

\draw (33)--(-0.57, 0.20);

\end{tikzpicture}
}

%% Kant 12-3
\newcommand{\acsb}{
  \begin{tikzpicture}[scale=0.9, vertices/.style={draw, fill=black, circle, inner sep=1.5pt}]

%\node (31c) at (5.2-1,0) {c.};
\node [vertices, gray] (31) at (-0.33,-0.2) {};
\node [vertices, gray] (32) at (-0.67,-0.2) {};
\node [vertices] (33) at (-0.5,0.7) {};

\node[anchor = west] at (31) {$\, c$ };
\node[anchor = east] at (32) {$a \,$ };
\node[anchor = east] at (33) {$b$ \hskip 0.1pt {}};

\node (ss3) at (-0.5, -0.2) {};
\draw (ss3) ellipse (0.58 and 0.42);

%\fill (dd3) circle (0.05);

\draw (33)--(-0.42, 0.20);

\end{tikzpicture}
}

%%%%%%%%%%%%%%%%%%%%%%%%%%%%%%%%%%
%% Boblen i midten
%% Kant 12-3

\newcommand{\abcss}{
  \begin{tikzpicture}[scale=0.9, vertices/.style={draw, fill=black, circle, inner sep=1.5pt}]

%\node (31c) at (5.2-1,0) {c.};
\node [vertices, gray] (31) at (-0.33,-0.2) {};
\node [vertices, gray] (32) at (-0.67,-0.2) {};
\node [vertices, gray] (33) at (-0.5, 0.03) {};

\node[anchor = west] at (31) {$\, c$ };
\node[anchor = east] at (32) {$b \,$ };
\node[anchor = south east] at (33) {$a$};

\node (ss3) at (-0.5, -0.13) {};
\draw (ss3) ellipse (0.55 and 0.55);

%\fill (dd3) circle (0.05);
%\draw (33)--(-0.42, 0.20);

\end{tikzpicture}
}

%%%%%%%%%%%%%%%%%%%%%%%%%%%%%%%%%%%%%%%%%%%%%%%%%%%
%% Hexagon med konforme preordener

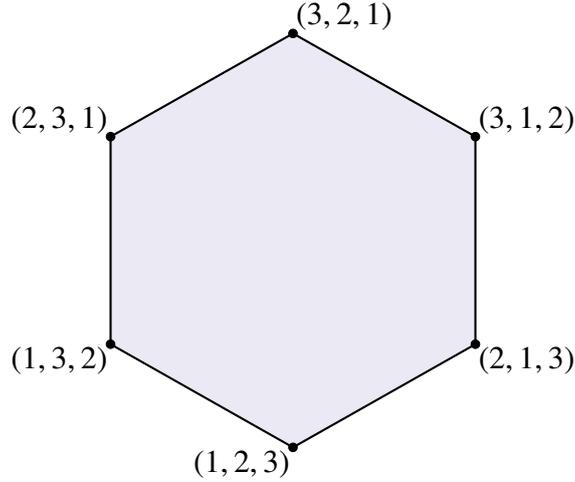
\begin{figure}
\begin{center}
%Første bilde
\begin{tikzpicture}[inner sep=1pt, scale=0.8]

\coordinate (a) at (0,0) ; % [anchor =  north east] {$(1,2,3)$};
\coordinate (b) at (3, 1.71);  %[anchor = north west] {$(2,1,3)$};
\coordinate (c) at (3, 5.15);  %[anchor = north east] {$(3,1,2)$};
\coordinate (d) at (0, 6.86);  %[anchor = north east]{$(3,2,1)$};
\coordinate (e) at (-3, 5.15);  % [anchor = north west] {$(2,3,1)$};
\coordinate (f) at (-3, 1.71);  % [anchor = south west] {$(1,3,2)$};

    %Draw permutohedron: 
            \draw[black, thick, fill=tol1, fill opacity = 0.1] 
            (a) -- (b)  -- (c) -- (d) -- (e) -- (f) -- (a);
            %Draw centerpoint and vertices:
            %\draw[black, fill] (0.5,)    circle (1pt);
            \draw[black, fill] (a)     circle (2pt);
            \draw[black, fill] (b)     circle (2pt);
            \draw[black, fill] (c)     circle (2pt);
            \draw[black, fill] (d)     circle (2pt);
            \draw[black, fill] (e)     circle (2pt);
            \draw[black, fill] (f)     circle (2pt);
            % \node[] at (-0.5, 0.8660)   [anchor = east]       { $f$};

\node[] at (a) [anchor = north east] {$(1,2,3)$};
\node[] at (b) [anchor = north west] {$(2,1,3)$};
\node[] at (c) [anchor = south west] {$(3,1,2)$};
\node[] at (d) [anchor = south west] {$(3,2,1)$};
\node[] at (e) [anchor = south east] {$(2,3,1)$};
\node[] at (f) [anchor = north east] {$(1,3,2)$};
      
          \end{tikzpicture}
          \end{center}
        \caption{Two-dimensional permutahedron}
        \label{fig:toperm}
    \end{figure}

\begin{figure}
\begin{center}
%Første bilde
\begin{tikzpicture}[inner sep=1pt, scale=0.8]

\coordinate (a) at (0,0) ; % [anchor =  north east] {$(1,2,3)$};
\coordinate (b) at (3, 1.71);  %[anchor = north west] {$(2,1,3)$};
\coordinate (c) at (3, 5.15);  %[anchor = north east] {$(3,1,2)$};
\coordinate (d) at (0, 6.86);  %[anchor = north east]{$(3,2,1)$};
\coordinate (e) at (-3, 5.15);  % [anchor = north west] {$(2,3,1)$};
\coordinate (f) at (-3, 1.71);  % [anchor = south west] {$(1,3,2)$};

\coordinate (x) at (-1.5,0.855) ; % [anchor =  north east] {$(1,2,3)$};
\coordinate (y) at (1.5, 0.855);  %[anchor = north west] {$(2,1,3)$};
\coordinate (z) at (3, 3.43);  %[anchor = north east] {$(3,1,2)$};
\coordinate (t) at (1.5, 6.005);  %[anchor = north east]{$(3,2,1)$};
\coordinate (u) at (-1.5, 6.005);  % [anchor = north west] {$(2,3,1)$};
\coordinate (v) at (-3, 3.43);  % [anchor = south west] {$(1,3,2)$};

\coordinate (s) at (0, 3.43);

    %Draw permutohedron: 
            \draw[black, thick, fill=tol1, fill opacity = 0.1] 
            (a) -- (b)  -- (c) -- (d) -- (e) -- (f) -- (a);
            %Draw centerpoint and vertices:
            %\draw[black, fill] (0.5,)    circle (1pt);
            \draw[black, fill] (a)     circle (2pt);
            \draw[black, fill] (b)     circle (2pt);
            \draw[black, fill] (c)     circle (2pt);
            \draw[black, fill] (d)     circle (2pt);
            \draw[black, fill] (e)     circle (2pt);
            \draw[black, fill] (f)     circle (2pt);
            % \node[] at (-0.5, 0.8660)   [anchor = east]       { $f$};

\node[] at (a) [anchor = north east] {$\cba$ \hskip 5mm {}};
\node[] at (b) [anchor = north west] {$\cab$ \hskip 5mm {}};
\node[] at (c) [anchor = south west] {$\acb$ \hskip 5mm {}};
\node[] at (d) [anchor = south west] {$\abc$ \hskip 5mm {}};
\node[] at (e) [anchor = south east] {$\bac$ \hskip 5mm {}};
\node[] at (f) [anchor = north east] {$\bca$ \hskip 5mm {}};
      
\node[] at (x) [anchor = north east] {$\bcsa$};
\node[] at (y) [anchor = north west] {$\csab$};
\node[] at (z) [anchor = west] {$\acsb$};
\node[] at (t) [anchor = south west] {$\asbc$};
\node[] at (u) [anchor = south east] {$\absc$};
\node[] at (v) [anchor = east] {$\bsac$};

\node[] at (s) {$\abcss$};
\end{tikzpicture}
          \end{center}
        \caption{Conforming preorders for the two-dimensional permutahedron}
        \label{fig:toperm-conform}
    \end{figure}
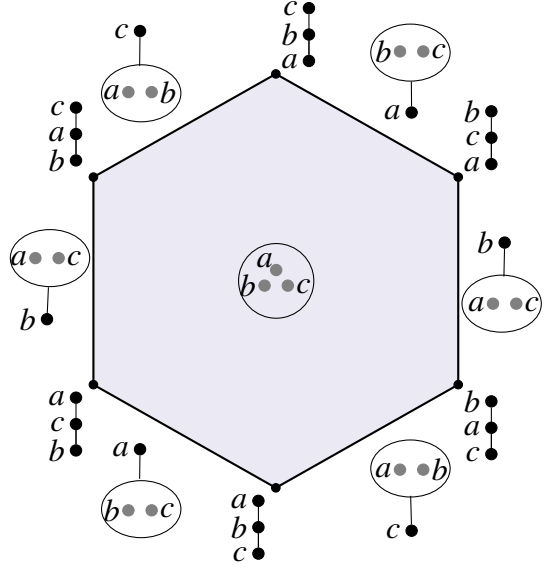

\begin{example} \label{ex:bij-pent}
  Figures \ref{fig:togenperm} and \ref{fig:togenperm-conform} shows the
  generalized permutahedron, a pentagon, of the submodular function defined on
  $\pow(\{a,b,c\})$ by 
  \[ z(a) = z(b) = z(c) = 3, \quad z(ab) = z(bc) = 5, \quad
    z(ac) = z(abc) = 6, \]
  and the corresponding preorders.
\end{example}

%%%%%%%%%%%%%%%%%%%%%%%%%%%%%%%%%%%%%%%%%%%%%%
%% Totale ordener tre elementer

%%123
\newcommand{\recba}{
  \begin{tikzpicture}[scale=0.7, vertices/.style={draw, fill=black, circle, inner sep=1.5pt}]

%\node (31c) at (5.2-1,0) {c.};
\node [vertices] (31) at (-0.5,0) {};
\node [vertices] (32) at (-0.5,0.5) {};
\node [vertices] (33) at (-0.5,1) {};

\node[anchor = east] at (31) {$c$ \hskip 4mm {}};
\node[anchor = east] at (32) {$b$ \hskip 4mm {}};
\node[anchor = east] at (33) {$a$ \hskip 4mm {}};
\draw (31)--(32);
\draw (31)--(33); 

\end{tikzpicture}
}

%%%%%%%%%%%%%%%%%%%%%%%%%%%%%%%%%%%%%%%%%%%%%%%%%%%

%%213 
\renewcommand{\cab}{
  \begin{tikzpicture}[scale=0.7, vertices/.style={draw, fill=black, circle, inner sep=1.5pt}]

%\node (31c) at (5.2-1,0) {c.};
\node [vertices] (31) at (-0.5,0) {};
\node [vertices] (32) at (-0.5,0.5) {};
\node [vertices] (33) at (-0.5,1) {};

\node[anchor = east] at (31) {$c$ \hskip 4mm {}};
\node[anchor = east] at (32) {$a$ \hskip 4mm {}};
\node[anchor = east] at (33) {$b$ \hskip 4mm {}};
\draw (31)--(32);
\draw (31)--(33); 

\end{tikzpicture}
}

%% 312

\renewcommand{\acb}{
  \begin{tikzpicture}[scale=0.7, vertices/.style={draw, fill=black, circle, inner sep=1.5pt}]

%\node (31c) at (5.2-1,0) {c.};
\node [vertices] (31) at (-0.5,0) {};
\node [vertices] (32) at (-0.5,0.5) {};
\node [vertices] (33) at (-0.5,1) {};

\node[anchor = east] at (31) {$a$ \hskip 4mm {}};
\node[anchor = east] at (32) {$c$ \hskip 4mm {}};
\node[anchor = east] at (33) {$b$ \hskip 4mm {}};
\draw (31)--(32);
\draw (31)--(33); 

\end{tikzpicture}
}

%% 321

\renewcommand{\abc}{
  \begin{tikzpicture}[scale=0.7, vertices/.style={draw, fill=black, circle, inner sep=1.5pt}]

%\node (31c) at (5.2-1,0) {c.};
\node [vertices] (31) at (-0.5,0) {};
\node [vertices] (32) at (-0.5,0.5) {};
\node [vertices] (33) at (-0.5,1) {};

\node[anchor = east] at (31) {$a$ \hskip 4mm {}};
\node[anchor = east] at (32) {$b$ \hskip 4mm {}};
\node[anchor = east] at (33) {$c$ \hskip 4mm {}};
\draw (31)--(32);
\draw (31)--(33); 

\end{tikzpicture}
}
%% 231

\renewcommand{\bac}{
  \begin{tikzpicture}[scale=0.7, vertices/.style={draw, fill=black, circle, inner sep=1.5pt}]

%\node (31c) at (5.2-1,0) {c.};
\node [vertices] (31) at (-0.5,0) {};
\node [vertices] (32) at (-0.5,0.5) {};
\node [vertices] (33) at (-0.5,1) {};

\node[anchor = east] at (31) {$b$ \hskip 4mm {}};
\node[anchor = east] at (32) {$a$ \hskip 4mm {}};
\node[anchor = east] at (33) {$c$ \hskip 4mm {}};
\draw (31)--(32);
\draw (31)--(33); 

\end{tikzpicture}
}

%% 132

\renewcommand{\bca}{
  \begin{tikzpicture}[scale=0.7, vertices/.style={draw, fill=black, circle, inner sep=1.5pt}]

%\node (31c) at (5.2-1,0) {c.};
\node [vertices] (31) at (-0.5,0) {};
\node [vertices] (32) at (-0.5,0.5) {};
\node [vertices] (33) at (-0.5,1) {};

\node[anchor = east] at (31) {$b$ \hskip 4mm {}};
\node[anchor = east] at (32) {$c$ \hskip 4mm {}};
\node[anchor = east] at (33) {$a$ \hskip 4mm {}};
\draw (31)--(32);
\draw (31)--(33); 

\end{tikzpicture}
}

%%%%%%%%%%%%%%%%%%%%%%%%%%%%%%
%% Konforme til kantene, med boblen OPPE
%% Leses nedenfra og opp

%% Kant 12-3
\renewcommand{\csab}{
  \begin{tikzpicture}[scale=0.9, vertices/.style={draw, fill=black, circle, inner sep=1.5pt}]

%\node (31c) at (5.2-1,0) {c.};
\node [vertices] (31) at (-0.5,-0.2) {};
\node [vertices, gray] (32) at (-0.33,0.7) {};
\node [vertices, gray] (33) at (-0.67,0.7) {};

\node[anchor = east] at (31) {$c$ \hskip 0.1pt {}};
\node[anchor = west] at (32) {$\, b${}};
\node[anchor = east] at (33) {$a\, $};

\node (ss3) at (-0.5, 0.7) {};
\draw (ss3) ellipse (0.58 and 0.42);

%\fill (dd3) circle (0.05);

\draw (31)--(-0.45, 0.25);

\end{tikzpicture}
}

%% Kant 23-1
\renewcommand{\asbc}{
  \begin{tikzpicture}[scale=0.9, vertices/.style={draw, fill=black, circle, inner sep=1.5pt}]

%\node (31c) at (5.2-1,0) {c.};
\node [vertices] (31) at (-0.5,-0.2) {};
\node [vertices, gray] (32) at (-0.33,0.7) {};
\node [vertices, gray] (33) at (-0.67,0.7) {};

\node[anchor = east] at (31) {$a$ \hskip 0.1pt {}};
\node[anchor = west] at (32) {$\, c${}};
\node[anchor = east] at (33) {$b\, $};

\node (ss3) at (-0.5, 0.7) {};
\draw (ss3) ellipse (0.58 and 0.42);

%\fill (dd3) circle (0.05);

\draw (31)--(-0.44, 0.32);

\end{tikzpicture}
}

%% Kant 13-2
\renewcommand{\bsac}{
  \begin{tikzpicture}[scale=0.9, vertices/.style={draw, fill=black, circle, inner sep=1.5pt}]

%\node (31c) at (5.2-1,0) {c.};
\node [vertices] (31) at (-0.5,-0.2) {};
\node [vertices, gray] (32) at (-0.33,0.7) {};
\node [vertices, gray] (33) at (-0.67,0.7) {};

\node[anchor = east] at (31) {$b$ \hskip 0.1pt {}};
\node[anchor = west] at (32) {$\, c${}};
\node[anchor = east] at (33) {$a\, $};

\node (ss3) at (-0.5, 0.7) {};
\draw (ss3) ellipse (0.58 and 0.42);

\draw (31)--(-0.57, 0.27);

\end{tikzpicture}
}

%%%%%%%%%%%%%%%%%%%%%%%%%%%%%%
%% Konforme til kantene, med boblen NEDE
%% Leses nedenfra og opp

%% Kant 12-3
\renewcommand{\absc}{
  \begin{tikzpicture}[scale=0.9, vertices/.style={draw, fill=black, circle, inner sep=1.5pt}]

%\node (31c) at (5.2-1,0) {c.};
\node [vertices, gray] (31) at (-0.33,-0.2) {};
\node [vertices, gray] (32) at (-0.67,-0.2) {};
\node [vertices] (33) at (-0.5,0.7) {};

\node[anchor = west] at (31) {$\, b$ };
\node[anchor = east] at (32) {$a \,$ };
\node[anchor = east] at (33) {$c$ \hskip 0.1pt {}};

\node (ss3) at (-0.5, -0.2) {};
\draw (ss3) ellipse (0.58 and 0.42);

%\fill (dd3) circle (0.05);

\draw (33)--(-0.57, 0.25);

\end{tikzpicture}
}

%% Kant 12-3
\renewcommand{\bcsa}{
  \begin{tikzpicture}[scale=0.9, vertices/.style={draw, fill=black, circle, inner sep=1.5pt}]

%\node (31c) at (5.2-1,0) {c.};
\node [vertices, gray] (31) at (-0.33,-0.2) {};
\node [vertices, gray] (32) at (-0.67,-0.2) {};
\node [vertices] (33) at (-0.5,0.7) {};

\node[anchor = west] at (31) {$\, c$ };
\node[anchor = east] at (32) {$b \,$ };
\node[anchor = east] at (33) {$a$ \hskip 0.1pt {}};

\node (ss3) at (-0.5, -0.2) {};
\draw (ss3) ellipse (0.58 and 0.42);

%\fill (dd3) circle (0.05);

\draw (33)--(-0.57, 0.20);

\end{tikzpicture}
}

%% Kant 12-3
\renewcommand{\acsb}{
  \begin{tikzpicture}[scale=0.9, vertices/.style={draw, fill=black, circle, inner sep=1.5pt}]

%\node (31c) at (5.2-1,0) {c.};
\node [vertices, gray] (31) at (-0.33,-0.2) {};
\node [vertices, gray] (32) at (-0.67,-0.2) {};
\node [vertices] (33) at (-0.5,0.7) {};

\node[anchor = west] at (31) {$\, c$ };
\node[anchor = east] at (32) {$a \,$ };
\node[anchor = east] at (33) {$b$ \hskip 0.1pt {}};

\node (ss3) at (-0.5, -0.2) {};
\draw (ss3) ellipse (0.58 and 0.42);

%\fill (dd3) circle (0.05);

\draw (33)--(-0.42, 0.20);

\end{tikzpicture}
}

%%%%%%%%%%%%%%%%%%%%%%%%%%%%%%%%%%
%% Boblen i midten
%% Kant 12-3

\renewcommand{\abcss}{
  \begin{tikzpicture}[scale=0.9, vertices/.style={draw, fill=black, circle, inner sep=1.5pt}]

%\node (31c) at (5.2-1,0) {c.};
\node [vertices, gray] (31) at (-0.33,-0.2) {};
\node [vertices, gray] (32) at (-0.67,-0.2) {};
\node [vertices, gray] (33) at (-0.5, 0.03) {};

\node[anchor = west] at (31) {$\, c$ };
\node[anchor = east] at (32) {$b \,$ };
\node[anchor = south east] at (33) {$a$};

\node (ss3) at (-0.5, -0.13) {};
\draw (ss3) ellipse (0.55 and 0.55);

%\fill (dd3) circle (0.05);
%\draw (33)--(-0.42, 0.20);

\end{tikzpicture}
}

%%%%%%%%%%%%%%%%%%%%%%%%%%
%% Cherry posetet
\newcommand{\acAb}{
  \begin{tikzpicture}[scale=0.7, vertices/.style={draw, fill=black, circle, inner sep=1.5pt}]

%\node (31c) at (5.2-1,0) {c.};
\node [vertices] (31) at (-0.8,0) {};
\node [vertices] (32) at (-0.2,0) {};
\node [vertices] (33) at (-0.5,0.7) {};

\node[anchor = east] at (31) {$a\,$ };
\node[anchor = west] at (32) {$\, c$};
\node[anchor = east] at (33) {$b \hskip 1.5mm $};
\draw (31)--(33);
\draw (32)--(33); 

\end{tikzpicture}
}

    %%%%%%%%%%%%%%%%%%%%%%%%%%%%%%%%%%%%
    %% Hexagon med konforme preordener

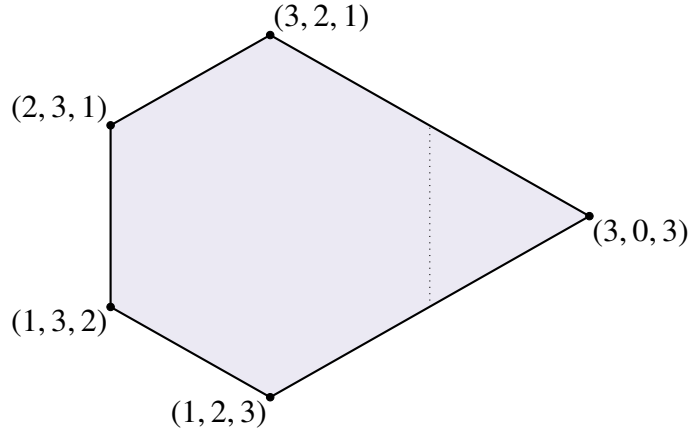
\begin{figure}
\begin{center}
%Første bilde
\begin{tikzpicture}[inner sep=1pt, scale=0.7]

\coordinate (a) at (0,0) ; % [anchor =  north east] {$(1,2,3)$};
\coordinate (b) at (3, 1.71);  %[anchor = north west] {$(2,1,3)$};
\coordinate (c) at (3, 5.15);  %[anchor = north east] {$(3,1,2)$};
\coordinate (bc) at (6, 3.43);
\coordinate (d) at (0, 6.86);  %[anchor = north east]{$(3,2,1)$};
\coordinate (e) at (-3, 5.15);  % [anchor = north west] {$(2,3,1)$};
\coordinate (f) at (-3, 1.71);  % [anchor = south west] {$(1,3,2)$};

    %Draw permutohedron: 
            \draw[black, thick, fill=tol1, fill opacity = 0.1] 
            (a) -- (bc) -- (d) -- (e) -- (f) -- (a);
            \draw[black, dotted] 
            (b) -- (c);
            %Draw centerpoint and vertices:
            %\draw[black, fill] (0.5,)    circle (1pt);
            \draw[black, fill] (a)     circle (2pt);
            \draw[black, fill] (bc)     circle (2pt);
            %\draw[black, fill] (c)     circle (2pt);
            \draw[black, fill] (d)     circle (2pt);
            \draw[black, fill] (e)     circle (2pt);
            \draw[black, fill] (f)     circle (2pt);
            % \node[] at (-0.5, 0.8660)   [anchor = east]       { $f$};

\node[] at (a) [anchor = north east] {$(1,2,3)$};
\node[] at (bc) [anchor = north west] {$(3,0,3)$};
%\node[] at (c) [anchor = south west] {$(3,1,2)$};
\node[] at (d) [anchor = south west] {$(3,2,1)$};
\node[] at (e) [anchor = south east] {$(2,3,1)$};
\node[] at (f) [anchor = north east] {$(1,3,2)$};
      
          \end{tikzpicture}
          \end{center}
          \caption{Generalized permutahedron $\Pi(z)$ of
            Example \ref{ex:bij-pent}}
        \label{fig:togenperm}
    \end{figure}    

\begin{figure}
\begin{center}
%Første bilde
\begin{tikzpicture}[inner sep=1pt, scale=0.8]

\coordinate (a) at (0,0) ; % [anchor =  north east] {$(1,2,3)$};
\coordinate (b) at (3, 1.71);  %[anchor = north west] {$(2,1,3)$};
\coordinate (c) at (3, 5.15);  %[anchor = north east] {$(3,1,2)$};
\coordinate (bc) at (6, 3.83);  %[anchor = north east] {$(3,1,2)$};
\coordinate (d) at (0, 6.86);  %[anchor = north east]{$(3,2,1)$};
\coordinate (e) at (-3, 5.15);  % [anchor = north west] {$(2,3,1)$};
\coordinate (f) at (-3, 1.71);  % [anchor = south west] {$(1,3,2)$};

\coordinate (x) at (-1.5,0.855) ; % [anchor =  north east] {$(1,2,3)$};
\coordinate (y) at (3, 1.71);  %[anchor = north west] {$(2,1,3)$};
%\coordinate (z) at (3, 5.15);  %[anchor = north east] {$(3,1,2)$};
\coordinate (t) at (3, 5.15);  %[anchor = north east]{$(3,2,1)$};
\coordinate (u) at (-1.5, 6.005);  % [anchor = north west] {$(2,3,1)$};
\coordinate (v) at (-3, 3.43);  % [anchor = south west] {$(1,3,2)$};

\coordinate (s) at (0, 3.43);

    %Draw permutohedron: 
            \draw[black, thick, fill=tol1, fill opacity = 0.1] 
            (a) -- (bc) -- (d) -- (e) -- (f) -- (a);
            %Draw centerpoint and vertices:
            %\draw[black, fill] (0.5,)    circle (1pt);
            \draw[black, fill] (a)     circle (2pt);
            \draw[black, fill] (bc)     circle (2pt);
           % \draw[black, fill] (c)     circle (2pt);
            \draw[black, fill] (d)     circle (2pt);
            \draw[black, fill] (e)     circle (2pt);
            \draw[black, fill] (f)     circle (2pt);
            % \node[] at (-0.5, 0.8660)   [anchor = east]       { $f$};

\node[] at (a) [anchor = north east] {$\cba$ \hskip 5mm {}};
\node[] at (bc) [anchor = north west] {$\acAb$ \hskip 5mm {}};
%\node[] at (c) [anchor = south west] {$\acb$ \hskip 5mm {}};
\node[] at (d) [anchor = south west] {$\abc$ \hskip 5mm {}};
\node[] at (e) [anchor = south east] {$\bac$ \hskip 5mm {}};
\node[] at (f) [anchor = north east] {$\bca$ \hskip 5mm {}};
      
\node[] at (x) [anchor = north east] {$\bcsa$};
\node[] at (y) [anchor = north west] {$\csab$};
%\node[] at (z) [anchor = west] {$\acsb$};
\node[] at (t) [anchor = south west] {$\asbc$ };
\node[] at (u) [anchor = south east] {$\absc$};
\node[] at (v) [anchor = east] {$\bsac$};

\node[] at (s) {$\abcss$};
\end{tikzpicture}
          \end{center}
         \caption{Conforming preorders for the submodular function $z$
         of Example \ref{ex:bij-pent}}
        \label{fig:togenperm-conform}
    \end{figure}
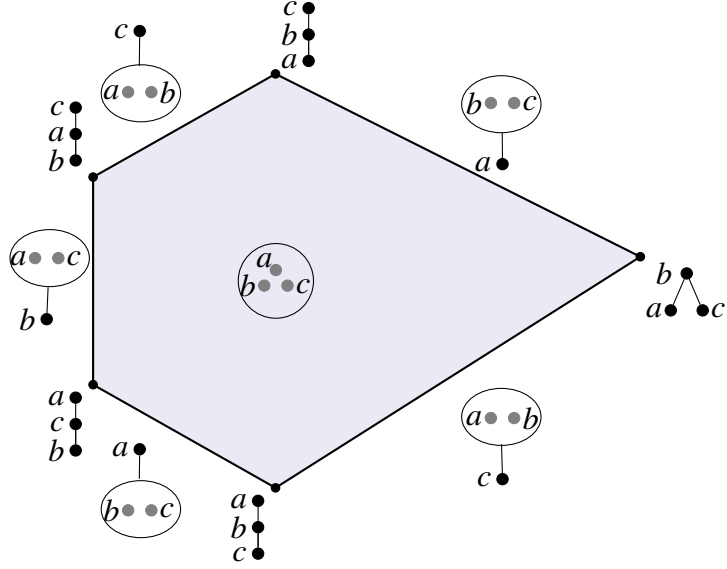

We derive the following on the partially ordered set $\Pre(z)$
of conforming preorders.

\begin{proposition} \label{pro:conf-minmax} Let $z$ be a submodular function.

  \noindent a. $\Pre(z)$ has a unique maximal element, the
  totally disconnected preorder
$\cpre(z)$. Moreover {there is equality between sets of connected components
 $\mathbf c(P) = \mathbf c\big(\cpre(z)\big)$} for any $P \in \Pre(z)$.

 \noindent b. The meet of all preorders of $\Pre(z)$ is the preorder $\pre(z)$
defined in Subsection \ref{subsec:pre-bf}. Moreover
  $\mathbf b(P) = \mathbf b (\pre(z))$ for every  minimal $P \in \Pre(z)$.

  \noindent c. In particular, if  $z$ is a finite submodular function
  then every minimal $P$ in $\Pre(z)$  is a poset.
\end{proposition}

\begin{proof}
  a.  If $P$ conforms to $z$, a decomposition  $z = z_1 \cdot z_2$
 corresponds to a decomposition
 $P = P_1 \sqcup P_2$ into a disconnected union of two preorders.
 Furthermore $P_1$ conforms to $z_1$, and similarly $P_2$
  to $z_2$. Iterating this argument, the decomposition $z = z_1 \cdots z_r$ into
  indecomposables gives a decomposition $P_1 \sqcup \cdots \sqcup P_r$ of $P$
  into indecomposable preorders. Thus $\mathbf c(P) = \mathbf c\big(\ctop(z)\big)$, which proves a.

  \medskip
  \noindent b. Suppose $a$ and $b$ are not comparable in $\pre(z)$.
  Let $L^a$ be a linear extension of $\pre(z)$ with $a > b$ and
  $L^b$ similarly with $b > a$. Then $\Psi_z \circ \Phi_z(L^a)$ and
  and $\Psi_z \circ \Phi_z(L^b)$ conform to $z$. Whether $a,b$ are comparable
  or not in these preorders, in their meet $a$ and $b$ will not be
  comparable.

  If $a$ and $b$ are in the same bubble in $\pre(z)$, they are in
  the same bubble in every preorder $\succeq \pre(z)$, and so
  in the meet.

  If $a < b$ in $\pre(z)$, in any  $P$ conforming to $z$ we have $a \leq_P b$.
  But letting  $L$ be a linear extension of $\pre(z)$, it is compatible
  with $z$ and $P = \Psi_z \circ \Phi_z(L)$ conforms and is $\succeq \pre(z)$.
  Thus $a \leq_P b$. We cannot have $b \leq_P a$ as $P \preceq L$, and
  so $a <_P b$.

%  Let
%  $B_S^T$ be the total preorder with two bubbles where $S = \downarrow a$
%  and $T = I \backslash S$. 
%  Then $P = \Psi_z \circ \Phi_z(B_S^T)$ conforms to $z$ and $P \preceq
%  B_S^T$. If $a, b$ are incomparable in $P$, the conclusion in part b
%  holds concerning how $a$ and $b$ relate. 
%If $a$ and $b$ are comparable in $P$ we must have 
%strictly $a <_P b$. Analogously we can make a conforming preorder $P^\prime$
%with either $a,b$ incomparable or where strictly $b <_P a$.
%Thus when $a,b$ are incomparable in $\pre(z)$ they also are incomparable
%in the meet of all preorders conforming to $z$. 

%  Suppose $a < b$ in $\pre(z)$. 
%  In any conforming preorder $P$ we then have $a \leq_P b$ since
%  $\pre(z) \preceq P$. The same argument
%  as above exhibits a conforming preorder $P$ with strictly $a <_P b$. 
%The upshot is that $\pre(z)$ is the meet of all conforming preorders.

\medskip
Let $L$ be a linear extension of $\pre(z)$. Then they have the same
set of bubbles. The total preorder $L$ is compatible with $z$.
Also $\pre(z) \preceq \Psi_z \circ \Phi_z(L) \preceq L$. So
by Proposition \ref{pro:faces-egp}
the dimension
of the face $F = \Phi_z(L)$ is $|I| - |{\mathbf b}(\pre(z))|$.

In a convex polyhedron all minimal faces have the same dimension.
All minimal $P$ conforming to $z$ correspond to minimal faces and
so by Proposition \ref{pro:faces-egp} all such $|I|- |{\mathbf b}(P)|$
are equal and, by the above, are $\leq |I| - |{\mathbf b}(\pre(z))|$. This gives
$|{\mathbf b}(P)| \geq |{\mathbf b}(\pre(z))|$.
But $\pre(z) \preceq P$, so ${\mathbf b}(\pre(z))$ is a refinement
of ${\mathbf b}(P)$. Whence these sets of bubbles are equal.

\medskip
\noindent c. When $z$ is finite, $\pre(z)$ is the discrete poset.
\end{proof}

%%%

\subsection{The normal fan}

\begin{proposition}
   Let $z : \pow(I) \pil \hat{\RR}$ be a submodular function.
   If the face $F$ of $\Pi(z)$ corresponds to the preorder $P$,
   the normal cone associated to $F$ is $k(P)$.
 \end{proposition}

 \begin{proof}
   Given a direction $y \in k(P)$, we also have $y \in k(L)$ for 
   some linear extension $L$ of $P$. Let
   \[ C_1 < C_2 < \cdots < C_r \]
   be the total ordering of the bubbles of $L$.
   Then
   \[ y_{1} \geq y_{2} \geq \cdots \geq y_{r}, \]
   where
     $y_{i}$ is the value of $y$ at the coordinate positions of $C_i$. 
   The value of $y$ on a point $x \in \egp(z)$ is
   \[ \sum_{i = 1}^r y_{i} x_{C_i} \quad (\text{where } x_{C_i}
     = \sum_{j \in C_i} x_j). \]
   Let $D_i = \dai{L}{C_i}$ be the down-set in $L$ generated by $C_i$.
   We may write the above expression as
   \[ y_{r}x_{D_r} + (y_{{r-1}} - y_{r})x_{D_{r-1}} +
       \cdots + (y_{1} - y_{2}) x_{D_1}. \]
     Each $D_j$ is also a down-set of $P$. We have $x_{D_i} \leq z(D_i)$ and
     on the face $F = \Phi_z(P)$ we have $x_{D_i} = z(D_i)$. This shows
     that the above expression is maximal on the face $F$.

     \medskip
     Now suppose $y$ is not in $k(P)$. Define an equivalence relation on
     $I$ by $i \sim j$ if $y_i = y_j$. The equivalence classes may be
     totally ordered according to the values of the $y_i$, so as
     to become the bubbles of a total preorder $L$. Let $G = \Phi_z(L)$.
     By the argument
     above this is the face of $\egp(z)$ such that $y$ attains its
     maximum. If $F \subseteq G$, then
     \[ P = \Psi_z(F) \preceq \Psi_z(G) \preceq L. \]
  But then $y \in k(L) \sus k(P)$, a contradiction. So $G$
     does not contain the face $F$, and $y$ is not
     maximal on $F$.
   \end{proof}
  
\subsection{Finite submodular functions}

Given a {\it finite}  submodular function $z : \pow(I) \pil \RR$,
note that any total order on $I$ is  compatible with $z$. 
For each minimal $P$ in $\Pre(z)$ (note that this will be a poset),
let $L(P)$ be the set of its linear extensions.

\begin{proposition} For a finite-valued $z$,
  if $P,Q$ are distinct posets of $\Pre(z)$,
  the sets of linear extensions $L(P)$ and $L(Q)$ are disjoint. 
  \end{proposition}

\begin{proof} Let $L$ be a total order with 
  $P \preceq L$. Then
  \[ P = \Psi_z \circ \Phi_z (P) \preceq \Psi_z \circ \Phi_z (L) \preceq L. \]
  But then $\Psi_z \circ \Phi_z(L)$ is a poset (since $L$ is), and so
  is minimal in $\Pre(z)$. Then $P = \Psi_z \circ \Phi_z(L)$ and is
  uniquely determined.
\end{proof}
\begin{corollary} The union
  \[ \bigsqcup_{P \text{{\rm { poset in }} } \Pre(z)}L(P) \]
  is a partition of the set of total orders on $I$. 
\end{corollary}

\begin{remark} \label{rem:bij-linext}
  The set of minimal elements of $\Pre(z)$ then is a complete fan
  of posets, as defined in \cite{PRW2009}. In \cite{MW2009} such
  partitions of linear extensions or equivalently of permutations,
  are called convex rank tests. In particular \cite[Sec.6]{MW2009} counts
  such linear extensions.
\end{remark}

\begin{remark}
%While we have $\Phi_z\circ \Psi_z=\mop{Id}_{\Gamma(z)}$,
The composition
\begin{eqnarray*}
\Psi_z\circ \Phi_z: \mathcal P\big(\mathcal P(I)\big) &\longrightarrow& \mathcal P\big(\mathcal P(I)\big) \\
\mathcal S&\longmapsto & \overline {\mathcal S}
\end{eqnarray*}
is a closure operation giving a topology
$\overline {\mathcal S}$ on $\mathcal P(I)$.
This topology depends however on $z$, and is in general finer than the
ordinary closure topology generated by all unions of all finite intersections
of elements in $\cS$.
\end{remark}
%for which the closed sets are precisely the elements $\mathcal S\in\mathcal P\big(\mathcal P(I)\big)$ such that $\overline {\mathcal S}=\mathcal S$, i.e. the elements of $\Pcon(z)$.
%%%%%

%%%%%%%%%%%%%%%%%%%%%%%%%%%%%%%%%%%%%%%%%%%%%%%
%% Examples of submodular systems

\section{Examples}
\label{sec:exa}
We consider three classes of submodular functions and describe the preorders
which conform to these functions. These classes are i) the functions
$\low_P$ associated to preorders, ii) submodular functions of matroids,
and iii) those of building sets. In the latter case it is well known
\cite[Section 22]{AA2017} and \cite{FeSt2004} that the faces of the associated
generalized permutahedron, the nestohedron, are in bijection to $\cB$-forests.
We show these are precisely the conforming preorders.

%%%%%
\subsection{Preorders}\label{smqp}
%%%
Recall from Paragraph \ref{subsec:pre-bf} that any submodular function $z$ on a finite set $I$ gives rise to a topology $\top(z)$ for which the open sets are the subsets $A\subseteq I$ such that $z(A)< \infty$. The restriction of this map to $\{0, \infty\}$-valued submodular functions is a bijection \cite[Theorem 15.9]{AA2017}. Identifying the topologies on $I$ with preorders via the Alexandroff correspondence, the inverse map is given by $P\mapsto\mop{low}_P$, with
\[\mop{low}_P(A)=0 \hbox{ if } A \hbox{ is a downset for }P, \hbox{ and } \hbox{low}_P(A)=\infty \hbox { if not}.\]

The extended generalized permutahedron associated to $\mop{low}_P$ is
the \textsl{preorder cone} of $P$ (terminology from \cite[Sec.15]{AA2017}),
given by
\[\egp(\low_P) =\left\{x\in\mathbb R^I, \,\, \sum_{i\in I}x_i=0 \hbox{ and }\sum_{i\in A}x_i\le 0 \hbox{ for any $P$-downset }A\right\}.\]

The associated normal fan of $\egp(\low_P)$ is the cone $k(P)$ (and its faces)
of Proposition \ref{pro:braid-cone}, see Proposition \ref{pro:mod-desc}.
The conforming preorders of $\low_P$, which is $\overline{P}$, 
correspond to the faces of $k(P)$.
%and this $\overline{P}$. 

\subsection{Matroids} \label{subsec:exa-mat}
%%%
Let $M$ be a matroid on the finite set $I$, \cite{Ox2006}. The rank function
$z_M : \mathcal P(I) \pil \RR$, which associates to any $A\subset I$ the maximal cardinality of any independent set contained in $A$, is submodular. The
set of vertices of the corresponding generalized permutahedron is in bijection  with the
set of bases of the matroid. For a basis $B \sus I$, the corresponding vertex
is $e_B = \sum_{b \in B} e_b$, where $e_b$ is the $b$'th coordinate vector.
\begin{definition}
A \textsl{circuit} of the matroid $M$ is a minimal subset $C\subseteq I$ such that $z_M(C)<|C|$.
\end{definition}
\noindent
In particular, a circuit is never independent.
\begin{lemma}
Let $B$ be a basis of $M$, and let $c\in I\setminus B$. There is a unique circuit $C\subseteq B\sqcup\{c\}$, which necessarily contains $c$.
\end{lemma}\label{circuit}
\begin{proof}
  If there are two distinct circuits $C_1, C_2\sus B \cup \{c\}$, then
  none of them are contained in $B$. Thus $c \in C_1 \cap C_2$. By
the axioms for circuits \cite[1.1.3]{Ox2006}, there is then a circuit in
  $(C_1 \cup C_2) \backslash \{c \}$. But such a circuit would be contained
  in $B$, which is impossible.
  \end{proof}
  
%We clearly have $z_M(B\sqcup \{c\})=z_M(B)<|B\sqcup \{c\}|$. For any $C',C''$ contained in $B\sqcup\{c\}$, and with rank strictly smaller that the cardinal (therefore necessarily containing $\{c\}$), we have by submodularity
%\begin{eqnarray*}
%z_M(C'\cap C'')&\le &z_M(C')+z_M(C'')-z_M(C'\cup C'')\\
%&\le & |C'|+|C''|-2-(|C'\cup C''|-1)\\
%&\le& |C'\cap C''|-1.
%\end{eqnarray*}
%The circuit $C$ is therefore given by the intersection of all subsets of $B\sqcup\{c\}$ whose cardinal strictly exceeds the rank.

  \noindent Let us denote by $P_B$ be the conforming poset corresponding to the
  vertex $e_B$, namely
\[P_B=\Psi_{z_M}(\{e_B\}).\]
Let $L$ be the set of loops of the matroid (the elements of $I$ of rank zero).
Here is a complete description of the poset $P_B$:

\begin{proposition} \label{pro:exa-mat}
  Given a matroid $M$ whose submodular function is
  $z_M$, and let $B$ be a basis for $M$.
  \begin{itemize}
    \item[a.] The set of minimal elements of $P_B$
      is $B \cup L$.
      \item[b.] If $c \in I \backslash B$, let $C \sus B \cup \{c\}$ be
  the unique circuit given by Lemma \ref{circuit}. Then $c \geq d$ in $P_B$
  if and only if $d \in C$.
  \item[c.] 
    The poset $P_B$ has rank one or zero.
    \item[d.] The preorders $Q$ with $P_B \btl Q$ correspond
      precisely to the faces of the matroid polytope which contain
     the vertex $e_B$.  
    \item[e.] If $Q$ is any preorder conforming to $z_M$, 
      then $Q$ has rank one or zero.
    
\end{itemize}
\end{proposition}

\begin{proof} We fix a basis $B$ and write $x = e_B$, a 
  vertex of the matroid polytope. It corresponds to a conforming
  poset $P_B$ for the submodular function $z_M$.  

  \noindent a.
  %Let the poset $P_B$ correspond to the vertex $x = e_B$ of the matroid
  %polytope. 
  Let $a$ be a minimal element of $P_B$. This holds if and only if $\{a\}$ is
  a down-set of $P_B$, or if and only if $z_M(\{a\}) = x_a$. If $a$ is a loop, both of these
  are zero. If $a$ is an element of $B$, both are $1$, and if $a$ is neither
  a loop nor in $B$, then $z_M(\{a\}) = 1$ and $x_a = 0$, which proves
  part a.\\

  \medskip
  \noindent b. Let $c \in I \backslash B$.
  The coordinate $x_c = 0$, while $x_b= 1$ for $b \in
  C \backslash \{c\}$. Thus $\sum_{b \in C} x_b = |C|-1$ and as $C$
  is a circuit, this is $z_M(C)$. So $C$
  is a down-set of $P$. Thus if $c \geq d$
  in $P_B$ then $d \in C$. On the other hand, suppose for some $d \in C$ that
  we do not have $d \leq c$. Since, from part a,
  all elements in $C \backslash \{c\}$
  are minimal elements,  $C \backslash \{d \}$ would also
  be a down-set for $P_B$. Hence we should have
  \[ z_M(C \backslash \{d\}) = \sum_{b \in C\backslash \{d\}} x_b. \]
  Since $C$ is a circuit, $C \backslash \{d \}$ is independent and the
  left side is $|C| -1$.
But as $x_c = 0$, it is seen that the right side
is $|C|-2$. Hence we cannot avoid having $d \leq c$ for $d \in C$.

\medskip
\noindent c. As all elements of $C \backslash \{c\}$ are minimal, it is clear that
  the height of $P_B$ is one or zero (the latter if and only if any such $c$ is minimal,
  i.e. $c \in L$, and so $I = B \cup L$), which proves c.
  
  Finally that $Q$ conforms to $z_M$ and corresponds to a face $F$ of
  the matroid polytope containing
  $e_B$,  means precisely that $P_B \btl Q$. Whence by c. also the rank
  of $Q$ is one or zero. 
%  Finally, as $P_B \btl Q$ for some $B$, d. follows from c.
\end{proof}
%%%

%%%%%%%%%%%%%%%%%%%%%%%%%%%%%%%
%% Grafen som gir matroide

\begin{figure}
\begin{center}
%Første bilde
\begin{tikzpicture}[scale=1.2, vertices/.style={draw, fill=black, circle,
inner sep=1.5pt}]

\coordinate (a) at (0,0);
\coordinate (d) at (0,2);
\coordinate (b) at (2,0);
\coordinate (c) at (2,2);
\coordinate (x) at (1,0);
\coordinate (y) at (2,1);
\coordinate (z) at (1,2);
\coordinate (w) at (0,1);
\coordinate (u) at (1,1);

\draw [help lines, white] (-2,-1) grid (3,2);
 %Make axes:
        \draw[black, thick] (d)--(a)--(b)--(c)--(d)--(b);
        %\draw[black, thick] (0,0) -- (0,3.5) node [anchor=west] {x_2};
	%\draw[black, thick] (0,0,0) -- (0,0,3.5) node [anchor=west] {z};

\node[anchor=north] at (x) {$a$}; 
\node[ anchor=west] at (y) {$b$};
\node[ anchor=south] at (z) {$c$}; 
\node[ anchor=east] at (w) {$d$};
\node[ anchor=south] at (u) {$e$};

\node[vertices] at (a) {}; 
\node[vertices] at (b) {};
\node[vertices] at (c) {}; 
\node[vertices] at (d) {};

\draw[black, very thick] (a)--(b);

\end{tikzpicture}
\end{center}
 \caption{The graph gives a matroid on its edges}
        \label{fig:mat-graph}
\end{figure}
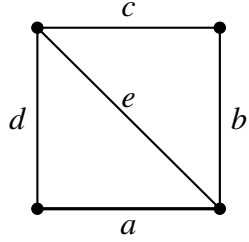

%%%%%%%%%%%%%%%%%%%%%%%%%%%%%%%%%%%%%%%%%%%%
%% Konforme poset til matroiden

\begin{figure}
\begin{center}
%Første bilde
\begin{tikzpicture}[scale=1.2, vertices/.style={draw, fill=black, circle,
inner sep=1.5pt}]

\draw [help lines, white] (0,0) grid (4,2);

\coordinate (a1) at (0,0);
\coordinate (b1) at (0.5,0);
\coordinate (c1) at (1,0);
\coordinate (d1) at (0.25,1);
\coordinate (e1) at (0.75,1);

\coordinate (a2) at (2,0);
\coordinate (b2) at (2.25,1);
\coordinate (c2) at (3,0);
\coordinate (d2) at (2.75,1);
\coordinate (e2) at (2.5,0);

\node[ anchor=north] at (a1) {$a$};
\node[ anchor=north] at (b1) {$b$};
\node[ anchor=north] at (c1) {$c$}; 
\node[ anchor=south] at (d1) {$d$};
\node[ anchor=south] at (e1) {$e$};

\node[ anchor=north] at (a2) {$a$};
\node[ anchor=south] at (b2) {$b$}; 
\node[ anchor=north] at (c2) {$c$};
\node[ anchor=south] at (d2) {$d$}; 
\node[ anchor=north] at (e2) {$e$};

%%%%%%%%%%%%%
\node[vertices] at (a1) {}; 
\node[vertices] at (b1) {};
\node[vertices] at (c1) {}; 
\node[vertices] at (d1) {};
\node[vertices] at (e1) {};

\node[vertices] at (a2) {};
\node[vertices] at (b2) {}; 
\node[vertices] at (c2) {};
\node[vertices] at (d2) {}; 
\node[vertices] at (e2) {};

\draw[black, thick] (a1)--(d1)--(b1)--(e1)--(c1)--(d1);
\draw[black, thick] (a2)--(b2)--(e2)--(d2)--(c2);

\end{tikzpicture}
\end{center}
\caption{ Bases give vertices of the matroid polytope.
  Conforming preorders corresponding to bases $\{a,b,c\}$ and
$\{a,e,c\}$ of Figure \ref{fig:mat-graph}}
        \label{fig:mat-conform}
\end{figure}
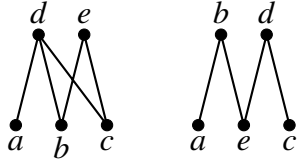

\begin{example} \label{ex:exa-trekant}
  Consider the submodular function $z : \pow(\{a,b,c\}) \pil \NN$
  given by
  \[ z(S) = \begin{cases} 2, & |I| = 1\\ 3, & |I| = 2 \\ 4, & |I| = 3 \end{cases}.
  \]
  Its generalized permutahedron with conforming preorders of vertices
  is given on the
  right of Figure \ref{fig:trekant-conform}. This is the submodular
  function of the building set consisting of the subsets
  $\{a\}, \{b\}, \{c\}, \{a,b,c\}$.

  Consider now the uniform matroid $U_{1,3}$ whose rank (i.e. submodular)
  function is $z(I) = 1$ for $|I| \geq 1$.
  Its generalized permutahedron with conforming preorders of vertices
  is given on the
  left of Figure \ref{fig:trekant-conform}.
\end{example}

%%123
\newcommand{\Aab}{
  \begin{tikzpicture}[scale=0.9, vertices/.style={draw, fill=black, circle, inner sep=1.5pt}]

%\node (31c) at (5.2-1,0) {c.};
\node [vertices] (31) at (-0.5,0) {};
\node [vertices] (32) at (-0.8,0.7) {};
\node [vertices] (33) at (-0.2,0.7) {};

\node[anchor = east] at (31) {$c$ \hskip 0.1pt {}};
\node[anchor = east] at (32) {$b \, $};
\node[anchor = west] at (33) {$\, a$};

\draw (31)--(33);
\draw (31)--(32); 

\end{tikzpicture}
}

\newcommand{\Abc}{
  \begin{tikzpicture}[scale=0.9, vertices/.style={draw, fill=black, circle, inner sep=1.5pt}]

%\node (31c) at (5.2-1,0) {c.};
\node [vertices] (31) at (-0.5,0) {};
\node [vertices] (32) at (-0.8, 0.7) {};
\node [vertices] (33) at (-0.2, 0.7) {};

\node[anchor = east] at (31) {$a$ \hskip 0.1pt {}};
\node[anchor = east] at (32) {$c\, $};
\node[anchor = west] at (33) {$\, b $};

\draw (31)--(33);
\draw (31)--(32);

\end{tikzpicture}
}

\newcommand{\Aca}{
  \begin{tikzpicture}[scale=0.9, vertices/.style={draw, fill=black, circle, inner sep=1.5pt}]

%\node (31c) at (5.2-1,0) {c.};
\node [vertices] (31) at (-0.5,0) {};
\node [vertices] (32) at (-0.8, 0.7) {};
\node [vertices] (33) at (-0.2, 0.7) {};

\node[anchor = east] at (31) {$b$ \hskip 0.1pt {}};
\node[anchor = east] at (32) {$a \,${}};
\node[anchor = west] at (33) {$\, c $};

\draw (31)--(33);
\draw (31)--(32); 

\end{tikzpicture}
}

%%%%%%%%%%%%%%%%%%%%%%%%%%%%%%%%%%%%%%%%%%%%%%%%%%%
%% Hexagon med konforme preordener

\begin{figure}
%\begin{center}
%Første bilde
\begin{tikzpicture}[inner sep=1pt, scale=0.8]

\coordinate (a) at (-2,0) ; % [anchor =  north east] {$(1,2,3)$};
\coordinate (b) at (2, 0);  %[anchor = north west] {$(2,1,3)$};
\coordinate (c) at (0, 3.46);  %[anchor = north east] {$(3,1,2)$};

    %Draw permutohedron: 
            \draw[black, thick, fill=tol1, fill opacity = 0.1] 
            (a) -- (b)  -- (c) -- (a);
            %Draw centerpoint and vertices:
            %\draw[black, fill] (0.5,)    circle (1pt);
            \draw[black, fill] (a)     circle (2pt);
            \draw[black, fill] (b)     circle (2pt);
            \draw[black, fill] (c)     circle (2pt);
         
            % \node[] at (-0.5, 0.8660)   [anchor = east]       { $f$};

\node[] at (a) [anchor = east] {$\Abc$ \hskip 5mm {}};
\node[] at (b) [anchor = west] {\hskip 3mm $\Aab$};
\node[] at (c) [anchor = south west] {$\Aca$ \hskip 5mm {}};

\node[] at (a) [anchor = north] {\scriptsize{${\overset{} {{(1,0,0)}}}$}};
\node[] at (b) [anchor = north] {\scriptsize{${\overset{} {{(0,0,1)}}}$}};
\node[] at (c) [anchor = south east] {\scriptsize{$(0,1,0)$}};

%\end{tikzpicture}
 %        \end{center}
 %       \caption{Conforming preorders \\ for vertices of matroid polytope}
 %       \label{fig:toperm-conform}
%      \end{figure}

%  \begin{figure}
%\begin{center}
%Første bilde
%\begin{tikzpicture}[inner sep=1pt, scale=0.8]

\coordinate (a) at (7,0) ; % [anchor =  north east] {$(1,2,3)$};
\coordinate (b) at (11, 0);  %[anchor = north west] {$(2,1,3)$};
\coordinate (c) at (9, 3.46);  %[anchor = north east] {$(3,1,2)$};

    %Draw permutohedron: 
            \draw[black, thick, fill=tol1, fill opacity = 0.1] 
            (a) -- (b)  -- (c) -- (a);
            %Draw centerpoint and vertices:
            %\draw[black, fill] (0.5,)    circle (1pt);
            \draw[black, fill] (a)     circle (2pt);
            \draw[black, fill] (b)     circle (2pt);
            \draw[black, fill] (c)     circle (2pt);
         
            % \node[] at (-0.5, 0.8660)   [anchor = east]       { $f$};

\node[] at (a) [anchor = east] {$\Abc$ \hskip 5mm {}};
\node[] at (b) [anchor = west] {\hskip 3mm $\Aab$};
\node[] at (c) [anchor = south west] {$\Aca$ \hskip 5mm {}};

\node[] at (a) [anchor = north] {\scriptsize{${\overset{} {{(2,1,1)}}}$}};
\node[] at (b) [anchor = north] {\scriptsize{${\overset{} {{(1,1,2)}}}$}};
\node[] at (c) [anchor = south east] {\scriptsize{$(1,2,1)$}};

\end{tikzpicture}
% \end{center}
\caption{Conforming preorders for matroid and building set of
Example \ref{ex:exa-trekant}}
        \label{fig:trekant-conform}
    \end{figure}    

For a matroid all conforming preorders have height $\leq 1$
by Proposition \ref{pro:exa-mat}
The example above suggests the following.

\begin{problem} Let $w : \pow(I) \pil \RR$ be a submodular function
  such that all conforming preorders have height $\leq 1$. Is there
  then a matroid with rank function $z$ such that $\Pre(z) = \Pre(w)$?
\end{problem}

\subsection{Building sets, nestohedra, and $\cB$-forests}
Recall the notion of building set from Subsection \ref{subsec:egp-building}.
The generalized permutahedron $\Delta_{\cB}$ associated to
the building set $\cB$ is the Minkowski sum
\[ \DB = \sum_{J \in \cB} \Delta_J.\]
%By \cite{AA2017}, it  corresponds to the submodular function
%\[ z(K) = | \{ J \in \CB \, | \, J \cap K \neq \emptyset \}|. \]
An element of the building permutahedron $\Delta_\cB$ may be written
$p = \sum_{J \in \cB} \alpha^J$, where $\alpha^J$ is a point in $\Delta_J$.

%For a subset $A \sus I$ {and for any linear form $t$ on $\RR I$}, write $\fone_A$
%for the linear form restricted to $\RR A$. %$\sum_{a \in A} t_a$.
For $A \sus I$ the value of the linear form $\fone_A$
(see Section \ref{sec:egpclass}) on $p$ is
\begin{equation} \label{eq:build-xa}
  \fone_A(p) = \sum_{J \in \cB} \fone_A(\alpha^J) = \sum_{J \cap A \neq \emptyset}
  \fone_A(\alpha^J).
  \end{equation}
Let $z$ be the associated submodular function of the generalized
permutahedron. By \cite[Sec. 12]{AA2017} the value $z(A)$ is the
maximum value of this for $p$ on the generalized permutahedron.
If $A \cap J$ is non-empty, let $a_J$ be an element.
Letting $\alpha_J = e_{a_J}$, we see that the maximum
value of \eqref{eq:build-xa} then becomes:
\[ z(A) = \sum_{ J \cap A \neq \emptyset} 1 =
  | \{ J \in \cB \, | \, J \cap A \neq \emptyset \}|. \]

%\subsection{Faces of building permutahedrons}

The faces of a $\DB$ are in bijection with {\it nested sets} or equivalently
with $\cB$-{\it forests}, by \cite{FeSt2004, Pos2009}, 
(recall these notions from Subsection
\ref{subsec:egp-building}).
Let us explain how
a $\cB$-forest gives a face of $\DB$. We need first the following.

\begin{lemma}
  Let $\cB$ be a building set on $I$, and $N$ a $\cB$-forest on $I$.
  Let $J \sus I$ be an element of $\cB$. Considering $N$ as a preorder,
  the set of minimal elements $\min_N J$ is a subset of a bubble of $N$, i.e.
  the restriction of the preorder to $\min_N J$ is the coarse preorder.
\end{lemma}

\begin{proof}
We may write $\min_N J$ as a disjoint union $C_1 \sqcup \cdots \sqcup C_r$
where the $C_i$ are coarse preorders, and they are pairwise incomparable.
We show that $r = 1$. 
Consider $\uparrow_N C_1$ which, by definition of $\cB$-forest, is in $\cB$.
By the axiom for building set, $J \, \cup \uparrow_N C_1$ is in $\cB$.
Continuing we get that $J \,  \cup_{i=1}^r \uparrow_N C_i$ is in $\cB$. But this
latter is simply  $\cup_{i=1}^r \uparrow_N C_i$. If $r \geq 2$, this is
not in $\cB$ by definition of $N$.
\end{proof}

Given the building permutahedron $\Delta_\cB = \sum_{J \in \cB} \Delta_J$,
by the proof of
\cite[Thm.7.3]{Pos2009} the face associated to $N$ is
\[ F_N = \sum_{J \in B} \Delta_{\min_NJ}. \]
We now show that the preorder associated to $F_N$ by the functor $\Psi_z$
is precisely the preorder $N$.

\begin{proposition}
  If $A$ is a down-set of $N$, for every $p$ in the face $F_N$, equality $\fone_A(p) = z(A)$ holds.
 If $B$ is not a down-set of $N$, there is $p \in F_N$ such that
 $\fone_B(p) < z(B)$. As a consequence the preorder associated to the face
 $F_N$ by $\Psi_z$ is
  precisely $N$. 
\end{proposition}

\begin{proof} Let $J \sus I$ be an element of $\cB$.
  If $J \cap A \neq \emptyset$, then
  $\min _N J \sus A$ as $A$ is a down-set for $N$ and $\min_N J$ is
  a coarse preorder.
  For
  \[ p = \sum_{J \cap A \neq \emptyset} \alpha^{\min_N J}
    \in F_N = \sum_{J \cap A \neq \emptyset} \Delta_{\min_N J}, \]
  we therefore see that
  \[ \fone_A(p) = \sum_{J \cap A \neq \emptyset}
    \fone_A(\alpha^{\min_N J})= \sum_{J \cap A \neq \emptyset} 1 = z(A). \]
  \medskip
  As $B$ is not a down-set of $N$, there is $q \in B$ and $p \leq_N q$ with
  $p \not \in B$. Consider $J_0 = \uparrow_N p \in \cB$.
  Then $J_0 \cap B \neq \emptyset$,
  and $p \in \min_N J_0$ and so $e_p \in \Delta_{\min_N J_0}$.
  Let
  \[ r = e_p + \sum_{\overset{J \cap B \neq \emptyset}{J \neq J_0}} u_j
    \in \Delta_{\min_N J_0} + \sum_{\overset{J \cap B \neq \emptyset}{J \neq J_0}}
    \Delta_{\min_N J} \]
  be an element of $F_N$. As $\fone_B(e_p) = 0$ and $\fone_B(u_J) \leq 1$,
  we see that
  $\fone_B(r) < z(B)$.
  \end{proof}

  %%%

%%%%%%%%%%%%%%%%%%%%%%%%%%%%%%%%%%%%%%%%%%%%%%%%%
  %% Modular functions

\newcommand{\pphat}[1]{\widehat{#1}}
\newcommand{\nfan}{\mathcal{N}}
\newcommand{\modu}{\rm{mod}}

\section{Modular functions}
\label{sec:mod}

The class of modular functions $z$ on $I$, i.e. those such that we have
equality in \eqref{eq:intro-submod}, is particularly nice.
The EGP $\egp(z)$ is then a cone, a translate of $\egp(\low_P)$.
Moreover, to a preorder $P$ conforming to $z$ we
associated a modular function $z^P$.  This is important for the construction of cointeracting bialgebras in
Section \ref{sec:bim}. We show that the
cone $\egp(z^P)$ is the cone of $\egp(z)$ at
the face corresponding to $P$.

\medskip
A function $z : \pow(I) \pil \hRR$ is {\it modular} iff
for every $A, B \sus I$:
\begin{equation} \label{eq:mod-def}
  z(A) + z(B) = z(A \cap B) + z(A \cup B).
  \end{equation}
Recall $\top(z)$ is the associated topology, whose open sets are the 
$A \sus I$ with $z(A)$ finite. It corresponds to a preorder $P = \pre(z)$
and $z(A)$ is finite iff $A$ is a down-set of $\pre(z)$.
For a preorder $P$ on $I$ let $\widehat{P} \sus \pow(I)$ be the down-sets
of $P$. A modular function can then be considered a function
$\widehat{P} \pil \RR$ such that the above equation \eqref{eq:mod-def} holds
for all $A,B \in \widehat{P}$. 

\subsection{Characterization of modular functions}
For the modular function $z$, each bubble $C \in {\mathbf b}(P)$,
with  $P = \pre(z)$ generates
a downset $\dai{P}C$. We then get a function
\[ v : {\mathbf b}(P) \pil \RR, \quad v(C) = z(\dai P C). \]

Given a preorder $P$, let $\modu[P]$ be the set of modular functions $z$ with
$P = \pre(z)$. The above gives a map
{
\begin{equation} \label{eq:mod-bpr}
  \modu[P] \pil \RR^{\mathbf b(P)}.
\end{equation}
}

\begin{proposition}
The association \eqref{eq:mod-bpr} above 
is a bijection.
In particular if $P$ is a poset (which is equivalent to $\mathbf b(P) = I$),
modular functions in $\modu[P]$ are in bijection with $\RR^I$. {Modular} functions may thus be given by pairs $(P,{\mathbf a})$
 where $P$ is a preorder and ${\mathbf a} \in \RR^{\mathbf b(P)}$.
\end{proposition}

\begin{proof}
  Given a map $v : \mathbf b(P) \pil \hRR$ we show there is a uniquely defined
  modular function $z : \pphat{P} \pil \RR$ which gives $v$ by
  the correspondence
  above. Let $D$ be a downset of $P$. We define inductively $z(D)$ by
  the size of $D$. Clearly minimal $D$'s are minimal bubbles, and 
  $z(D) = v(D)$. If $D$ is not of the form $\dai P C$ for a bubble $C$, we may
  write $D = A_1 \cup B_1$ for two strictly smaller down-sets. Since
  we must have
  \[ z(A_1 \cup B_1) + z(A_1 \cap B_1) = z(A_1) + z(B_1), \]
  this determines $z(D)$. But we must show this is independent of the
  decomposition. So let $D = A_2 \cup B_2$ be another decomposition.
  Then $z(D)$ is defined respectively as:
  \begin{equation} \label{eq:mod-zD}
    z(A_1) + z(B_1) - z(A_1 \cap B_1), \quad
    z(A_2) + z(B_2) - z(A_2 \cap B_2).
    \end{equation} 
  Expanding the first expression in \eqref{eq:mod-zD} by:
  \[ A_1 = (A_1 \cap A_2) \cup (A_1 \cap B_2), \quad
    B_1 = (B_1 \cap B_2) \cup (B_1 \cap A_2), \]
  we get
  \begin{align*}
    z(A_1 \cap A_2) + z(A_1 \cap B_2) & -  z(A_1 \cap A_2 \cap B_2)
    + z(B_1 \cap A_2) + z(B_1 \cap B_2) - z(B_1 \cap B_2 \cap A_2) \\
                                      & - z(A_1 \cap B_1 \cap A_2) -
                                        z(A_1 \cap B_1 \cap B_2) + z(A_1 \cap A_2 \cap B_1 \cap B_2).
  \end{align*}
  Expanding the second expression of \eqref{eq:mod-zD}, it is seen, by symmetry,
  to give the same as above.
\end{proof}

\begin{remark} The submodular function $\low_P$ is modular and
  this corresponds to the zero map 
  $v : {\mathbf b}(P) \pil \RR$.
  There is a map
  \[ \RR I \times {\mathbf b}(P) \pil \RR, \quad
    (x,C) \mapsto  x_C = \sum_{i \in C} x_i. \]
  This gives a map $b : \RR I \pil \RR^{{\mathbf b}(P)}$. Given
  an arbitrary $v \in  \RR^{{\mathbf b}(P)}$, let $x$ be in the preimage
  of $v$ by the map $b$. For the modular function $z$ corresponding
  to $v$, the extended generalized permutahedron $\Pi(z)$
  is the cone $\egp(\low_P)$ translated by the vector $x$.
\end{remark}

\begin{proposition} \label{pro:mod-desc}
  The following are equivalent for a submodular function $z$:
  \begin{itemize}
  \item[a.] $z$ is modular,
  \item[b.] $\pre(z)$ conforms to $z$ (so it is the unique minimal element
    in $\Pre(z)$),
  \item[c.] The normal fan of $\Pi(z)$ has a single maximal cone,
    which is $k\big(\pre(z)\big)$,
  \item[d.] $\Pi(z)$ has a unique minimal face. Then this minimal face
    is an affine space. 
  \end{itemize}
\end{proposition}

\begin{proof}
  Suppose $P = \pre(z)$ conforms to $z$. By 
  Proposition \ref{pro:conf-minmax}.b
  this is equivalent to $\Pre(z)$ having a unique minimal
  element and thus $\Pi(z)$ a unique minimal face. Minimal faces of
  $\Pi(z)$ correspond to maximal cones in the normal fan.
  This shows the equivalence of $b,c$ and $d$.
  Furthermore in case d, the minimal face is $\alin(P) \cap \egp(z)$
  where $P = \pre(z)$. 
  But if $L$ is a linear extension of $\pre(z)$, by Lemma \ref{lem:conf-alinbu}
  $\alin(P) = \alin(L)$ and by Lemma \ref{lem:egp-lin} this is
  $\sus \egp(z)$.

  Let us show that these imply a.
  The submodular function  $z$ gives a function
  $v : \mathbf b(P) \pil \hRR$ by $v(C) = z(\dai P C)$.
  This $v$ induces again a modular function $z^\prime$. We show
  that $z = z^\prime$ by induction of the size of the function argument.
  Let $A$ be a down-set of $P = \pre(z)$ which is
  not principal, i.e. not $\dai P C$ for some bubble $C$. Let
  $u,v$ be maximal bubbles in $A$ and $A_u = A \backslash \{v \}$,
  $A_v = A \backslash \{u\}$. By induction
  $z(A_u) = z^\prime(A_u)$ and $z(A_v) = z^\prime(A_v)$. Let
  $B = A_u \cap A_v = A\backslash \{u,v\}$. This is a down-set and since
  $A \backslash B = \{u,v \}$ is disconnected, and $P$ conforms to
  $z$, we have
  \[ z(A) + z(B) = z(A_u) + z(A_v). \] This shows $z(A) = z^\prime(A)$.
  So $z$ is modular, showing part a.

  \medskip
  Suppose then $z$ is modular. We show $P = \pre(z)$ conforms to $z$.
  Let $B,A$ down-sets in $P$,
  let $C = B \backslash A$ and suppose
  \[ z^\prime = z_C = z^\prime_1 \cdot z^\prime_2 \] where
  $C = C_1 \cup C_2$ is a disjoint union, and $z^\prime_i$ is defined on $C_i$.
  By Lemma \ref{lem:prod-sum}:
  \[ z(B) + z(A) = z(A \cup C_1) + z(A \cup C_2). \]
  Then $A \cup C_1$ and $A \cup C_2$ give finite values, and so are
  down-sets in $\pre(z)$. Whence
  $C = C_1 \sqcup C_2$ is a disconnected union. The converse direction
  is also clear. 
\end{proof}

\begin{remark}
R.Stanley \cite{Sta1986} considers the order polytope which is
the cone $k(P)$ intersected with the box $0 \leq y_i \leq 1$.
He shows that the faces of $k(P)$ related by inclusion
correspond to $\underline{P}$ ordered by $\lhd$.
This is also discussed in \cite[Sec.15]{AA2017}
\end{remark}

\begin{lemma} \label{lem:mod-res}
  If $z$ is modular, and $z(S) < \infty$, then
  $z\restr{S}$ and $z_{/S}$ are modular.
\end{lemma}

\begin{proof}
  This is straightforward to verify.
  \end{proof}

\subsection{The cone submodular function}

\begin{definition} For a preorder $P$ compatible with the submodular
  function $z$, 
let $z^P$ be the submodular function
\[ z^P(A) = \begin{cases} z(A), & A \text{ down-set of }P, \\ 
 \infty, & \text {if not.}
\end{cases} \]  
\end{definition}

\medskip
If $S$ is a down-set of a preorder $P$ compatible with $z$,
let $P\restr{S}$ be the restriction of $P$ to $S$, and $P_{/S}$ the restriction of
$P$ to the up-set $I \backslash S$.
If $C$ is a convex subset of $P$ and $C = B \backslash A$ where
$A \sus B$ are down-sets of $P$, we write $P_{|C} := (P\restr{B})_{/A}$. 

\begin{lemma} \label{lem:ext-ps}
   Let $P$ be compatible with $z$, and $S$ a down-set of $P$. Then:
   \begin{itemize}
   \item[a.] $(z^P)\restr{S} = (z\restr{S})^{P\srestr{S}}$,
   \item[b.] $(z^P)_{/S} = (z_{/S})^{P_{/S}}$,
   \end{itemize}
 \end{lemma}

 \begin{proof} We show part b.
 For $E \sus I \backslash S$ we have:
 \[ (z^P)_{/S}(E) = \begin{cases} z(E \cup S) - z(S), & E \cup S
     \text{ down-set of } P \\
     \infty, & \text{ otherwise. }
   \end{cases} \]
 We also have
 $(z_{/S})^{P/S}(E) = z_{/S}(E)$ for $E$ a down-set of $P/S$, which
 is the same as $E \cup S$ being a down-set of $P$. This gives:
  \[ (z_{/S})^{P/S}(E) = \begin{cases} z(E \cup S) - z(S), & E \cup S
     \text{ down-set of } P \\
     \infty, & \text{ otherwise. }
   \end{cases} \]
 Part a is shown similarly.
\end{proof}

\begin{proposition} \label{pro:mod-zp}
  Let $z$ be a submodular function, and $P$ a preorder compatible with $z$.

  \begin{itemize}
  \item[a.] The following statements hold:
  \begin{itemize}
    \item[i.] $P$ conforms to $z^P$,
    \item[ii.] $z^P$ is modular,
    \item[iii.] $\Pre(z^P) = \overline{P}$.
    \end{itemize}
  \item[b.] If $z$ is modular then $z_P$ is modular
  \end{itemize}
\end{proposition}

\begin{proof}
\noindent a.i.  Let $A \sus B$ be down-sets of $P$ and $C = B \backslash A$.
  If $C = C_1 \sqcup C_2$ is disconnected, then 
  \begin{equation}
    z(B) + z(A) = z(A \cup C_1) + z(A \cup C_2)
    \end{equation}
  {(by Lemma \ref{lem:prod-sum})}, and then we also have
  \begin{equation} \label{eq:mod-mod}
    z^P(B) + z^P(A) = z^P(A \cup C_1) + z^P(A \cup C_2),
    \end{equation}
   Conversely if this last equality holds, then $A \cup C_1$ and $A \cup C_2$
   are both down-sets of $P$ and so $C = C_1 \sqcup C_2$ is a disconnected
   union. Whence $P$ conforms to $z^P$.

  a.ii. Let us argue that $z^P$ is modular. Let $E,F \sus I$ with
   $z^P(E), z^P(F)$ finite. Denote:
   \[ A = E \cap F, \, B = E \cup F, \quad C = B \backslash A,\,
     C_1 = E \backslash A, \, C_2 = F \backslash A. \]
   Since $E,F$ are down-sets of $P$, so are $A$ and $B$. Furthermore,
   $C = C_1 \sqcup C_2$ is then a disconnected union. Whence we can apply
   \eqref{eq:mod-mod}, which
   here becomes:
   \[ z^P(E \cup F) + z^P(E \cap F) = z^P(E) + z^P(F), \]
   and so $z^P$ is modular.

   a.iii. If $Q$ conforms to $z^P$ then $P \preceq Q$. Since $P$
   conforms to $z^P$, Proposition \ref{pro:conf-PbtrQ} gives $P \btl Q$.
   Conversely, if $P \btl Q$, then $P \preceq Q$ and again
  Proposition \ref{pro:conf-PbtrQ} gives $Q$ conforming to $z^P$.

  \medskip
  \noindent b. It is enough to show this for $z_C$ where $C$ is a bubble
  of $P$, since products of modular functions are modular.
  So let $C = B \backslash A$, and let $S$ and $T$ be subsets of $C$.
  As $z$ is modular, we have
  \[ z(A \cup S) + z(A \cup T) = z(A \cup S \cup T)
    + z\big(A \cup (S \cap T)\big).\]
  Since $z(A)$ is finite, this gives
  \[ z_C(S) + z_C(T) = z_C(S \cup T) + z_C(S \cap T). \]
\end{proof}

\begin{proposition} \label{pro:mod-cone}
  Let $P$ conform to the submodular function $z$. 
The extended generalized permutahedron $\egp(z^P)$ is a 
translated cone of dimension $|I| - |\mathbf c(P)|$, the same as the
dimension of $\egp(z)$. The lineality subspace of $\egp(z^P)$ is the affine
    linear space spanned by the face $\egp(z_P)$. 
%    ({\color{blue} must be interpreted suitably}) of dimension $c(P)$.
\end{proposition}

\begin{proof}
Consider first the situation where $P$ has a single
    component. The relations  defining $\egp(z^P)$ are then
    \begin{equation} \label{eq:mod-rel}
      x_I = z(I), \quad x_A \leq z(A), \,\,  A  \text{ down-sets of } P.
      \end{equation}
    Using Proposition \ref{pro:com-BA} one sees that $z(A) = \sum_C z_C(C)$, where we
    sum over all bubbles contained in $A$.
    
Consider again the filtration \eqref{eq:coint-Bfilt} for $B = P$.  For each
bubble $C_i = B_i \backslash B_{i+1}$ of $P$, let $u_{C_i}$ be a point in
$\RR C_i$ whose coordinates sum to $z_{C_i}(C_i)$.
As $I = \sqcup_i C_i$, the sum $u = \sum_i u_{C_i}$ is a point in $\RR I$,
with $j$'th coordinate denoted $u_j$. 
For $S \sus I$ write $u_S = \sum_{i \in S} u_i \in \RR S$. 
Perform the translation $x^\prime = x - u$. The equality and
inequalities listed above in \eqref{eq:mod-rel} then become:
\[ x^\prime_I = 0, \quad x^\prime_A \leq 0, \,  A \text{ down-sets of } P.\]
This is a cone of dimension $|I| - 1$, and so $\egp(z^P)$ is such
a cone suitably translated. The lineality space is the affine space 
where all inequalities are equalities. Hence this is given by the
equations $x^\prime_C = 0$ for each bubble $C$. This translates to the
equations $x_C = z_C(C)$ for each bubble, and this is precisely the
affine linear span of the face $\egp(z_P)$.

\medskip
If $P$ has several components, write
$P = P_1 \sqcup P_2 \sqcup \cdots \sqcup P_r$,
where $r = |\mathbf c(P)|$,
with $P_i$ a preorder whose underlying set is $I_i$. 
We have the product
\[ \egp(z^P) = \egp(z^{P_1}) \times \egp(z^{P_2}) \times
  \cdots \times \egp(z^{P_r}), \]
and so the dimension of this is $\sum_i (|I_i| - 1) = |I| - |\mathbf c(P)|$.
We also have for the face
\[ \egp(z_P) = \egp(z_{P_1}) \times \egp(z_{P_2}) \times
  \cdots \times \egp(z_{P_r})\]
and so the lineality space of the cone $\egp(z^P)$ is the affine linear
space spanned by $\egp(z_P)$.
\end{proof}

The cone which is the translation of $\egp(z^P)$ is described with generators
in \cite[Thm.3.28]{Fuj2005}.

\begin{example}
  The submodular function $z : \pow(\{a,b,c\}) \pil \RR$ of the standard
  two-dimensional permutahedron  is given by
  \[ z(S) = \begin{cases} 3, & |S| = 1, \\ 5, & |S| = 2 \\
      6, & |S| = 3 \end{cases}. \]
  The preorder $P$ which is the linear order $b < c < a$ conforms
  to $z$, and the corresponding face is the vertex $(1,3,2)$.
  Then $z^P$ is the submodular (actually modular) function corresponding to the
  cone at this vertex, Figure \ref{fig:cone-vertex-toperm}.

  The preorder $P$ with bubble $\{b,c\}$ and $\{ b,c \} < a$ corresponds
  to an edge of the permutahedron, and the submodular (actually modular)
  function $z^P$
  gives the cone with affine linear span spanned by this edge,
  Figure  \ref{fig:cone-edge-toperm}.
\end{example}

%% 132

\renewcommand{\bca}{
  \begin{tikzpicture}[scale=0.7, vertices/.style={draw, fill=black, circle, inner sep=1.5pt}]

%\node (31c) at (5.2-1,0) {c.};
\node [vertices] (31) at (-0.5,0) {};
\node [vertices] (32) at (-0.5,0.5) {};
\node [vertices] (33) at (-0.5,1) {};
\node (3p) at (-1, 0.5) {};

\node[anchor = east] at (31) {$b$ \hskip 4mm {}};
\node[anchor = east] at (32) {$c$ \hskip 4mm {}};
\node[anchor = east] at (33) {$a$ \hskip 4mm {}};
\draw (31)--(32);
\draw (31)--(33);
\node[anchor = east] at (3p) {$P\!:  \, $ \hskip 4mm {}};

\end{tikzpicture}
}

%% Kant 12-3
\newcommand{\pbcsa}{
  \begin{tikzpicture}[scale=0.9, vertices/.style={draw, fill=black, circle, inner sep=1.5pt}]

%\node (31c) at (5.2-1,0) {c.};
\node [vertices, gray] (31) at (-0.33,-0.2) {};
\node [vertices, gray] (32) at (-0.67,-0.2) {};
\node [vertices] (33) at (-0.5,0.7) {};
\node (3pp) at (-1, 0.3) {};

\node[anchor = west] at (31) {$\, c$ };
\node[anchor = east] at (32) {$b \,$ };
\node[anchor = east] at (33) {$a$ \hskip 0.1pt {}};

\node (ss3) at (-0.5, -0.2) {};
\draw (ss3) ellipse (0.58 and 0.42);
\node[anchor = east] at (3p) {$P\!: \, \, $ \hskip 4mm {}};

%\fill (dd3) circle (0.05);

\draw (33)--(-0.57, 0.20);

\end{tikzpicture}
}

%% Kant 12-3
\renewcommand{\bcsa}{
  \begin{tikzpicture}[scale=0.9, vertices/.style={draw, fill=black, circle, inner sep=1.5pt}]

%\node (31c) at (5.2-1,0) {c.};
\node [vertices, gray] (31) at (-0.33,-0.2) {};
\node [vertices, gray] (32) at (-0.67,-0.2) {};
\node [vertices] (33) at (-0.5,0.7) {};
%\node (3pp) at (-1, 0.3) {};

\node[anchor = west] at (31) {$\, c$ };
\node[anchor = east] at (32) {$b \,$ };
\node[anchor = east] at (33) {$a$ \hskip 0.1pt {}};

\node (ss3) at (-0.5, -0.2) {};
\draw (ss3) ellipse (0.58 and 0.42);
%\node[anchor = east] at (3p) {$P\!: \, \, $ \hskip 4mm {}};

%\fill (dd3) circle (0.05);

\draw (33)--(-0.57, 0.20);

\end{tikzpicture}
}

%% Kant 13-2
\renewcommand{\bsac}{
  \begin{tikzpicture}[scale=0.9, vertices/.style={draw, fill=black, circle, inner sep=1.5pt}]

%\node (31c) at (5.2-1,0) {c.};
\node [vertices] (31) at (-0.5,-0.2) {};
\node [vertices, gray] (32) at (-0.33,0.7) {};
\node [vertices, gray] (33) at (-0.67,0.7) {};

\node[anchor = east] at (31) {$b$ \hskip 0.1pt {}};
\node[anchor = west] at (32) {$\, c${}};
\node[anchor = east] at (33) {$a\, $};

\node (ss3) at (-0.5, 0.7) {};
\draw (ss3) ellipse (0.58 and 0.42);

\draw (31)--(-0.57, 0.27);

\end{tikzpicture}
}

\renewcommand{\abcss}{
  \begin{tikzpicture}[scale=0.9, vertices/.style={draw, fill=black, circle, inner sep=1.5pt}]

%\node (31c) at (5.2-1,0) {c.};
\node [vertices, gray] (31) at (-0.33,-0.2) {};
\node [vertices, gray] (32) at (-0.67,-0.2) {};
\node [vertices, gray] (33) at (-0.5, 0.03) {};

\node[anchor = west] at (31) {$\, c$ };
\node[anchor = east] at (32) {$b \,$ };
\node[anchor = south east] at (33) {$a$};

\node (ss3) at (-0.5, -0.13) {};
\draw (ss3) ellipse (0.55 and 0.55);

%\fill (dd3) circle (0.05);
%\draw (33)--(-0.42, 0.20);

\end{tikzpicture}
}

%%%%%%%%%%%%%%%%%%%%%%%%%%%%%%%%%%%%
%% Kjegle for et punkt
      
\begin{figure}
\begin{center}
%Første bilde
\begin{tikzpicture}[inner sep=1pt, scale=0.8]

\coordinate (a) at (0,0) ; % [anchor =  north east] {$(1,2,3)$};
\coordinate (b) at (3, 1.71);  %[anchor = north west] {$(2,1,3)$};
\coordinate (c) at (3, 5.15);  %[anchor = north east] {$(3,1,2)$};
\coordinate (d) at (0, 6.86);  %[anchor = north east]{$(3,2,1)$};
\coordinate (e) at (-3, 5.15);  % [anchor = north west] {$(2,3,1)$};
\coordinate (f) at (-3, 1.71);  % [anchor = south west] {$(1,3,2)$};

\coordinate (x) at (-1.5,0.855) ; % [anchor =  north east] {$(1,2,3)$};
\coordinate (y) at (1.5, 0.855);  %[anchor = north west] {$(2,1,3)$};
\coordinate (z) at (3, 3.43);  %[anchor = north east] {$(3,1,2)$};
\coordinate (t) at (1.5, 6.005);  %[anchor = north east]{$(3,2,1)$};
\coordinate (u) at (-1.5, 6.005);  % [anchor = north west] {$(2,3,1)$};
\coordinate (v) at (-3, 3.43);  % [anchor = south west] {$(1,3,2)$};

\coordinate (s) at (2, 4.50);

\coordinate (fo) at (-3, 7.2);
\coordinate (fh) at (3, -1.71);
\coordinate (fu1) at (7.4, 2.565);
\coordinate (fu2) at (0.74, 9.8);

\draw[white, fill=tol1, fill opacity = 0.1] 
            (fh) -- (fu1)  -- (fu2) -- (fo)--(f);

    %Draw permutohedron: 
            \draw[black, dotted, fill=tol1, fill opacity = 0.09] 
            (a) -- (b)  -- (c) -- (d) -- (e) -- (f) -- (a);
            %Draw centerpoint and vertices:
            %\draw[black, fill] (0.5,)    circle (1pt);
            \draw[black, fill] (a)     circle (2pt);
            \draw[black, fill] (b)     circle (2pt);
            \draw[black, fill] (c)     circle (2pt);
            \draw[black, fill] (d)     circle (2pt);
            \draw[black, fill] (e)     circle (2pt);
            \draw[black, fill] (f)     circle (2pt);
            % \node[] at (-0.5, 0.8660)   [anchor = east]

\draw[black, very thick]
(f) -- (fo);

\draw[black, very thick]
(f) -- (fh);

%\node[] at (a) [anchor = north east] {$\cba$ \hskip 5mm {}};
%\node[] at (b) [anchor = north west] {$\cab$ \hskip 5mm {}};
%\node[] at (c) [anchor = south west] {$\acb$ \hskip 5mm {}};
%\node[] at (d) [anchor = south west] {$\abc$ \hskip 5mm {}};
%\node[] at (e) [anchor = south east] {$\bac$ \hskip 5mm {}};
\node[] at (f) [anchor = east] {$\bca\, \,\,  $ \hskip 10mm {}};
      
\node[] at (a) [anchor = north east] {$\bcsa$};
%\node[] at (y) [anchor = north west] {$\csab$};
%\node[] at (z) [anchor = west] {$\acsb$};
%\node[] at (t) [anchor = south west] {$\asbc$};
%\node[] at (u) [anchor = south east] {$\absc$};
\node[] at (e) [anchor = east] {$\bsac$};

\node[] at (s) {$\abcss$};
\end{tikzpicture}
          \end{center}
          \caption{The cone $\Pi(z^P)$ for $P$,
            and the conforming preorders of the four faces of this cone}
        \label{fig:cone-vertex-toperm}
      \end{figure}

 %%%%%%%%%%%%%%%%%%%%%%%%%%%%%%%%%%%
 %% Kjegle for en kant
      
\begin{figure}
\begin{center}
%Første bilde
\begin{tikzpicture}[inner sep=1pt, scale=0.8]

\coordinate (a) at (0,0) ; % [anchor =  north east] {$(1,2,3)$};
\coordinate (b) at (3, 1.71);  %[anchor = north west] {$(2,1,3)$};
\coordinate (c) at (3, 5.15);  %[anchor = north east] {$(3,1,2)$};
\coordinate (d) at (0, 6.86);  %[anchor = north east]{$(3,2,1)$};
\coordinate (e) at (-3, 5.15);  % [anchor = north west] {$(2,3,1)$};
\coordinate (f) at (-3, 1.71);  % [anchor = south west] {$(1,3,2)$};

\coordinate (x) at (-1.5,0.855) ; % [anchor =  north east] {$(1,2,3)$};
\coordinate (y) at (1.5, 0.855);  %[anchor = north west] {$(2,1,3)$};
\coordinate (z) at (3, 3.43);  %[anchor = north east] {$(3,1,2)$};
\coordinate (t) at (1.5, 6.005);  %[anchor = north east]{$(3,2,1)$};
\coordinate (u) at (-1.5, 6.005);  % [anchor = north west] {$(2,3,1)$};
\coordinate (v) at (-3, 3.43);  % [anchor = south west] {$(1,3,2)$};

\coordinate (s) at (2, 4.50);

\coordinate (fo) at (-6, 3.42);
\coordinate (fh) at (3, -1.71);
\coordinate (fu1) at (7.4, 2.565);
\coordinate (fu2) at (0.74, 9.8);

\draw[white, fill=tol1, fill opacity = 0.1] 
            (fh) -- (fu1)  -- (fu2) -- (fo)--(f);

    %Draw permutohedron: 
            \draw[black, dotted, fill=tol1, fill opacity = 0.09] 
            (a) -- (b)  -- (c) -- (d) -- (e) -- (f) -- (a);
            %Draw centerpoint and vertices:
            %\draw[black, fill] (0.5,)    circle (1pt);
            \draw[black, fill] (a)     circle (2pt);
            \draw[black, fill] (b)     circle (2pt);
            \draw[black, fill] (c)     circle (2pt);
            \draw[black, fill] (d)     circle (2pt);
            \draw[black, fill] (e)     circle (2pt);
            \draw[black, fill] (f)     circle (2pt);
            % \node[] at (-0.5, 0.8660)   [anchor = east]

\draw[black, very thick]
(f) -- (fo);

\draw[black, very thick]
(f) -- (fh);

%\node[] at (a) [anchor = north east] {$\cba$ \hskip 5mm {}};
%\node[] at (b) [anchor = north west] {$\cab$ \hskip 5mm {}};
%\node[] at (c) [anchor = south west] {$\acb$ \hskip 5mm {}};
%\node[] at (d) [anchor = south west] {$\abc$ \hskip 5mm {}};
%\node[] at (e) [anchor = south east] {$\bac$ \hskip 5mm {}};
%\node[] at (f) [anchor = north east] {$\bca$ \hskip 5mm {}};
      
\node[] at (x) [anchor = north east] {$\pbcsa$};
%\node[] at (y) [anchor = north west] {$\csab$};
%\node[] at (z) [anchor = west] {$\acsb$};
%\node[] at (t) [anchor = south west] {$\asbc$};
%\node[] at (u) [anchor = south east] {$\absc$};
%\node[] at (e) [anchor = east] {$\bsac$};

\node[] at (s) {$\abcss$};
\end{tikzpicture}
          \end{center}
          \caption{The cone $\Pi(z^P)$ for $P$,
            and the conforming preorders of the two faces of this cone}
        \label{fig:cone-edge-toperm}
      \end{figure}
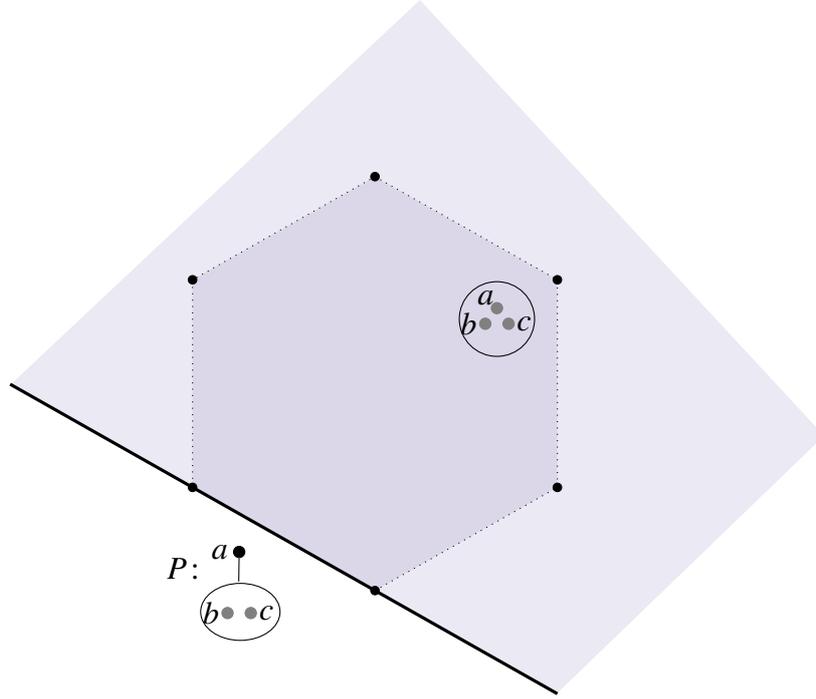

  \begin{lemma} \label{lem:mod-rpq}
    If $R \lhd P$ corresponds to $P \btl Q$ by Corollary \ref{cor:pos-bij},
    then 
    \[ (z^P)_Q = (z_Q)^R. \]
  \end{lemma}

  \begin{proof}
    \[ (z^P)_Q = \prod_{C \in \mathbf b(Q)} (z^P)_C, \]
    and each $(z^P)_C$ is equal to $(z_C)^{P_{|C}}$.
    On the other hand, as $R$ is the disjoint union
    $R = \sqcup_{C \in \mathbf b(Q)} R_{|C}$ we have:
    \[ (z_Q)^R = \prod_{C \in \mathbf b(Q)} (z_C)^{R_{|C}}. \]
      By Lemma \ref{cor:poset-pc}  $P_{|C} = R_{|C}$
      for each bubble $C \in \mathbf b(P)$.
    \end{proof}

    %%%%%%%%%%%%%%%%%%%%%%%%%%%%%%%%%%%%%%%%%%%%%
%% Extensions 

    \newcommand{\bst}{B_S^T}
    
\section{Extensions of conforming preorders}
\label{sec:ext}
We show a surprising extension result for preorders conforming to
submodular functions, essential for the arguments in Section \ref{sec:bim}:
If $S \sus I$ and $z(S) < \infty$, and $P_1$ conforms to $z_{|S}$ and
$P_2$ conforms to $z_{/S}$, they glue together in a {\it unique} way to
a preorder $P$ conforming to $z$.

\medskip
Note that if $P$ is compatible with  $z$ then
$P\restr {S}$ is compatible with $z\restr{S}$ and $P_{/S}$
compatible with  $z_{/S}$.
The same holds with ``compatible'' replaced by ``conforms''.

If $z(S) < \infty$,
the total preorder $B_S^T$ with two bubbles $S < T$
is compatible with $z$.
The associated submodular function is
\[ z_{\bst} = z\restr{S} \cdot z_{/S}. \]
The corresponding face of $\egp(z)$ is given by:
\[ \Phi_z(B_S^T) = \egp(z_{B_S^T}) =  \Pi(z\srestr{S}) \times \Pi(z_{/S}). \]
 
The following appears at first quite surprising.
It says that there is a unique way
to glue
together conforming preorders $P_1$ for $z\restr{S}$ and $P_2$ for
$z_{/S}$ to a preorder $P$ which conforms to $z$. 

\begin{proposition} \label{pro:compres-ps}
  Let $z : \pow(I) \pil \hRR$ be a submodular function,  $S \sus I$ with
  $z(S) < \infty$,  and $T = I \backslash S$.
  There is a bijection between:
\begin{itemize}
\item[1.] Preorders $P$ conforming to $z$ having $S$ as a down-set of $P$,
\item[2.] Pairs of:
\begin{itemize}
\item Preorders $P_1$ conforming to $z\restr{S}$,
\item Preorders $P_2$ conforming to $z_{/S}$.
\end{itemize}
\end{itemize}
Given $P$,  the associated pair is $P_1 = P\restr{S}$ and $P_2 = P_{/S}$.
If $P$ corresponds to the face $F$ of $\Pi(z)$, $P_1$ to the face $F_S$
of $\Pi(z\restr{S})$, and $P_2$ to the face $G_T$ of $\Pi(z_{/S})$, then
$F = F_S \times G_T$.
\end{proposition}

\begin{proof}
By Lemma \ref{lem:conf-alinbu}:
\[ \alinb(P) = \alinb(P\restr{S}) \oplus \alinb(P_{/S}) \sus
   \RR S \oplus \RR T. \]
 Since $S$ is a down-set in $P$ we have $P \preceq \bst$.
 Thus
 \begin{equation} \label{eq:compres-inc}
   \Phi_z(P) \sus \Phi_z(\bst) = \egp(z_{\bst}) = \egp(z\restr{S})
   \times \egp(z_{/S}).
 \end{equation}
    We have:
    \[ \Phi_z(P) = \egp(z) \cap \alin(P). \] Together with
    \eqref{eq:compres-inc} and Lemma \ref{lem:conf-alinbu} this gives:
    \begin{align} \label{eq:ext-ps}
      \Phi_z(P) & = \big(\egp(z\restr{S}) \times \egp(z_{/S})\big) \cap
                           (\alinb(P\restr{S}) \times \alinb(P_{/S}))  \notag \\
                         & = \Phi_z(P\restr{S}) \times \Phi_z(P_{/S}).
\end{align}
Thus if $P$ corresponds to $F$, then $F = F_S \times G_T$, where
$F_S$ corresponds to $P\restr{S}$ and $G_T$ to $P_{/S}$.

Conversely given $P_1$ and $P_2$. We get a face $F_S$ of
$\egp(z\restr{S})$ and $G_T$ of $\egp(z_{/S})$. Then $F_S \times G_T$
is a face of $\egp(z\restr{S}) \times \egp(z_{/S})$ and this in turn
is a face of $\egp(z)$. Thus $F_S \times G_T$ is a face $F$ of
$\egp(z)$, and this corresponds to a preorder $P$ which conforms to $z$.
Since
\[ P = \Psi_z(F) = \Psi_z(F_S \times G_T) \]
we get that $S$ is a down-set of $P$.
Also
\[\Phi_z(P) = \Phi_z \circ \Psi_z(F) = F = F_S \times G_T. \]
Then by \eqref{eq:ext-ps} $P_1 = P\restr{S}$ and $P_2 = P_{/S}$.
\end{proof}

\begin{proposition} \label{pro:compres-rpq}
  Let $Q$ conform to $z$. Then the preorders conforming to $z_Q$ is:
  \[ \Pre(z_Q) = \{ R = P \wedge Q^{\op} \, | \, P \in \Pre(z)
    \text{ and } P \btl Q\}. \]
  (By Corollary \ref{cor:pos-bij} these are the $R \lhd P$ corresponding
  to $P \btl Q$ for the fixed $Q \in \Pre(z)$.)
\end{proposition}

\begin{proof}
 % Note that such $R$ are those with  $R \lhd P$ corresponding 
 % to $P \btl Q$ by Corollary \ref{cor:pos-bij}.
  Assume $R \lhd P$ corresponds to $P \btl Q$ and $P \in \Pre(z)$.
  Then $R = \sqcup_{C \in \mathbf b(Q)} R_{|C}$, and each $R_{|C}$
  coincides with $P_{|C}$, Corollary \ref{cor:poset-pc}. The preorder
  $P_{|C}$ conforms to $z_C$,
  % by {Proposition \ref{pro:conf-minmax}},
  and so each $R_{|C}$ conforms
  to $z_C$. But then $R$ conforms to $z_Q$.

  \medskip \noindent Conversely assume $R \in \Pre(z_Q)$. Make a filtration of
  down-sets of $Q$:
  \[ Q = Q_r \supseteq Q_{r-1} \supseteq \cdots \supseteq Q_1 \supseteq Q_0 =
    \emptyset, \]
  such that the bubbles of $Q$ are the $C_j = Q_j \backslash Q_{j+1}$.
  Let $I_t$ be the underlying set of $Q_t$. 
  We have $z_Q = \prod_{C \in \mathbf b(Q)} z_C$ and each $z_C$ is indecomposable.
  When $R$ conforms to $z_Q$ we then have the disconnected union
  $R = \sqcup_{C \in \mathbf b(Q)} R_{|C}$ where $R_{|C}$ conforms to $z_C$,
  and $R_{|C}$ is thus connected. Let $R_t = \sqcup_{j \leq t} R_{|C_j}$.

  We want to construct $P$. We successively make a unique $P_t$ on the set $I_t$
  such that:
  \begin{itemize}
  \item $P_t$ conforms to $z\restr{I_t}$,
  \item $R_t \lhd P_t$,
  \item $P_t \btl Q_t$.
  \end{itemize}
  At the start we have $P_1 = R_1$.
  Suppose $P_t$ has been constructed on $I_t$. The preorder $R_{|C_{t+1}}$
  conforms to
  $z_{C_{t+1}}$ where $C_{t+1} = I_{t+1}\backslash I_t$. Proposition
  \ref{pro:compres-ps}
  then gives unique $P_{t+1}$ conforming to $z\restr{I_{t+1}}$.
  Furthermore it is clear the $R_{t+1} \lhd P_{t+1}$. We show that $P_{t+1}
  \preceq Q_{t+1}$: since both conform to $z\restr{I_{t+1}}$, Proposition
   \ref{pro:conf-PbtrQ} then gives $P_{t+1} \btl Q_{t+1}$.
    Let then $p < q$ be a cover relation in $P_{t+1}$ where we may assume
    $q \in C_{t+1}$ and $p \in I_t$. Let $p$ be in the bubble $C_j$.
    We need to show that $C_{t+1} > C_j$. If not let $D$ be the downset
    of $Q_{t+1}$ generated by $C_{t+1}$ and $C_j$, and let
    $E = D \backslash \{C_{t+1}, C_j\}$. Then $D \backslash E$ is disconnected
    in $Q$ and so $z_{D \backslash E}$ is decomposable. But since $P_{t+1}$
    conforms to $z\restr{I_{t+1}}$, then $(P_{t+1})_{|D \backslash E}$ would be
    disconnected. But both $P_{C_t}$ and $P_{C_j}$ are connected, and
    the relation $p < q$ connects these components, so this is wrong.
    Thus $P_{t+1}  \preceq Q_{t+1}$ and so $P_{t+1}  \btl Q_{t+1}$.
  \end{proof}

%%%%%%%%%%%%%%%%%%%%%%%%%%%%%%%%%%%%%
    %% Bimonoids in cointeraction

\section{Bimonoids in cointeraction}
\label{sec:bim}
Let $\setx$ be the category of finite sets with bijections.
By \cite{AA2017} submodular functions give a species
$\SM : \setx \pil \set$.
Here we consider the linearized
species $\SM : \setx \pil \vect$, where the latter is the category
of vector spaces over the real numbers.
We work with {the symmetric monoidal category of vector species \cite{J1981, BLL1998} endowed with the Cauchy product, and with monoids, comonoids,
bimonoids, and Hopf monoids in that category}. For these notions, see \cite{AM2010}.  $\SM$ comes with a product and a coproduct $\Delta$
making it a Hopf monoid in species. It is equivalent to a Hopf
monoid $\EGP : \setx \pil {\vect}$ of extended generalized permutahedra,
and a main result of \cite{AA2017} is to give a cancellation free
formula for the antipode. \\

The restriction of $\SM$ to modular functions
{$\MOD : \setx \pil \vect$} becomes a bimonoid in species
which can be endowed with a distinct
coproduct {$\delta:\MOD\to\MOD\odot\MOD$ where $\odot$ stands for the Hadamard product \cite{AM2010}, making it a double twisted bialgebra in the sense of L. Foissy \cite{Fo2019}}. We show that $(\SM, \cdot, \Delta)$ is comodule bimonoid
over $(\MOD, \cdot, \delta)$, so $\SM$ and $\MOD$ are in {\it cointeraction}.
Moreover there is a natural morphism $\phi : \SM \pil \MOD$ of comodule
bialgebras. 
%Let $\setx$ be the category of finite sets with bijections as morphisms.
%Denote by $\vect$ the category of vector spaces over the real numbers
%(but it could equally well have been over any field).

\subsection{Hopf monoids in species of submodular and modular functions}
Define a species $\SM : \setx \to \vect$ where $\SM[I]$ is the vector
space with basis the
submodular functions $\pow(I) \to \hRR$. In particular note
that when $I = \emptyset$ then $\pow(I)$
is the one element set $\{ \emptyset \}$, and
there is a single submodular function in $\SM[I]$, denoted ${\mathbf 0}$,
which sends
the empty set to $0$.\\

Let $U : \setx \pil \vect$ be the species sending $\emptyset$ to
$\RR$ (the one-dimensional vector space), and any $I \neq \emptyset$ to
the vector space $(0)$. 
The species $\SM$ becomes a monoid of species with product:
\begin{equation} \label{eq:bim-mult}
  \mu_{S,T} : \SM[S] \te \SM[T] \longrightarrow \SM[I],
  \quad u \te v \mapsto u\cdot v.
  \end{equation}
It becomes a comonoid in species by:
\[\Delta_{S,T} : \SM[I] \longrightarrow \SM[S] \te \SM[T],
\quad z \mapsto \begin{cases} z\restr{S} \te z_{/S}, & z(S) < \infty \\
                               0, & \text{ otherwise }
			       \end{cases} \]
The unit and the counit are given by:
\begin{equation} \label{eq:bim-unit}
\eta : U[S] \pil \SM[S], \, \,  1 \mapsto {\mathbf 0} \text{ when } S =
\emptyset, \qquad
\epsilon : \SM[S] \pil U[S],  \, \, z \mapsto \begin{cases} 1, &  S =
\emptyset,  \\ 0, & S \neq \emptyset
\end{cases}
\end{equation}

This makes $(\SM, \mu, \Delta, \eta, \epsilon)$
a bimonoid in species, actually a Hopf monoid {isomorphic to the Hopf monoid of extended generalized permutahedra \cite[Theorem 12.4]{AA2017}.} We also have the species $\MOD: \setx \to \vect$ where $\MOD[I]$ is
vector space with basis the set of
modular functions $\pow(I) \to \hRR$. This is a subspecies of $\SM[I]$.
By Lemma \ref{lem:mod-res} the product and coproduct above
restrict to $\MOD[I]$, making also $\MOD$ a bimonoid in species.
Actually both $\SM$ and $\MOD$ are Hopf monoids (for this notion
see \cite{AM2010}).

\subsection{Morphism of Hopf monoids}
For a submodular function $z$ denote by $\minPre(z)$ the minimal elements
in $\Pre(z)$, the minimal preorders conforming to $z$.

There is a morphism of species
\[ \phi: \SM \pil \MOD , \quad z \mapsto \sum_{P \in \minPre(z)} z^P. \]
Note that elements in $\minPre(z)$ correspond to the minimal
faces of $\Pi(z)$. In particular, if $z$ is a finite submodular
function, we are simply summing over the vertices of $\Pi(z)$.

\begin{proposition}
The morphism $\phi$ is a morphism of bimonoids in species.
\end{proposition}

\begin{proof}
  If $P_1$ is minimal conforming to $u \in \SM[S]$ and $P_2$ to $v \in \SM[T]$
  the disjoint union $P_1 \sqcup P_2$
  conforms to $z = u \cdot v$. Conversely
  if $z = u \cdot v$ and $P$ conforms to $z$, then $P$ is a disjoint union
  of preorders on $S$ and $T$. This shows that $\phi$ is a morphism
  of monoids. That $\phi$ is a morphism of comonoids, means that the following diagram commutes
  for $S$ with $z(S) < \infty$:
  \[ \xymatrix{ z \ar@{|->}[rr]^{\Delta_{S,T}} \ar@{|->}[d]_{\phi} &&
      z\restr{S} \te z_{/S} \ar@{|->}[d]^{\phi \te \phi}\\
      \sum_{P \in \minPre(z)} z^P \ar@{|->}[rr]^{\Delta_{S,T}} && }
       \]
    The lower map and the right vertical map send these to respectively:
    \[  \sum_{\scriptsize{\begin{matrix} {P \in \minPre(z)} \\ {S \text{ down-set of }P}\end{matrix}}}
        (z^P)\restr{S} \te (z^P)_{/S},
        \quad \sum_{\scriptsize{\begin{matrix} {P_1 \in \minPre(z\restr{S})} \\
          {P_2 \in \minPre(z_{/S})} \end{matrix}}}
      (z\restr{S})^{P_1} \te (z_{/S})^{P_2}. \]
    By Lemma \ref{lem:ext-ps} and Proposition \ref{pro:compres-ps}
    these sums are equal.
  \end{proof}

  \subsection{$\MOD$ as internal bimonoid and $\SM$ 
    as comodule Hopf monoid}
Recall that $z^P$ is modular by Proposition \ref{pro:mod-zp}.
   There is an internal comodule map:
  \[ \delta_{\SM} : \SM[I] \pil \SM[I] \te \MOD[I],
    \quad z \mapsto \sum_{Q \in \Pre(z)} z_Q \te z^Q. \]
  Since faces $F$ of the EGP $\egp(z)$ correspond to conforming
  preorders, when $F$ corresponds to $Q$, then $z_Q$ is the submodular
  function of the face $F$, and so encodes the faces $\subseteq F$.
  On the other hand $z^Q$ by Propositions \ref{pro:mod-zp}, \ref{pro:mod-cone}
  is a cone
  which encodes the faces of $\egp(z)$ which are $\supseteq F$. 

  Again by Proposition \ref{pro:mod-zp},
  the comodule map above restricts to an internal coproduct:
  \[ \delta_{\MOD} : \MOD[I] \pil \MOD[I] \te \MOD[I],
    \quad z \mapsto \sum_{Q \in \Pre(z)} z_Q \te z^Q. \]
  Let $\oU : \setx \pil \vect$ be the species with $\oU[I] = \RR$
  for every $I$ (denoted ${\bf Com}$ in \cite{Fo2019}).
  There is a counit map 
  \[ \varepsilon_{\MOD} : \MOD[I] \pil \oU[I], \quad
    z \mapsto \begin{cases} 1,  & \pre(z) \text{ is totally disconnected,} \\
      0, & \text{ otherwise. }
    \end{cases} \]
  We can see that
  \[ (\id_{\SM} \te \varepsilon_{\MOD}) \circ \delta_{\SM} = \id_{\SM}.\]
  by considering the maximal preorder $Q = \cpre(z)$ in $\Pre(z)$,
  since then $z_Q  = z$,
  and $z^Q$ is modular with associated preorder the totally disconnected
  ${\mathbf c}(Q)$. 
  Furthermore
  \[ (\id_{\MOD} \te \varepsilon_{\MOD}) \circ \delta_{\MOD} =
    (\varepsilon_{\MOD} \te \id_{\MOD})\circ \delta_{\MOD} = \id_{\MOD}, \]
  which for the second equality is ensured by the minimal preorder
  $P = \pre(z)$ in $\Pre(z)$, since then $z^P = z$ and $z_P$ is modular with
  associated preorder the totally disconnected $\mathbf b(P)$. 

  \begin{proposition}
    The map $\delta_{\MOD}$ is coassociative and
    the map $\delta_{\SM}$ makes $\SM$ a comodule over the comonoid
    $\MOD$.
  \end{proposition}

  \begin{proof}
    We have
    \[ (\id \te \delta_{\MOD}) \circ \delta_{\SM} :
      z \mapsto \sum_{P \in \Pre(z)} z_P \te z^P \mapsto
      \sum_{\tiny{\begin{matrix} P \in \Pre(z) \\ Q \in \Pre(z^P) \end{matrix}}}
      z_P \te (z^P)_Q \te (z^P)^Q. \]
Thus we are summing over all pairs $P \btl Q$ in $\Pre(z)$.
Also
  \[ (\delta_{\SM} \te \id) \circ \delta_{\SM} :
      z \mapsto \sum_{Q \in \Pre(z)} z_Q \te z^Q \mapsto
      \sum_{\tiny{\begin{matrix} Q \in \Pre(z) \\ R \in \Pre(z^Q) \end{matrix}}}
      (z_Q)_R \te (z_Q)^R \te z^Q. \]
    By Lemma \ref{lem:mod-rpq}, $(z^P)_Q = (z_Q)^R$ when $R \lhd P$
    corresponds to $P \btl Q$. Also by Proposition \ref{pro:compres-rpq}
    $\Pre(z_Q)$ are the  $R \lhd P$ corresponding to $P \btl Q$,
    so again in the second expression we are summing over all
    pairs $P \btl Q$ in $\Pre(z)$. But since $R$ and $P$ have the same
    bubbles, we have $(z_Q)_R = z_P$. 
\end{proof}

The multiplication $\mu$ from \eqref{eq:bim-mult},  and unit $\eta$
from \eqref{eq:bim-unit}
makes
  $(\MOD, \mu, \delta_{\MOD}, \eta, \varepsilon_{\MOD})$ into
  an internal bimonoid in species, or a twisted bimonoid in
  the terminology of \cite{Fo2019}. 

\begin{proposition}
  The map $\phi : (\SM, \delta_{\SM}, \varepsilon_{\SM})
  \pil (\MOD, \delta_{\MOD}, \varepsilon_{\MOD})$ is a morphism of comodule species.
\end{proposition}

\begin{proof}
  We must show the following diagram commutes:
  \[ \xymatrix{ z \ar@{|->}[rr]^{\delta_{\SM}} \ar@{|->}[d]_{\phi} &&  \sum_{Q \in \Pre(z)} z_Q \te z^Q \ar@{|->}[d]^{\phi \te \ben}\\
      \sum_{P \in \minPre(z)} z^P \ar@{|->}[rr]^{\delta_{\MOD}} &&}
    \]
 The lower map and the right vertical map send these to respectively:
    \[\sum_{P \in \minPre(z)}  \sum_{P \btl Q} (z^P)_Q \te (z^P)^Q ,
      \quad \sum_{R \in \minPre(z_Q)} \sum_{Q \in \Pre(z)} (z_Q)^R \te z^Q . \]
    By Lemma \ref{lem:mod-rpq} and Proposition \ref{pro:compres-rpq}
    these are equal.
    %$\Pre(z_Q)$ are those $ \in \Pre(z)$ such that $P \btl Q$.
    %For $P \btl Q$ and $P$ conforming to $z$, by Lemma BB we have
    %$(z_Q)^P = (z^P)_Q$, and $ (z^P)^Q = z^Q$ is clear.
  \end{proof}

  \subsection{A version of cointeraction}
\label{subsec:co-co}
 For information on cointeracting bialgebras, i.e. comodule bialgebras
 over a bialgebra, one may consult \cite{Molnar1977, Man2018}.
 For a bialgebra $B$ and a $B$-comodule $A$ which is itself a bialgebra,
 we say that $B$ is a cointeracting bialgebra of $A$ if
 $A$ is a bialgebra in the category of $B$-comodules. This means that
 the product $A \te A \mto{\mu} A$ and the coproduct $A \mto{\Delta} A \te A$
 also are $B$-comodule maps, as well as the
 unit and counit of $A$.
 The notions of cointeracting bimonoid in species are then analogous.
 See \cite{Fo2019, Fo2022}, where 
%For the analogous notion of cointeracting bimonoids in the species framework,
they are known as double twisted bialgebras.

\begin{lemma} The multiplication
  $\mu : \SM \te \SM \pil \SM$ is a morphism of comodules
  over $\MOD$.
\end{lemma}

\begin{proof} We must check that the following diagram commutes:
  \[  \xymatrix{
    \SM[S] \te \MOD[S] \te \SM[T] \te \MOD[T] \ar[r]_{\ben \te \tau_{2,3} \te
      \ben} &  \SM[S] \te \SM[T] \te \MOD[S] \te \MOD[T]
    \ar[d]_{\ben \te \ben \te \mu} \\
    \SM[S] \te \SM[T] \ar[d]_{\mu} \ar[u]_{\delta_{\SM} \te \delta_{\SM}}
    &  \SM[S] \te \SM[T] \te \MOD[I] \ar[d]_{\mu \te \ben}\\
    \SM[I] \ar[r]_{\delta_{\SM}} & \SM[I] \te \MOD[I]}.
\] 
This means that the two composite maps:
\[ \xymatrix{ z \te w \ar@{|->}[rr]^-{\delta_{\SM \te \SM}} \ar@{|->}[d]_{\mu} & & 
    \underset{\tiny{\begin{matrix} P_1 \in \Pre(z) \\ P_2 \in \Pre(w)
        \end{matrix}}}{\sum}
    z_{P_1} \te  w_{P_2} \te z^{P_1} \cdot w^{P_2}   \ar@{|->}[d]^-{\mu \te \ben}\\
    z \cdot w \ar@{|->}[rr]^-{\delta_{\MOD}}  & & 
    \underset{\scriptsize{Q \in \Pre(z\cdot w)}}{\sum}
    (z\cdot w)_Q \te (z\cdot w)^Q}.
    \]
    are the same. This is so because a preorder $Q$ conforms to the product
    $z \cdot w$ iff $Q$ is a disconnected union $Q = P_1 \sqcup P_2$ where
    $P_1$ conforms to $z$ and $P_2$ conforms to $w$. 
  \end{proof}
  
  The coproduct $\Delta: \SM \pil \SM \te \SM$ is however not
  a morphism of comodule species. Thus $\SM$ and $\MOD$ are not in cointeraction
  in the normal sense. However there is a related version with holds.

\begin{proposition} \label{pro:bim-nstand}
There is a commutative diagram:
\[  \xymatrix{
    \SM[I] \ar[r]^{\delta_{\SM}} \ar[d]_{\Delta_{S,T}} &
    \SM[I] \te \MOD[I]  \ar[d]^{\ben \te \Delta_{S,T}} \\
    \SM[S] \te \SM[T] \ar[d]_{\delta_{\SM} \te \delta_{\SM}} &
    \SM[I] \te \MOD[S] \te \MOD[T]\\
    \SM[S] \te \MOD[S] \te \SM[T] \te \MOD[T] \ar[r]_{\ben \te \tau_{2,3} \te
      \ben} & \SM[S] \te \SM[T] \te \MOD[S] \te \MOD[T] \ar[u]_{\mu \te \ben \te
      \ben}}.
\]
Restricting to $\MOD \sus \SM$ this makes $\MOD$ a comodule species
over itself. Hence $(\MOD, \mu, \delta_\SM)$ is a double twisted
bialgebra, \cite{Fo2019, Fo2022}. 
\end{proposition}

\begin{remark}
  If $\SM$ had been a comonoid comodule over $\MOD$, we should in this
  diagram have had a map $\Delta_{S,T} \te \ben$ on the upper right side,
  rather than $\ben \te \Delta_{S,T}$. This does however not give
  a commutative diagram.

   However one might say that $\SM$ is a {\it non-counital} twisted bialgebra.
  By the embedding $\MOD \sus \SM$,
  one could replace $\MOD$ by $\SM$ in the diagram, and still have
  a commutative
  diagram. The right $\SM$ is then a comodule {\it non-counital}
  comonoid over the left $\SM$.
  It would only fail to have a counit for the internal coproduct
  $\delta_{\SM}$. Indeed, the reader can easily be convinced, in view of
  Proposition \ref{pro:mod-zp} and the definition of
$\delta_{\SM}$, that the equality
\[ (\varepsilon\te\id_{\SM}) \circ \delta_{\SM}(z) = \id_{\SM}(z)\]
fails for any $\varepsilon:\SM[I] \pil \oU[I]$ when $z$ is not modular.
\end{remark}

\begin{proof}[Proof of Proposition \ref{pro:bim-nstand}]
Following the upper right maps we get:
\[ z \overset{\delta_{SM}}{\longmapsto} \sum_{P \in \Pre(z)} z_P \te z^P
  \longmapsto{}  \sum_{\scriptsize{\begin{matrix} {P \in \Pre(z)} \\ {S \text{ down-set of } P}\end{matrix}}}
  z_P \te (z^P)\restr{S} \te (z^P)_{/S}. \]
Following the maps on the left side around, we get that $z$ maps to
\[ \sum_{S : z(S) < \infty}
  \sum_{P_1 \in \Pre(z_{|S})} \sum_{P_2 \in \Pre(z_{/S})}
  (z\restr{S})_{P_1} \cdot (z_{/S})_{P_2} \te (z\restr{S})^{P_1} \te
  (z_{/S})^{P_2}. \]
By Proposition \ref{pro:compres-ps}, $P_1$ and $P_2$ glue
together to a unique $P$, and so these two expressions coincide.
\end{proof}

%{\begin{remark}
%    Given that $\MOD$ is included in $\SM$, we may wonder whether $(\SM, \mu, \%Delta, \delta_{\SM})$ is a double bimonoid, also called double twisted bialgebra in L. Foissy's terminology \cite{Fo2022}.
%    This is not the case, due to lack of counit for the internal coproduct $\delta_{\SM}$. Indeed, the reader can easily be convinced, in view of Proposition \ref{pro:mod-zp} and the definition of
% $\delta_{\SM}$,
%that the equality
%\[ (\varepsilon\te\id_{\SM}) \circ \delta_{\SM}(z) = \id_{\SM}(z)\]
%fails for any $\varepsilon:\SM[I] \pil \oU[I]$ when $z$ is not modular.
%\end{remark}
%}

\subsection{Duals} \label{subsec:co-duals}
Since each $\SM[I]$ and $\MOD[I]$ are infinite dimensional vector
spaces dualizing these spaces to make $\delta_{\SM}$ and $\Delta_{\SM}$
into products does not work well in full generality.

\begin{example}
  We extend Example \ref{ex:bij-pent}. Consider submodular functions
  $z_u : \pow(\{a,b,c\}) \pil \RR$ given by
  \[ z_u(a) = z_u(b) = z_u(c) = 3, \quad z_u(ab) = z_u(bc) = 5, \quad
    z_u(ac) = 5+u, \quad z(abc) = 6. \]
  These are submodular for $u \in [-2, 1]$. All these have
  the total order $P: c < b < a$ as a conforming preorder. The
  associated term to this preorder in $\delta_{\SM}(z_u)$, is
  $z_{u,P} \te z_u^P$ where
  \[ z_{u,P} = z_a \cdot z_b \cdot z_c, \quad
    \text{where for } \ell  = a, b, c, \quad
    z_\ell : \pow(\{\ell\}) \pil \RR, \, \, z_\ell(\ell) = 3. \]
  Furthermore
  \[ z_u^P(c) = 3, \quad z_u^P(b,c) = 5, \quad z_u^P(a,b,c) = 6. \]
  We see $z_{u,P} \te z_u^P$ does not depend on $u$, and may write $z_P \te z^P$
  for this. 
  Thus if we ``dualize'' $\delta$, the image of this $z_P \te z^P$ would
  be infinitely many $z_u$'s.
\end{example}

One may remedy the above by fixing a bound $N \in \NN$, and restrict to 
submodular functions $z : \pow(I) \pil \hat{\NN}$ with $z(S) \leq N$
when $z(s)$ finite. Then one could dualize $\delta$ and $\Delta$ to
an internal and an external product
\begin{align*} \SM \te \MOD & \pil \SM &   \SM \te \SM & \pil \SM \\
  z \te y & \mapsto z \bullet y  &  z \te w & \mapsto z *_{\SM} w.
\end{align*}
The latter would also restrict to a product $*_{\MOD}$ on $\MOD \te \MOD$.
One can then consider $\SM$ as ``acting'' on $\MOD$ to give an element
in $\SM$: Given  $z \in \SM$, it acts on
$y \in \MOD$ to give $z \bullet y$ in $\SM$. 
This ``action'' is associative
\[ z \bullet (y_1 \bullet y_2) = (z \bullet y_1) \bullet y_2, \]
and it respects the product $*$:
\[ z \bullet (y_1 {*_{\MOD}} y_2) = (z \bullet y_1) *_{\SM} (z \bullet y_2). \]

The internal product $z \bullet y$ would be a sum of submodular
``extension'' functions $w$ corresponding to extended generalized
permutahedra $\egp(w)$ having $\egp(z)$
as a face and $\egp(y)$ a cone codifying the faces of $\egp(w)$ containing
$\egp(z)$.

\begin{lemma}
  Given $z \in \SM[I]$ and $y \in \MOD[I]$. The internal product
  $z \bullet y$ can only be nonzero if the connected components
  $\cpre(z) = \pre(z)^{\bullet}$ equals the bubbles $\pre(y)^{\circ}$.
  This corresponds to the minimal affine space containing $\egp(z)$
  equalling the lineality space of $\egp(y)$.
\end{lemma}

\begin{proof}
  The product $z \bullet y$ must be of the form $w_P \bullet w^P$
  for some submodular function $w \in \SM[I]$, and $P$ conforming to
  $w$. But then the indecomposable factors of $w_P$ correspond to
  the bubbles of $P$.
\end{proof}

The external product $z * w$ for $z \in \SM[S]$ and $w \in \SM[T]$
would be all submodular functions $v \in \SM[I]$ with
$\egp(z) \times \egp(w)$ a face of $\egp(v)$.

\subsection{Morphisms to the species of preorders}
\label{subsec:bimon-pre}

Consider the species $\PRE : \setx \pil \vect$
where $\PRE[I]$ is the vector space with basis the preorders on $I$.
This becomes a Hopf monoid with product the disjoint union of preorders,
and with coproduct $\Delta$ given by
\[ P \mto{\Delta} \sum_{S \text{ down-set of } P} P\restr{S} \te P\restr{S^c}. \]
It also has an internal coproduct given by
\[ P \mto{\delta} \sum_{Q : P \btl Q} R \te Q \]
where $R \lhd Q$ corresponds to $P \btl Q$.
By \cite{FFM2017} $(\PRE, \mu, \Delta, \delta)$
becomes a double bimonoid in the category of species, i.e. a double twisted bialgebra in the sense of \cite{Fo2019}.
There is a map of species;
\[ \psi : \MOD \pil \PRE, \quad z \overset{\psi}\mapsto \pre(z).  \]

\begin{proposition}
  $\psi$ is a  morphism of double bimonoid species.
\end{proposition}

  \begin{proof}
To see that it preserves the coproducts $\Delta$: When $z$ is modular and $P = \pre(z)$, it is immediate
    that $\pre(z\restr{S}) = P\restr{S}$ and $\pre(z_{/S}) = P_{/S}$. It also preserves the coproducts $\delta$. Indeed,
    $Q \in \Pre(z)$ means that $P \btl Q$, and so
    $\pre(z^Q) = Q$. The minimal element
    $\pre(z_Q)$ of $\Pre(z_Q)$ is by
    Proposition \ref{pro:compres-rpq} the preorder $R$ with
    $R \lhd P$ corresponding to $P \btl Q$.
  \end{proof}
  {
\begin{remark}
  It would be interesting to adapt to the double monoid $\MOD$ the doubling
  procedure of \cite{AM2021}, applied to $\PRE$ therein.
\end{remark}
}
  
  %%%%%%%%%%%%%%%%%%%%%%%%%%%%%%
  %% Polynomials

\newcommand{\Ehr}{{\rm{Ehr}}}
\newcommand{\Ehrs}{{\rm{Ehr}^*}}
\newcommand{\linvert}{
  \begin{tikzpicture}[scale=.5, vertices/.style={draw, fill=black, circle, inner sep=1.5pt}]
%\draw [help lines, white] (-1,0) grid (1,1);
\node [vertices] (a) at (0,0) {};
\node [vertices] (b) at (0,1.2) {};
% \node [vertices] (c) at (0.5,0) {};
\coordinate [label=left: $a$] (a) at (0,0);
\coordinate [label=left: $b$] (b) at (0,1.2);
%\node at (3.44,.44) {$a$};
%\node at (3.44,.44) {$a$};

\foreach \to/\from in {a/b}
\draw (\to)--(\from);
\end{tikzpicture}
}

\newcommand{\trev}{
\begin{tikzpicture}[scale=.5, vertices/.style={draw, fill=black, circle, inner sep=1.5pt}]
\draw [help lines, white] (-1,0) grid (1,1);
\node [vertices] (b) at (-0.7,1) {};
\node [vertices] (a) at (0,0) {};
\node [vertices] (c) at (0.7,1) {};
\coordinate [label=left: $a$] (a) at (0,0);
\coordinate [label=above: $b$] (b) at (-0.7,1);
\coordinate [label=above: $c$] (c) at (0.7,1);

\foreach \to/\from in {a/b, a/c}
\draw (\to)--(\from);
\end{tikzpicture}
}

\subsection{Matroids and cointeractions}
\label{subsec:bimon-mat}
Restriction and contraction are standard notions for matroids.
The Hopf monoid structure on $\SM$ given by $z \mapsto \sum_{S \sus I}
z_{|S} \te z_{/S}$ restricts to matroids. It is considered in
\cite{DFM2018, AA2017}.
Cointeracting bialgebras  exclusively for matroids has not been established.
But we may enlarge the setting.

\begin{definition}
  A submodular function is $z : \pow(I) \pil \hRR$ is  a {\it matorder}
function if 
whenever $S \sus I$ with $z(S) < \infty$, then:
\begin{itemize}
\item[1.] $z(S) \leq |S|$,
\item[2.] $S \sus T$ implies $z(S) \leq z(T)$.
\end{itemize}
\end{definition}

Such functions give the least
common framework of matroids and the submodular
functions corresponding to preorders (those taking values in $\{0, \infty\}$).

Let $\SMat$ be the vector space generated by matorder functions, and
$\MMat$ the vector space generated by modular matorder functions.
Then $\SMat$ is a Hopf monoid, $\MMat$ a bimonoid, and $\SMat$ is
``acting'' on $\MMat$ as discussed in
Subsections \ref{subsec:co-co} and \ref{subsec:co-duals} above.
In particular, these generalize
the double bimonoid of preorders of \cite{FFM2017} given in Subsection
\ref{subsec:bimon-pre} above.

\section{Polynomials}
\label{sec:pol}

By the recent theory of L. Foissy \cite{Fo2022}, to a double bialgebra
$(B, \cdot, \Delta, \delta)$ (with $(B,\cdot, \Delta)$ a Hopf algebra) there is a unique double bialgebra morphism
$B \pil \QQ[x]$, to the polynomials. The modular functions become
such a double bialgebra, and so there is a distinguished composition
$\SM \mto{\phi} \MOD \pil \QQ[x]$ associating to any submodular function
a distinguished polynomial. Such polynomials usually count some
aspects of it source object $z$ and we describe this. The polynomial
occurs in \cite{AA2017} as the {\it basic invariant}. For
the submodular function of 
a matroid, the polynomial is the Billera-Jia-Reiner polynomial,
\cite{BJR2009}, see \cite[Sec.18.2]{AA2017}.

\subsection{Ehrhart polynomials}
Let  $z$ be modular with $P = \pre(z)$. Its associated
normal fan is then
a single cone $k(P)$ defined by $y_i \geq y_j$ if $i \leq_P j$.

If we put the additional requirement that $y_i \in [0,1]$, we get a
lattice polytope $\Pi(P)$.  So this is the part of the cone
$k(P)$ in the box $[0,1]^{|P|}$.
Let $k \cdot \Pi(P)$ be the homothetic expansion with center
$0$ and factor $k$. The {\it Ehrhart polynomial} $\Ehr_{\Pi(P)}(k)$
counts the number of
points in $\big(k \cdot \Pi(P)\big) \cap \ZZ^I$. The number of points which are in
the {\it interior} of $(k+1) \cdot \Pi(P) \cap \ZZ^I$ is given by
the polynomial
\[ \Ehrs_{\Pi(P)}(k) = (-1)^{\dim P} \Ehr_{\Pi(P)}(-k-1). \]
This is the interesting polynomial for our purpose.

\begin{example}
  Let
  \[ P = \linvert. \] The cone $k(P)$ is the halfplane $y_a \geq y_b$.
The polytope $(k+1) \cdot \Pi(P)$ is the region

\[k+1 \geq y_a \geq y_b \geq 0. \]
The interior points here are all integer pairs $(r,s)$ with
$k \geq r > s \geq 1$. Thus $\Ehrs_P(k) = \binom{k}{2} = \frac{k^2 - k}{2}$.
\end{example}

\begin{example}\label{ex:pol-tot}
If $P$ is the totally ordered set $[n] = \{1 < 2 < \cdots < n \}$
with $n$ elements, the order polytope $\Pi([n])$ \cite{Sta1986},
is defined by
\[ 1 \geq y_1 \geq y_{2} \geq \cdots \geq y_n \geq 0. \]
The interior integral points of $(k+1) \cdot \Pi([n])$ are given by
all integer sequences
\[ k \geq a_1 > a_{2} > \cdots > a_n \geq 1. \]
The number of such is the binomial number $\binom{k}{n}$.
\end{example}

\begin{example}
  Consider the V-poset:
  \[ P = \trev. \]
  The cone $k(P)$ is defined by $y_a \geq y_b$ and $y_a \geq y_c$.
  For the order polytope $\Pi(P)$ we also add the requirements
  $1 \geq y_a$ and $y_b,y_c \geq 0$.
  For the Ehrhart polynomial we must count all triples $(r,s,t)$ with
  $k \geq r > s,t \geq 1$. This is the sum:
  \[ (k-1)^2 + (k-2)^2 + \cdots + 1^2  = \frac{k(k-1)(2k-1)}{6}
    = k^3/3 - k^2/2 + k/6. \]
\end{example}

\subsection{Double bialgebras and associated polynomials}
We recall the notion of {\it double bialgebra} in
\cite{Fo2022}. These are cointeracting bialgebras, where the
underlying algebras are the same. 

Let $k$ be a field and $B$ a $k$-algebra with product $\mu$ and unit $\eta: k \pil B$.
It is a {\it double bialgebra}
if it has two structures as bialgebra:
\[ B_\Delta = (B, \mu, \Delta, \eta, \epsilon_\Delta), \quad B_\delta =
(B, \mu, \delta, \eta, \epsilon_\delta), \]
such that $B_\Delta$ is a comodule bialgebra over $B_\delta$.
In more detail:

\begin{itemize}
\item The comodule structure of $B_\Delta$ over $B_\delta$:
\[ \delta: B_\Delta \pil B_{\delta} \te B_{\Delta} \]
is simply $\delta : B \pil B \te B$
\item The coproduct $\Delta : B \pil B \te B$ is a comodule
morphism:
\begin{equation} \label{eq:dobi-comod}
(\ben_B \te \Delta) \circ \delta = m_{13,2,4} \circ (\delta \circ \delta)
\circ \Delta.
\end{equation}
\item The counit $\epsilon_\Delta : B \pil k$ is a comodule morphism:
  \begin{equation} \label{eq:dobi-counit}
    (\ben_B \te \epsilon_\Delta) \circ \delta = \eta_B \circ \epsilon_\Delta.
    \end{equation}
\end{itemize}

The basic example of this is the double bialgebra
$(\QQ[x], \mu, \Delta, \delta)$, where
$\mu$ is ordinary multiplication of polynomials, and
\[ \Delta(x) = 1 \te x + x \te 1, \quad \delta(x) = x \te x. \]
By the theory of Foissy \cite{Fo2022} when $B$ is a double bialgebra
such that $(B,\mu, \Delta)$ is a Hopf algebra,
there is a unique
double bialgebra morphism
\[ \theta : (B, \mu, \Delta, \delta)  \pil (\QQ[x], \mu, \Delta, \delta). \]
Thus every element of $B$ gives a distinguished polynomial.
%%%
\subsection{The polynomial associated to a submodular function}
%%%
By the bosonic Fock functor \cite[Chapter 15]{AM2010}, i.e.
considering isomorphism classes of submodular and modular functions, i.e.
the orbits of $\SM[I]$ and $\MOD[I]$ for the bijections of $I$,
we get (by abuse of notation we still use $\SM$ and $\MOD$ to
denote the corresponding bialgebras):
\begin{itemize}
\item Hopf algebras $(\SM, \mu, \Delta)$ and $(\MOD, \mu, \Delta)$,
\item A bialgebra $(\MOD, \mu, \delta_{\MOD})$,
\item $(\SM, \mu, \Delta, \delta_{\SM})$ is a cointeracting comodule
  bialgebra over the bialgebra $(\MOD, \mu, \delta_{\MOD})$,
\item $(\MOD, \mu, \Delta, \delta_{\MOD})$ is a double bialgebra.  
\item The map $\overline{\phi} : \SM \pil \MOD$ is a morphism of
  comodule bialgebras over the bialgebra $\MOD$.
\end{itemize}

\noindent By the theory of Foissy \cite{Fo2022} there is then a unique
double bialgebra morphism
\[ \theta : (\MOD, \mu, \Delta, \delta)  \pil (\QQ[x], \mu, \Delta, \delta). \]
By composition we also get a morphism
\begin{equation} \label{eq:pol-thfi}
  \theta \circ \overline{\phi} : (\SM, \mu, \Delta, \delta) \pil
  (\QQ[x], \mu, \Delta, \delta).
  \end{equation}
What is the polynomial associated to a submodular function $z$?

The unique double bialgebra morphism
$\tau : (\PRE, \mu, \Delta, \delta)  \pil (\QQ[x], \mu, \Delta,
\delta)$ has been studied by L. Foissy in \cite{Fo-Ehr}.
It associates to a preorder $P$ the Ehrhart polynomial
$\Ehrs_{\Pi(P)}(k)$. 
Composing with the morphism in Subsection \ref{subsec:bimon-pre}
there is a composition $\tau \circ \psi$. Due to the uniqueness
of this morphism, it must coincide with $\theta$. Thus if the modular function
corresponds to the pair $(P, \RR^{b(P)})$ the associated polynomial
is $\Ehrs_{\Pi(P)}(k)$. 

\begin{theorem} \label{thm:pol-pol}
  Let $z$ be a submodular function. The polynomial associated by
  \eqref{eq:pol-thfi} to $z$ is: 
  \[ \chi(z)(k) := \theta \circ \overline{\phi} (z) =
    \sum_{P \in \minPre(z)} \Ehrs_{\Pi(P)}(k). \]
  Thus the polynomial counts the {\it interior} points in the
  maximal cones of the normal fan of the extended generalized
  permutahedron $\Pi(z)$, intersected with the box $[0,k+1]^{|I|}$. 
\end{theorem}

\begin{proof}
  By \cite{Fo-Ehr} the polynomial associated to $P$ by $\tau$ is the Ehrhart
  polynomial $\Ehrs_{\Pi(P)}(k)$. Whence the polynomial associated to a
  {\it modular}
  function $z$ is $\Ehrs_{\Pi(P)}(k)$ where $P = \pre(z)$. The result for
  a {\it submodular} function then
  follows by the definition of $\phi$.
\end{proof}

\begin{example}
  By Example \ref{ex:pol-tot}
  for any modular function $z$ with $\pre(z) = [n]$, the total
preorder, the associated polynomial if $\binom{k}{n}$.
\end{example}

\begin{example} Let $|I| = n$. 
  For any submodular function $z : \pow(I) \pil \RR$,
  with $\Pi(z)$ the permutahedron, the minimal elements of
  $\pre(z)$ are the total orders on $I$. There are $n!$ such total orders.
  Since $z \mapsto \sum_{P \in \minPre(z)} z^P$, and each $z^P$ maps to 
  the polynomial $\binom{k}{n}$, the polynomial associated to a
  permutahedron submodular function is
  \[ n! \cdot \binom{k}{n} = k(k-1)(k-2) \cdots (k-n+1). \]
\end{example}

\subsection{Identifying the submodular polynomial}

Consider species in vector spaces over a field $\kk$.
To a Hopf monoid $\bH$ in this species, and a character
$\zeta : \bH \pil \kk$ define for each natural number $n$,
finite set $I$, and $x \in \bH[I]$:
\begin{equation} \label{eq:pol-chi} \chi_I(x)(n) :=
  \underset{I= S_1 \sqcup \cdots \sqcup S_n}\bigsum
  (\zeta_{S_1} \te \cdots \te \zeta_{S_n}) \circ \Delta_{S_1, \ldots, S_n}(x).
\end{equation}
By \cite[Prop.16.1]{AA2017} this is a polynomial function in $n$ of degree
at most $|I|$. By the Fock functor \cite[Chap.15]{AM2010}
we get a bialgebra which
we denote by $(\bH, \mu, \Delta)$. The above construction \eqref{eq:pol-chi}
induces a morphism of bialgebras
$(\bH, \mu, \Delta) \pil (\QQ[x], \mu, \Delta)$,
  \cite[Thm.3.9.2]{Fo2022}.

  Letting $\SM^{fin}$ be the subspecies of $\SM$ consisting of
  {\it finite} submodular functions $z : \pow(I) \pil \RR$, 
  \cite[Def.17.1]{AA2017} consider the {\it basic character} $\beta_I$ of
  $\SM^{fin}$ (they write ${\mathbf {GP}}$ instead of $\SM^{fin}$) given by
  \[ \beta_I(z) = \begin{cases} 1, & \egp(z) \text{ a point} \\
        0, & \text{otherwise}. \end{cases}\]
 We extend this to a character on $\SM$ by
        \[ \beta_I(z) = \begin{cases} 1, & \text{if }
              \egp(z) \text{ is an affine linear space} \\
              0, & \text{otherwise}. \end{cases} \]
              Note that $\egp(z)$ is an affine linear space iff $\egp(z)$ has
              a single face, which is so iff $z$ is modular with
 $\pre(z)$ totally disconnected.
 For the basic character $\beta_I$ the polynomial associated to
 a {\it finite} submodular function is the 
 {\it basic invariant} $\chi_{AA}(z)(k)$ of \cite[Section 17]{AA2017}.
    
 \begin{proposition} When $z$ is a finite submodular function,
   the polynomial $\chi(z)(k)$ of Theorem
   \ref{thm:pol-pol} is the basic invariant $\chi_{AA}(z)(k)$.
 \end{proposition}

 \begin{proof}
   By \cite[Thm.3.9]{Fo2022}, the double bialgebra morphism
   \[ (\MOD, \mu, \Delta, \delta) \mto{\theta} (\QQ[x], \mu, \Delta, \delta)\]
   is via \eqref{eq:pol-chi} associated to
   the counit character $\varepsilon_{\MOD}$ of
   $(\MOD, \mu, \Delta, \delta)$, given by
   \[  \varepsilon_{MOD} : \MOD[I] \pil \oU[I], \quad
    z \mapsto \begin{cases} 1  & \text{ if }\pre(z) \text{ is totally disconnected,} \\
      0 & \text{ otherwise. }
    \end{cases} \]
  Now we have $\phi: \SM \pil \MOD$.
  The composed bialgebra morphism
  $\chi : \SM \mto{\phi} \MOD \mto{\theta} \QQ[x]$
  is by \cite[Prop.16.3]{AA2017} associated via \eqref{eq:pol-chi} to the
  composed character $\varepsilon_{\MOD} \circ \phi$. We claim that this
  character is $\beta_I$. This will prove the proposition.
  
  The map $\phi$ sends $z \mapsto \underset{P \in \minPre(z)} \sum  z^P$.
  We have $\pre(z^P) = P$. If this is totally disconnected, then
  it is also the maximal element in $\Pre(z)$. Thus $\Pre(z)$ has a single
  element, and so $z$ is modular with $\pre(z)$ totally disconnected.
  So the composed character $\varepsilon_{\MOD} \circ \phi$ equals $\beta_I$.
\end{proof}

By \cite[Prop.18.3]{AA2017}, for $z$ the submodular function of a matroid $M$
on $I$,
the basic invariant $\chi(M)(k)$ is the Billera-Jia-Reiner polynomial
associated to matroid, \cite{BJR2009}. It is given by
\[ \chi(M)(n) = \text{ number of $m$-generic functions } y : I \pil
  [n], \]
where $y$ is $m$-generic if there is a unique basis $\{b_1, \ldots, b_r \}$
maximizing $y(b_1) + \cdots + y(b_r)$.

%%%%%%%%%%%%%%%%%%%%%%%%%%%%%%%%%%%%%%%%%%%%%%%%%%%%%%%%%%%%%%%%%%     
%% 6. References
%%%%%%%%%%%%%%%%%%%%%%%%%%%%%%%%%%%%%%%%%%%%%%%%%%%%%%%%%%%%%%%%%%
% \input separation
\bibliographystyle{amsplain}
\bibliography{biblio}

\end{document}

%%%%%

\end{document}